\documentclass{amsart}
\usepackage[a4paper]{geometry}
\usepackage{ucs}
\usepackage[utf8x]{inputenc}
\usepackage{amsthm}
\usepackage{url}
\makeatletter

\def\part{\@startsection{part}{0}%
  \z@{\linespacing\@plus\linespacing}{.5\linespacing}%
  {\normalfont\bfseries\centering}}
\makeatother
\newtheorem{thm}{Theorem}[section]
\newtheorem{lem}[thm]{Lemma}
\newtheorem{prop}[thm]{Proposition}
\newtheorem{cor}[thm]{Corollary}
\theoremstyle{remark}
\newtheorem{remark}{Remark}
\newtheorem{example}{{\bf {\em Example}}}
\usepackage{hyperref}
\usepackage{xcolor}
\hypersetup{
  colorlinks   = true, 
  urlcolor     = blue, 
  linkcolor    = blue, 
  citecolor   = red 
}

\usepackage{setspace}
\usepackage{amsmath}
\usepackage{amsfonts}
\usepackage{amssymb}

\newtheorem{definition}[thm]{{\bf {\em Definition}}}

\title[Multilinear Duality and Factorisation]{Multilinear Duality and Factorisation
  for Brascamp--Lieb-type Inequalities}
\author{Anthony Carbery, Timo S. H\"anninen and Stef\'an Ingi Valdimarsson}
\address{Anthony Carbery, 
School of Mathematics and Maxwell Institute for Mathematical Sciences, 
University of Edinburgh,
James Clerk Maxwell Building, 
Peter Guthrie Tait Road,
King's Buildings, 
Mayfield Road, 
Edinburgh, EH9 3FD, 
Scotland.}
\email{A.Carbery@ed.ac.uk}
\address{Timo S. H\"anninen,
Department of Mathematics and Statistics, University of Helsinki, P.O. Box 68, FI-00014 Helsinki, Finland, and School of Mathematics and Maxwell Institute for Mathematical Sciences, 
University of Edinburgh,
James Clerk Maxwell Building, 
Peter Guthrie Tait Road,
King's Buildings, 
Mayfield Road, 
Edinburgh, EH9 3FD, 
Scotland.}
\email{timo.s.hanninen@helsinki.fi}
\address{Stef\'an Ingi Valdimarsson,
Arion banki, Borgart\'un 19, 105  Reykjav\'ik, Iceland, and Science
Institute, University of Iceland, Dunhagi 5, 107 Reykjav\'ik, Iceland}
\email{sivaldimarsson@gmail.com}
\date{4th September 2018, revised 5th February and 30th September 2020.}
\begin{document}
\setcounter{tocdepth}{1}
\begin{abstract}
  We initiate the study of a duality theory which applies to norm inequalities for pointwise weighted
  geometric means of positive operators.
  The theory finds its expression in terms of certain pointwise factorisation properties of function spaces
  which are naturally associated
  to the norm inequality under consideration. We relate our theory to the
  Maurey--Nikisin--Stein theory of factorisation of operators,
  and present a fully multilinear version of Maurey's fundamental theorem on factorisation of operators
  through $L^1$. The development of the theory involves convex
  optimisation and minimax theory, functional-analytic considerations concerning the dual of $L^\infty$,
  and the Yosida--Hewitt theory of finitely additive measures. We consider the connections of the theory with the
  theory of interpolation of operators. We discuss the
  ramifications of the theory in the context of concrete families of geometric inequalities,
  including Loomis--Whitney inequalities, Brascamp--Lieb inequalities and multilinear Kakeya inequalities. 
\end{abstract}
\maketitle
\tableofcontents
\section{Introduction}\label{intro}
In this paper we introduce and develop a general functional-analytic principle which
gives a unifying framework for a range of multilinear phenomena that have recently arisen
in a number of areas of mathematical analysis.

\medskip
\noindent
We shall be mainly concerned with norm inequalities for pointwise
weighted geometric means
$$ \prod_{j=1}^d (T_j f_j(x))^{\alpha_j}$$
of positive linear operators $T_j$ defined on suitable spaces, where $\alpha_j \geq 0$
and $\sum_{j=1}^d \alpha_j = 1$. Before we describe the scope of our work in this paper, and to set the scene
for our study,
we briefly visit the analogous territory in the linear setting ($d=1$) in order to help provide a
context for what we are aiming to achieve. Throughout the whole paper we shall be dealing with real-valued
rather than complex-valued functions.

\subsection{The linear setting}\label{linset}

\medskip
\noindent
Let $X$ and $Y$ be measure spaces and let $T: L^p(Y) \to L^q(X)$ be a bounded linear operator, that is,
it satisfies
\begin{equation}\label{lin1}
\|Tf\|_q \leq A \|f\|_p
  \end{equation}
for all $f \in L^p(Y)$, for some $A > 0$. Here, $1 \leq p \leq \infty$ and $0 < q < \infty$. Since $L^q$
is a Banach space only when $q \geq 1$, it is natural to focus separately on the regimes $q \geq 1$ and $0 < q < 1$.  

\medskip
\noindent
(i) {\bf Case $q \geq 1$}. Since $\|h\|_q = \max\{ |\int hg| \, : \, \|g\|_{q'} = 1\}$, inequality
  \eqref{lin1} holds if and only if for all $ f \in L^p$ and {\em all} $g \in L^{q'}$ we have
  \begin{equation}\label{lin2}
    | \int (Tf) g| \leq A \|f\|_p\|g\|_{q'}.
    \end{equation}
  Using the relation $\int (Tf) g = \int f (T^* g)$, this is in turn
  equivalent to the statement $\|T^*g\|_{p'} \leq A \|g\|_{q'}$ -- that is the boundedness of the adjoint
  operator $T^*$ between the dual spaces of $L^q$ and $L^p$ respectively, (at least when
  $1 < p, \, q < \infty$). We are therefore firmly in the terrain
  of classical linear duality theory, a theory whose utility and importance cannot be overstated.
  Notice that if $q = 1$ and $T$ is also
  assumed to be positive (that is, $Tf \geq 0$ whenever $f \geq 0$), the equivalence of \eqref{lin1} and \eqref{lin2}
  is essentially without content since in this case it suffices to check on the function $g \equiv 1$.

  \medskip
  \noindent
(ii)  {\bf Case $0 <q <1$}. Since $\|h\|_q = \min\{ |\int hg| \, : \, \|g\|_{q'} = 1\}$, we have that
  \eqref{lin1} holds if and only if for all $f \in L^p$ {\em there exists} an (extended real-valued) $g \in L^{q'}$ such that
  \eqref{lin2} holds. Note that $q' < 0$ in this situation, so it is implicit that such a $g$ satisfies
  $g(x) \neq 0$ almost everywhere.
  It is a remarkable result of Maurey, that, under certain conditions -- such as positivity of $T$ -- given
  inequality
  \eqref{lin1}, there exists a {\em single} $g\in L^{q'}$ with $\|g\|_{q'} = 1$ such that
  \eqref{lin2} holds
  for all $f \in L^p$. Such a result is an instance of the celebrated theory of factorisation
  of operators which is developed in \cite{Mau}.
  Indeed, it is a case of {\em factorisation through $L^1$}
  since the inequality
  $$ | \int (Tf) g| \leq A \|f\|_p$$
  demonstrates that $T$ may be factorised as $T = M_{g^{-1}} \circ S$ where
  $S = M_g \circ T$ satisfies $\|S\|_{L^p \to L^1} \leq A$ and $M_{g^{-1}}$, the operator
  of multiplication by $g^{-1}$,
  satisfies $\|M_{g^{-1}}\|_{L^1 \to L^{q}} = \|g\|_{q'}^{-1} = 1$.
  
  \bigskip
  \noindent
  Observe that there is no obvious point of direct contact between the two regimes $q \geq 1$ and
  $0 < q < 1$ in this linear setting.

  \medskip
  \noindent
  The result of Maurey to which we refer falls within the wider scope of Maurey--Nikisin--Stein
  theory, which considers factorisation of operators in a broad variety of contexts. This includes
  consideration of non-positive operators, sublinear operators (for example maximal functions),
  operators with various domains and codomains, and factorisation through
  various weak- and strong-type spaces, often under some auxiliary hypotheses.
  The particular case of positive operators defined on normed
  lattices, taking values in $L^q$ for $q < 1$, and factorising
  through (strong-type) $L^1$ was considered by Maurey, however, and for this reason we refer specifically
  to the Maurey theory rather than the broader  Maurey--Nikisin--Stein theory. For an overview of this
  larger theory see \cite{GCRdeF}, \cite{Gilb}, \cite{Mau} and \cite{Pis}.
  
\subsection{The multilinear setting}\label{multlinset}
The purpose of this paper is to develop duality and factorisation theories for certain classes of
multilinear operators which are analogous to those that we have set out above in the linear setting. Amusingly,
the notion of ``factorisation'' manifests itself in two distinct ways in our development. One of
these is as a multilinear analogue of a formulation of a Maurey-type theorem as was briefly
outlined in the discussion of the case $0 < q < 1$ above. The other is that our
duality theory (corresponding to the case $1 \leq q < \infty$) will be expressed in terms
of pointwise factorisation properties of certain spaces of functions. {\em Even simple instances of these pointwise factorisation results are new and striking: see Section~\ref{appl1} below.}  

\medskip
\noindent
We begin by describing the scenario in which we shall work and the classes of operators we shall consider.

\medskip
\noindent
Let $(X,{\rm d}\mu)$ and $(Y_j,{\rm d}\nu_j)$, for $j=1,\dots,d$, be measure spaces,\footnote{Throughout the paper,
  when we refer to measure spaces $X$, $Y$ or $Y_j$ without explicit mention of the measure, it is implicit that
  the corresponding measures are $\mu$, $\nu$ and $\nu_j$ respectively, unless the context demands otherwise.}
let $\mathcal{S}(Y_j)$
denote the class of real-valued simple functions (i.e. finite linear combinations of 
characteristic functions of measurable sets of {\em finite} measure) on $Y_j$, and let $\mathcal{M}(X)$ 
denote the class of real-valued measurable functions on $X$. Let $T_1,\dots,T_d$ be 
linear maps 
$$ T_j : \mathcal{S}(Y_j) \to \mathcal{M}(X).$$
We suppose throughout that the $T_j$ are positive in the 
sense that if $f \geq 0$ almost everywhere on $Y_j$, then $T_j f \geq 0$ almost everywhere on $X$.

\medskip
In this paper we shall be concerned with ``multilinear'' Lebesgue-space inequalities of the form
\begin{equation}
\label{premainineq}
\left\|\prod_{j=1}^d (T_jF_j)^{\beta_j}\right\|_{L^q(X)} \leq C
\prod_{j=1}^d \Big\|F_j\Big\|_{L^{p_j}(Y_j)}^{\beta_j}
\end{equation}
where $0 < \beta_j < \infty$, $0 < p_j \leq  \infty$\footnote{We shall soon focus on the case $p_j \geq 1$
and $\sum_j \beta_j = 1$.}and
$0 < q \leq \infty$.

\medskip
\noindent
These inequalities are to be interpreted in an {\em a priori} sense, with the $F_j$ %initially
being nonnegative simple functions
defined on $Y_j$. We are especially interested in the case 
that either the $T_j$ are not bounded operators from $L^{p_j}(Y_j, {\rm d} \nu_j)$ to $L^q(X, {\rm d} \mu)$, or that they 
are bounded but do not enjoy effective bounds.

\medskip
\noindent
Strictly speaking such inequalities are multilinear only when each $\beta_j = 1$; we shall nevertheless abuse language
and will refer to the inequalities under consideration as ``multilinear''. In fact the case when $\sum_{j=1}^d \beta_j =1$
will play a special role in what follows. 
Of course we may always assume either that $q=1$ or that $\sum_{j=1}^d \beta_j = 1$. 

\medskip
\noindent
To fix ideas, we discuss some examples of inequalities falling under the scope of our study.

\medskip
\noindent
\subsection{Examples}\label{exples}

\begin{example}\label{Holder}{\rm [H\"older's inequality]
The multilinear form of H\"older's inequality for nonnegative functions is simply
$$ \int_X F_1(x)\cdots F_d(x) {\rm d} \mu(x) \leq \|F_1 \|_{L^{p_1}(X)} \cdots\|F_d \|_{L^{p_d}(X)}$$ 
where $p_j > 0$ and $\sum_{j=1}^d p_j^{-1} = 1$. This is of the form \eqref{premainineq}, with $T_j = I$ for all $j$,
$q=1$ and each $\beta_j = 1$. But, for any fixed set of positive exponents $\{\beta_j\}$, it is also trivially equivalent to
the inequality
$$ \| f_1^{\beta_1} \cdots f_d^{\beta_d} \|_q \leq  \|f_1 \|_{q_1}^{\beta_1} \cdots\|f_d \|_{q_d}^{\beta_d}$$
for all $ 0 < q_j < \infty$ and $0 <q < \infty$ which satisfy $\sum_{j=1}^d \beta_j q_j^{-1} = q^{-1}$. 
In particular, there is an equivalent formulation of the multilinear H\"older inequality taking the form
\eqref{premainineq} with $\sum_{j=1}^d\beta_j = 1$. In fact there are many such equivalent forms, limited
only by the requirement that $\sum_{j=1}^d \beta_j q_j^{-1} = q^{-1}$.
Special cases of choices of exponents $\{\beta_j, q_j, q\}$
satisfying this condition are (i) $\beta_j$ 
arbitrary subject to $\sum_{j=1}^d \beta_j = 1$, $q_j =1$ for all $j$, and $q =1$; and 
(ii)  $\beta_j = d^{-1}$ for all $j$, $q_j$ arbitrary subject to $\sum_{j=1}^d q_j^{-1} = 1$, and $q = d$. 
This observation demonstrates that we may expect that a given multilinear inequality might have {\em multiple}
equivalent manifestations, each of the form \eqref{premainineq}, with $\sum_{j=1}^d \beta_j = 1$. 
In the context of the factorisation theory we shall develop, each
manifestation of the inequality corresponds to a different factorisation property of
associated function spaces. See Section~\ref{rrr} for further discussion.}
\end{example}

\begin{example}\label{Loomis--Whitney}
{\rm [Loomis--Whitney inequality]
  For $1 \leq j \leq n$ let $\pi_j: \mathbb{R}^n \to \mathbb{R}^{n-1}$ be projection on the coordinate
  hyperplane perpendicular to the standard 
  unit basis vector $e_j$; that is, $\pi_j x = (x_1, \dots, \widehat{x_j}, \dots, x_n)$. The Loomis--Whitney
  inequality \cite{LW} for nonnegative functions is
$$ \int_{\mathbb{R}^n} F_1(\pi_1 x) \cdots F_n(\pi_n x) \, {\rm d} x
\leq \|F_1\|_{L^{n-1}(\mathbb{R}^{n-1})} \cdots \|F_n\|_{L^{n-1}(\mathbb{R}^{n-1})}.$$
For each $0 < p < \infty $, this is equivalent to the inequality
$$ \|f_1(\pi_1 x)^{1/n} \cdots f_n(\pi_n x)^{1/n} \|_{L^{np/(n-1)}(\mathbb{R}^n)} 
\leq \|f_1\|_{L^{p}(\mathbb{R}^{n-1})}^{1/n} \cdots \|f_n\|_{L^{p}(\mathbb{R}^{n-1})}^{1/n}.$$
Each of these inequalities is of the form \eqref{premainineq} with $\sum_{j=1}^n \beta_j = 1$. 

\medskip
\noindent
A very special case of the Loomis--Whitney inequality occurs in two dimensions where it becomes
the trivial identity
$$ \int_{\mathbb{R}^2} F_1(x_2)F_2(x_1) \; {\rm d} x_1  {\rm d} x_2 = \int_\mathbb{R} F_1 \int_\mathbb{R} F_2.$$
In spite of its simplicity, this example will play an important guiding role for us. See
Sections~\ref{factinterp}, \ref{Loomis--Whitneyrevisited}, \ref{NLLW} and \ref{BLfactdetails}.

\medskip
\noindent
The Loomis--Whitney inequality has many variants -- for example Finner's inequalities, the affine-invariant
Loomis--Whitney inequality and the nonlinear Loomis--Whitney inequality. See \cite{Fi}, \cite{BCW}, and
Sections ~\ref{Loomis--Whitneyrevisited} and \ref{NLLW}.
}
\end{example}

\begin{example}\label{BL}
{\rm [Brascamp--Lieb inequalities]
The class of Brascamp--Lieb inequalities includes the previous examples. Let $B_j : 
\mathbb{R}^n \to \mathbb{R}^{n_j}$ be linear surjections, $ 1 \leq j \leq d$. For $0 < p_j < \infty$ and $F_j$
nonnegative we consider the Brascamp--Lieb inequality
\begin{equation}\label{BL1}
\int_{\mathbb{R}^n} \prod_{j=1}^d F_j(B_jx)^{p_j}\; {\rm d} x \leq C \prod_{j=1}^d \left(\int_{\mathbb{R}^{n_j}}
F_j \right)^{p_j}.
\end{equation}
It is not hard to see that in order for this inequality to hold with a finite constant $C$, it is necessary that
$\sum_{j=1}^d p_j n_j = n$. It is known that the constant $C$ is finite if and only if,
in addition to  $\sum_{j=1}^d p_j n_j = n$, it holds that
$$ {\rm dim} \, V \leq \sum_{j=1}^d p_j {\rm dim} B_j V$$
for all $V$ in the lattice of subspaces of $\mathbb{R}^n$ generated by $\{ \ker B_j\}_{j=1}^d$. (See \cite{BCCT1}, 
\cite{BCCT2} and \cite{Vald}.) From this one sees easily that $\cap_{j=1}^d {\rm ker} B_j = \{ 0 \}$, 
$\sum_{j=1}^d p_j \geq 1$, and $p_j \leq 1$  are also necessary conditions
for the finiteness of $C$. A celebrated theorem of Lieb \cite{L} states that the value of the best constant $C$
is obtained by checking the inequality on Gaussian inputs $F_j$. Lieb's theorem generalises Beckner's theorem
\cite{Beckner} on extremisers for Young's convolution inequality.  

\medskip
\noindent
Suppose that $0 < r_j < \infty$ and $0 < s < \infty$.  Setting $F_j = f_j^{r_j}$ in \eqref{BL1} and taking $s$'th roots,
we see that \eqref{BL1} is equivalent to
$$\| \prod_{j=1}^d f_j(B_jx)^{p_jr_j/s} \|_{L^s(\mathbb{R}^n)} 
\leq {C}^{1/s} \prod_{j=1}^d \|f_j \|_{L^{r_j}(\mathbb{R}^{n_j})}^{p_jr_j/s}.$$
If $\sum_{j=1}^d p_j r_j = s$ this is an inequality of the form \eqref{premainineq} with $\sum_{j=1}^d \beta_j = 1$.
In particular we can take 
$r_j = n_j$ and $s =n$ to obtain the equivalent form
$$ \| \prod_{j=1}^d f_j(B_jx)^{p_jn_j/n} \|_{L^n(\mathbb{R}^n)} 
\leq {C}^{1/n} \prod_{j=1}^d \| f_j \|_{L^{n_j}(\mathbb{R}^{n_j})}^{p_jn_j/n};$$
or we can take $r_j = 1$ and $s = \sum_{j=1}^d p_j$ (recall that this number is at least $1$ when the inequality is nontrivial)
to obtain another equivalent form
\begin{equation}\label{BL3}
  \| \prod_{j=1}^d f_j(B_jx)^{p_j/s} \|_{L^s(\mathbb{R}^n)} 
  \leq {C}^{1/s} \prod_{j=1}^d \| f_j \|_{L^{1}(\mathbb{R}^{n_j})}^{p_j/s}.
  \end{equation}

\medskip
\noindent
A special case of the class of Brascamp--Lieb inequalities is the class of {\em geometric} Brascamp--Lieb inequalities.
Suppose that the linear surjections $B_j: \mathbb{R}^n \to \mathbb{R}^{n_j}$ satisfy
$$ \sum_{j=1}^d p_j B_j^\ast B_j = I_n.$$
Then, by a result of Ball and Barthe, (\cite{Ball} and \cite{Barthe}, see also \cite{BCCT1}) we have
\begin{equation}\label{BL2}
\int_{\mathbb{R}^n} \prod_{j=1}^d F_j(B_jx)^{p_j}\; {\rm d} x \leq \prod_{j=1}^d \left(\int_{\mathbb{R}^{n_j}}
F_j \right)^{p_j},
\end{equation}
and the sharp constant $1$ is achieved by the standard Gaussians $F_j(y) = e^{- \pi |y|^2}$. Correspondingly, in the
equivalent variants presented above, the constants are also $1$. The geometric Brascamp--Lieb inequalities include
a suitably reformulated version of the sharp Young inequality of Beckner \cite{Beckner}. 
See Section~\ref{appl1} and Section~\ref{GBL} for an application of the theory we present in the context of geometric Brascamp--Lieb inequalities.
}
\end{example}

\begin{example}\label{mgrt}
{\rm [Multilinear generalised Radon transforms]
  There is a vast literature on multilinear generalised Radon transforms into which we do not wish to enter.
  For us, this term will mean consideration of multilinear inequalities of the form \eqref{premainineq}
  when the operators $T_j$ take the form $T_jf = f \circ B_j$ for suitable mappings
  $B_j: X \to Y_j$. In most cases, $X$ and $Y_j$ will be endowed with a topological or smooth structure, and
  the mappings $B_j$ will respect that structure in such a way that issues of measurability do not arise.

  \medskip
  \noindent
  The class of multilinear generalised Radon transforms includes the Brascamp--Lieb inequalities. The most basic
  multilinear generalised Radon transform which is not included in the Brascamp--Lieb inequalities is probably
  the nonlinear Loomis--Whitney inequality. See Section~\ref{NLLW} below.
  
}
\end{example}

\begin{example}\label{MK}
{\rm [Multilinear Kakeya inequalities]
The Loomis--Whitney inequality of Example 
\ref{Loomis--Whitney} is equivalent to
$$ \int_{\mathbb{R}^n} \prod_{j=1}^n 
\left(\sum_{P _j \in \mathcal{P}_j} a_{P_j} \chi_{P_j}(x) \right)^{1/(n-1)} \; {\rm d}x 
\leq \prod_{j=1}^n \left(\sum_{P_j \in \mathcal{P}_j} a_{P_j}\right)^{1/(n-1)},$$
where $\mathcal{P}_j$ is a finite family of $1$-tubes in $\mathbb{R}^n$ which are 
parallel to the $j$'th standard basis vector $e_j$, and the $a_{P_j}$ are arbitrary positive numbers. 
(A $1$-tube is simply a neighbourhood of a doubly infinite line in $\mathbb{R}^n$ which has $(n-1)$-dimensional cross-sectional area equal to $1$.) Multilinear Kakeya 
inequalities have the same set-up, but now we allow the tubes in the family 
$\mathcal{P}_j$ to be {\em approximately} parallel to $e_j$, i.e. the direction 
$e(P) \in \mathbb{S}^{n-1}$ of the central axis of the tube $P \in \mathcal{P}_j$ must satisfy
$|e(P) - e_j| \leq c_n$ where $c_n$ is a small dimensional constant. Such inequalities have been 
studied in \cite{BCT}, \cite{MR2746348}, \cite{BG} and \cite{CV} and have proved to be very important over the last 
decade with significant applications in partial differential equations and especially in number theory --
see for example \cite{B7}, \cite{B8}, \cite{B9} and \cite{B10}.
The multilinear Kakeya inequality is the statement 
$$ \Big\|\prod_{j=1}^n \left(\sum_{P _j \in \mathcal{P}_j} a_{P_j} \chi_{P_j}(x) \right)^{1/n} \Big\|_{L^{n/(n-1)}(\mathbb{R}^n)} 
\leq C_n \prod_{j=1}^n \left(\sum_{P_j \in \mathcal{P}_j} a_{P_j}\right)^{1/n}.$$
This inequality is of the form \eqref{premainineq} with $X = \mathbb{R}^n$, $q = n/(n-1)$, $Y_j = \mathcal{P}_j$ with
counting measure, $p_j = 1$ for all $j$, $\beta_j = 1/n$ for all $j$, and $T ((a_{P_j}))(x) = \sum_{P _j \in \mathcal{P}_j} a_{P_j} \chi_{P_j}(x)$.
It was Guth's approach to such multilinear Kakeya inequalities in \cite{MR2746348} which inspired the present paper.

\medskip
\noindent
The recent multilinear $k_j$-plane Kakeya inequalities, and indeed the even more general perturbed Brascamp--Lieb
inequalities, both recently established by Zhang \cite{Z}, also fit into our framework, the latter as a generalisation of
inequality \eqref{BL3}.  
}
\end{example}

\bigskip
\noindent
We shall return to consider these examples in some detail later in Part III. In particular we shall discuss the
affine-invariant Loomis--Whitney inequality, the nonlinear Loomis--Whitney inequality and certain aspects of
Brascamp--Lieb inequalities in the light of the theory we develop.

\subsection{The weighted geometric mean operator}\label{wgmo}
\noindent
As we have just seen, all of our examples fit into the framework of inequality \eqref{premainineq} with
$\sum_{j=1}^d \beta_j = 1$, and with the $L^{p_j}$ (and $L^q$) spaces in the Banach regime, i.e. with
$p_j \geq 1$ (and $q \geq 1$). We shall therefore be concerned in this paper with norm
inequalities for the {\em weighted geometric mean} operator
$$ \mathcal{T}_\alpha:(f_1, \dots, f_d) \mapsto (T_1 f_1)^{\alpha_1} \cdots (T_d f_d)^{\alpha_d}$$ 
where $\alpha = (\alpha_1, \dots, \alpha_d)$ and the $\alpha_j$ are positive numbers satisfying
$ \sum_{j=1}^d \alpha_j = 1$. That is, we shall consider inequalities of the form

\begin{equation}
\label{submainineq}
\left\|\prod_{j=1}^d (T_jf_j)^{\alpha_j}\right\|_{L^q(X)}\leq A 
\prod_{j=1}^d \Big\|f_j\Big\|_{L^{p_j}(Y_j)}^{\alpha_j}
\end{equation}
for nonnegative simple functions $f_j \in \mathcal{S}(Y_j)$, in the regime $p_j \geq 1$ and $q > 0$.
While the case $q \geq 1$ is pertinent to our examples, we also wish to consider the case $0 < q < 1$
because this corresponds to the situation treated by Maurey in the linear setting. Throughout the paper,
we use the quantities $\alpha_j$ to represent positive numbers whose sum is $1$.

\medskip
\noindent
We have chosen to present our general theory for the weighted geometric mean operator $\mathcal{T}_\alpha$ --
which is manifestly {\em not} linear in its arguments $f_1, \dots , f_d$ -- mainly because of the
extra elegance and simplicity that such a treatment affords. Nevertheless, nearly all of the examples
above also have equivalent strictly multilinear formulations.
In particular, the multilinear Kakeya inequality of Example~\ref{MK} can be re-cast as the manifestly multilinear
$$ \Big\|\prod_{j=1}^n \left(\sum_{P _j \in \mathcal{P}_j} \beta_{P_j} \chi_{P_j}(x) \right) \Big\|_{L^{1/(n-1)}(\mathbb{R}^n)} 
\leq C \prod_{j=1}^n \left(\sum_{P_j \in \mathcal{P}_j} \beta_{P_j}\right).$$
The one class of examples 
that does not admit a genuinely multilinear reformulation consists of the perturbed Brascamp--Lieb inequalities which were
briefly mentioned in Example~\ref{MK}.

\medskip
\noindent
Our first purpose in this paper is to propose and undertake a systematic study of the duality theory
associated to the weighted geometric mean operator $\mathcal{T}_\alpha$ in the context of inequality
\eqref{submainineq} in the case $q \geq 1$, and some of its generalisations. It is hoped that the
framework for this multilinear duality theory will in time have applications in a wide
variety of contexts.
Our second purpose is to establish suitable analogues of Maurey's theorems in the
context of \eqref{submainineq} in the case $0 < q < 1$. Interestingly, the case $q=1$ will be central
to our development of both the regimes $q \geq 1$ and $0 < q < 1$, unlike in the classical linear
setting where the case $q=1$ is essentially vacuous, and in which there appears to be no direct
link between the two regimes $q \geq 1$ and $0 < q < 1$.

\subsection{A theory of multilinear duality -- the regime $q \geq 1$}\label{mdbanach}

\medskip
\noindent
We begin with the Banach regime $q \geq 1$.

\medskip
\noindent
One half of our duality theory -- the `easy' half -- is largely contained in the following 
simple observation, the content of which is that if we have a certain pointwise 
factorisation property for the space $L^{q^\prime}$, then the weighted geometric mean norm 
inequality \eqref{submainineq} will hold. 

\begin{prop}
\label{thmguthbaby}
Suppose that $T_j : L^{p_j}(Y_j) \to L^q(X)$ are positive linear operators, that $p_j, q \geq 1$
and that $\sum_{j=1}^d \alpha_j = 1$.
Suppose that for every nonnegative $G \in L^{q^\prime}(X)$ there exist nonnegative measurable functions $g_j$
defined on $X$ such that
\begin{equation}
\label{mainprobbaby}
\begin{aligned}
&G(x) \leq \prod_{j=1}^d g_j(x)^{\alpha_j}\qquad\text{a.e. on $X$,}\\
\text{and}\qquad &\Big\|\,T_j^\ast g_j\Big\|_{L^{{p_j}^\prime}(Y_j)} \leq A\Big\|G\Big\|_{L^{q^\prime}(X)}\qquad\text{for all $j$.}
\end{aligned}
\end{equation}
Then, for all nonnegative $f_j \in \mathcal{Y}_j$,
$$ \; \; \; \; \; \qquad\left\|\prod_{j=1}^d (T_jf_j)^{\alpha_j}\right\|_{L^q(X)}
\leq A \prod_{j=1}^d \Big\|f_j\Big\|_{L^{p_j}(Y_j)}^{\alpha_j} \; ;$$
that is, \eqref{submainineq} holds, for all nonnegative $f_j \in L^{p_j}(Y_j)$.
\end{prop}

\medskip
\noindent
For the (easy) proof and some discussion of this result, see the more general Proposition~\ref{thmguth} below.

\medskip
\noindent
Rather surprisingly, the implication in Proposition \ref{thmguthbaby} can be essentially reversed, and 
one of the main aims of this paper is to show that the factorisation property \eqref{mainprobbaby} enunciated in  Proposition 
\ref{thmguthbaby} is in fact {\em necessary} as well as sufficient for \eqref{submainineq} to hold. This is 
the second half of the multilinear duality principle referred to in the abstract of the paper.

\medskip
\noindent
Before coming to this, however, we note that if there is a subset of $X$ of positive measure upon which $T_jf_j$ vanishes for
all $f_j \in L^{p_j}$, then this subset will play no role in the analysis of inequality \eqref{submainineq}. There is
therefore no loss of generality in assuming such subsets do not exist. We formalise this notion by introducing
the notion of {\bf saturation} below.\footnote{For a related notion, see \cite{Zaanen}.}
In order to facilitate what follows later, we at the same time introduce
the closely related notion of  {\bf strong saturation}, and also make the definitions in slightly greater generality than what is required
by the current discussion. The definitions apply to linear operators $T: \mathcal{Y} \to \mathcal{M}(X)$,
with $\mathcal{Y}$ a normed lattice and $(X, {\rm d} \mu)$ a measure space,
which are positive in the sense that for every nonnegative $f \in \mathcal{Y}$ we have $Tf \geq 0$.
(What is currently relevant is the fact that the space of simple functions defined on a measure space $Y$,
together with the $L^p$ norm for $p \geq 1$, forms a normed lattice.) 

\begin{definition}\label{saturation}

  \noindent
  (i) We say that $T$ {\bf saturates} $X$ if for each subset
  $E \subseteq X$ of positive measure, there exists a subset 
  $E^\prime  \subseteq E$ with $\mu(E^\prime) > 0$ and a nonnegative $h \in \mathcal{Y}$
  such that $Th > 0$ a.e. on $E^\prime$.

  \smallskip
  \noindent
  (ii) We say that $T$ {\bf strongly saturates} $X$ if  there exists a nonnegative
$h \in \mathcal{Y}$ such that $Th$ is a.e. bounded away from $0$ on $X$.
\end{definition}

For further discussion of the relevance of these conditions, see Remarks~\ref{bnm} and \ref{mnb} below. If $T$ saturates a $\sigma$-finite measure space $X$, then there is an
increasing and exhausting sequence of measurable subsets on each of which $T$ is strongly saturating. For this and more, see Lemma~\ref{sat} below.

\medskip
\noindent
Now we can state one of the main results of the paper:
     
\begin{thm}
\label{thmmainbaby}
Suppose that $X$ and $Y_j$, for $j=1,\dots,d$, are measure spaces. % and that each $Y_j$ is $\sigma$-finite.
Suppose that the linear
operators $T_j: \mathcal{S}(Y_j) \to \mathcal{M}(X)$ are positive
 and that each $T_j$ saturates $X$. Suppose that $p_j \geq 1$ for all $j$, $ 1 \leq q \leq \infty$ and 
$\sum_{j=1}^d \alpha_j = 1$. When $q=1$ suppose additionally that $X$ is $\sigma$-finite. Finally, suppose that
\begin{equation*}
\left\|\prod_{j=1}^d (T_jf_j)^{\alpha_j}\right\|_{L^q(X)}
\leq A \prod_{j=1}^d \Big\|f_j\Big\|_{L^{p_j}(Y_j)}^{\alpha_j}
\end{equation*}
for all nonnegative simple functions $f_j$ on ${Y}_j$, $1 \leq j \leq d$. Then for every nonnegative
 $G \in L^{q^\prime}(X)$ there exist nonnegative measurable functions 
$g_j$ on $X$ such that 
\begin{equation}\label{factorisebaby}
G(x) \leq \prod_{j=1}^d g_j(x)^{\alpha_j}\qquad\mbox{a.e. on $X$,}
\end{equation}
and such that for each $j$, 
\begin{equation}\label{controlbaby}
\int_X g_j(x)T_jf_j(x) {\rm d}\mu(x) \leq A\|G\|_{L^{q^{\prime}}}\|f_j\|_{p_j}
\end{equation}
for all simple functions $f_j$ on ${Y}_j$.
\end{thm}

\begin{remark}{\rm 
Note that we have used the formulation
\eqref{controlbaby} instead of one explicitly involving $T_j^\ast$ as we did in \eqref{mainprobbaby} because it is not
immediately clear how $T_j^*$ should be defined this context.
 }
  \end{remark}

\medskip
\noindent
The special case of Theorem~\ref{thmmainbaby} corresponding to $q=1$ and $G \equiv 1$ can be singled out:

\begin{thm}
\label{cormainbaby}
Suppose that $X$ and $Y_j$, for $j=1,\dots,d$, are  measure spaces, with $X$ being $\sigma$-finite.
Suppose that the operators $T_j: \mathcal{S}(Y_j) \to \mathcal{M}(X)$ are positive
and that each $T_j$ saturates $X$. Suppose that $p_j\geq 1$, $\sum_{j=1}^d \alpha_j = 1$
and that
\begin{equation*}
\int_X \prod_{j=1}^d (T_jf_j)^{\alpha_j} {\rm d} \mu
\leq A \prod_{j=1}^d \Big\|f_j\Big\|_{L^{p_j}(Y_j)}^{\alpha_j}
\end{equation*}
for all nonnegative simple functions $f_j$ on ${Y}_j$, $1 \leq j \leq d$. Then there exist nonnegative
measurable functions $g_j$ on $X$ such that 
\begin{equation*}
1 \leq \prod_{j=1}^d g_j(x)^{\alpha_j}\qquad\mbox{a.e. on $X$,}
\end{equation*}
and such that for each $j$, 
\begin{equation}\label{controlbabybaby}
\int_X g_j(x)T_jf_j(x) {\rm d}\mu(x) \leq A \|f_j\|_{p_j}
\end{equation}
for all simple functions $f_j$ on ${Y}_j$.
\end{thm}

\medskip
\noindent
In fact,  Theorem~\ref{cormainbaby} implies Theorem~\ref{thmmainbaby}. Indeed, suppose that $1 < q \leq \infty$ and that 
\begin{equation*}
\left\|\prod_{j=1}^d (T_jf_j)^{\alpha_j}\right\|_{L^q(X)}
\leq A \prod_{j=1}^d \Big\|f_j\Big\|_{L^{p_j}(Y_j)}^{\alpha_j}
\end{equation*}
for all nonnegative simple functions $f_j$ on ${Y}_j$, $1 \leq j \leq d$. Then, for all nonnegative
$G \in L^{q'}(X)$ with $\|G\|_{L^{q'}}=1$, we have
\begin{equation*}
\int_X \prod_{j=1}^d (T_jf_j)^{\alpha_j} G \, {\rm d} \mu
\leq A \prod_{j=1}^d \Big\|f_j\Big\|_{L^{p_j}(Y_j)}^{\alpha_j}
\end{equation*}
for all nonnegative simple functions $f_j$ on ${Y}_j$, $1 \leq j \leq d$. It is easy to see that
if $T_j$ saturates $X$ with respect to the measure ${\rm d} \mu$, then it also does so with respect to $G \,
{\rm d}\mu$. Now the measure $G \, {\rm d} \mu$ is $\sigma$-finite irrespective of whether ${\rm d}\mu$ is $\sigma$-finite measure. Therefore, by Theorem~\ref{cormainbaby}
applied with the measure $G \, {\rm d} \mu$ in place of ${\rm d} \mu$, there are nonnegative measurable
functions $\gamma_j$ such that 
\begin{equation*}
1 \leq \prod_{j=1}^d \gamma_j(x)^{\alpha_j} \qquad G \, {\rm d} \mu \mbox{-a.e. on $X$,}
\end{equation*}
and such that for each $j$, 
\begin{equation*}
\int_X \gamma_j(x)T_jf_j(x) G(x) {\rm d}\mu(x) \leq A \|f_j\|_{p_j}
\end{equation*}
for all simple functions $f_j$ on ${Y}_j$. Setting $g_j = \gamma_j G$ gives the
desired conclusion of Theorem~\ref{thmmainbaby} when $q > 1$. When $q=1$, factorisation of the function
$1$ as in Theorem~\ref{cormainbaby} immediately yields a corresponding factorisation of each $G \in L^\infty$.

\medskip
\noindent
The results described here will follow from the more general Theorem~\ref{thmmain} below.

\subsubsection{An application to pointwise factorisation}\label{appl1}
As an application of Theorem~\ref{thmmainbaby}, we have the following sample result concerning pointwise factorisation of nonnegative functions in $L^2(\mathbb{R}^2)$:

\begin{thm}
  Let $v_1, v_2$ and $v_3$ be unit vectors in $\mathbb{R}^2$ with angle $2\pi/3$ between each pair. Then, for every nonnegative $G \in L^2(\mathbb{R}^2)$,
    there exist nonnegative locally integrable functions $g_1, g_2$ and $g_3$ such that
     $$ G(x) \leq g_1(x)^{1/3} g_2(x)^{1/3} g_3(x)^{1/3} \; \; \mbox{ a.e.}$$
    and, for each $j$, for almost every line $l$ in $\mathbb{R}^2$ which is
    parallel to $v_j$,
    $$ \int_l g_j {\rm d} \lambda  \leq \|G\|_2$$
    where ${\rm d}\lambda$ denotes Lebesgue measure on $l$.
\end{thm}

For further details, and many more results of this nature, see Section~\ref{GBL}
below.

\subsection{Multilinear Maurey-type factorisation -- the regime $0 < q < 1$.}\label{mmf}
We now state a multilinear Maurey-type theorem. 

\begin{thm}
\label{thmmainmaurey}
Suppose that $X$ and $Y_j$, for $j=1,\dots,d$, are measure spaces and that $X$ is  $\sigma$-finite. Suppose that the 
operators $T_j: \mathcal{S}(Y_j) \to \mathcal{M}(X)$ are positive  
and that each $T_j$ saturates $X$. Suppose that $p_j \geq 1$, $0 < q <1$,
$\sum_{j=1}^d \alpha_j = 1$ and that
\begin{equation}\label{987}
\left\|\prod_{j=1}^d (T_jf_j)^{\alpha_j}\right\|_{L^q(X)}
\leq A \prod_{j=1}^d \Big\|f_j\Big\|_{L^{p_j}(Y_j)}^{\alpha_j}
\end{equation}
for all nonnegative simple functions $f_j \in \mathcal{Y}_j$, $1 \leq j \leq d$. 
Then there exist nonnegative measurable functions 
$g_j$ on $X$ such that 
\begin{equation}\label{factorisebabymaurey}
\|\prod_{j=1}^d g_j(x)^{\alpha_j}\|_{L^{q'}(X)} = 1
\end{equation}
and such that for each $j$, 
\begin{equation}\label{controlbabymaurey}
\int_X g_j(x)T_jf_j(x) {\rm d}\mu(x) \leq A\|f_j\|_{L^{p_j}(Y_j)}
\end{equation}
for all simple functions $f_j$ on ${Y}_j$.
\end{thm}

We shall give the proof of this result, as a consequence of Theorem~\ref{cormainbaby}, in Section~\ref{maureysec} below.
Conversely, it is easy to see using H\"older's inequality that if there exist $g_j$ such that
\eqref{factorisebabymaurey} and \eqref{controlbabymaurey} hold, then so does
\eqref{987}.

\medskip
\noindent
This result can be seen as a factorisation result in the spirit of Maurey: if we let $S_j f_j(x) = g_j(x) T_j f_j(x)$,
$\mathcal{S}_\alpha(f_1, \dots , f_d) = \prod_{j=1}^d(S_j f_j)^{\alpha_j}$ and
$g(x) = \prod_{j=1}^d g_j(x)^{\alpha_j}$, then
$$ \mathcal{T}_\alpha  = M_{g^{-1}} \circ \mathcal{S}_\alpha$$
where
$$ \|S_j\|_{L^{p_j} \to L^1} \leq A $$
for all $j$, and
$$  \|M_{g^{-1}}\|_{L^1 \to L^q} = \|g\|_{q'}^{-1} = 1.$$
In fact it has the rather strong conclusion that {\em each} $S_j$ is bounded from $L^{p_j}(Y_j)$ to $L^1( {\rm d} \mu)$ with
constant at most $A$ (rather than the much weaker corresponding statement for the geometric mean $\mathcal{S}_\alpha$ alone). 

\medskip
\noindent
Other, different, versions of multilinear Maurey-type theorems have been studied. See for example
\cite{Schep} and \cite{Diestel}.

\begin{remark}\label{extn}
One may use the classical linear Maurey--Nikisin--Stein factorisation theory of
  positive operators (\cite{Mau}, Proposition 9) to upgrade conclusions
  \eqref{controlbaby} of Theorem~\ref{thmmainbaby}, \eqref{controlbabybaby} of Theorem~\ref{cormainbaby} and \eqref{controlbabymaurey} of
  Theorem~\ref{thmmainmaurey}. Indeed, each of these conclusions states that $T_j$ maps into a weighted
  $L^1$-space, and we can upgrade each to boundedness of $T_j$ into a suitable weighted
  $L^{p_j}$-space. For example, in the context of Theorem~\ref{cormainbaby}, we may conclude that there exist
  nonnegative measurable functions $\phi_j$ on $X$ such that
  $$\int_X \left(\prod_{j=1}^d\phi_j(x)^{{\alpha_j/p_j}}\right)^{-1/{(\sum_{j=1}^d {\alpha_j/{p_j}^\prime})}} {\rm d} \mu (x) \leq 1$$
  and
  $$ \left(\int_X |T_j f_j|^{p_j} \phi_j {\rm d} \mu\right)^{1/p_j} \leq A \|f_j\|_{p_j}$$
  for all simple $f_j$ on $Y_j$. See also Remark~\ref{paty} below. For stronger statements of this kind
  see the forthcoming \cite{CHV}.
  \end{remark}

\subsection{Structure of the paper}
The paper is divided into three parts.

\medskip
\noindent
In Part I (Sections 2--5) we present the theory of multilinear duality and factorisation and prove the main theorems.

\medskip
\noindent
In Section~\ref{sdmr} we state and discuss the main results at some length. The principal result is Theorem~\ref{thmmain}.
Taken together with Proposition~\ref{thmguth}, Theorem~\ref{thmmain} forms the statement of the multilinear duality
principle referred to in the abstract of the paper. (Theorem~\ref{thmmainbaby} and Proposition~\ref{thmguthbaby}
presented in this introduction
are more readily digested versions of Theorem~\ref{thmmain} and Proposition~\ref{thmguth} respectively.) The multilinear Maurey
factorisation theorem, Theorem~\ref{thmmainmaureyx}, is proved as a consequence of Theorem~\ref{thmmain}. (A more
digestible version of Theorem~\ref{thmmainmaureyx} is found as Theorem~\ref{thmmainmaurey} in this introduction.)

\medskip
\noindent
In Section~\ref{outline} we give a proof of a finitistic case of Theorem~\ref{thmmain} which recognises and emphasises
its structure as a convex optimisation or minimax problem. This perspective sets the scene for the remainder of the theoretical
part of the paper. In this case, none of the functional-analytic and measure-theoretic difficulties that we encounter
later are present. However, Section~\ref{outline} is not strictly speaking logically necessary for the development of the theory.

\medskip
\noindent
In Section~\ref{overv} we begin to address the proof of Theorem~\ref{thmmain}. Our strategy will be to first consider
the setting of finite measure spaces. Theorem~\ref{thmsimple} gives the main result in this case, and it
represents a crucial step in the proof of Theorem~\ref{thmmain}.
Already in Theorem~\ref{thmsimple} we are faced with substantial functional-analytic and measure-theoretic difficulties.
These derive from the need to establish certain compactness statements necessary for the application of a minimax theorem.
Briefly, they involve working with the dual space of $L^\infty(X)$, and dealing with various issues in the theory of
finitely additive measures.

\medskip
\noindent
In Section~\ref{dett} we give the details of the proofs of Theorem~\ref{thmsimple} and Theorem~\ref{thmmain}. We begin with
a couple of technical but very important lemmas. Next, we pass to the proof of the finite-measure result, Theorem~\ref{thmsimple},
via the minimax theory. Finally, for general $\sigma$-finite $X$, we ``glue together'' factorisations obtained for subsets $X$ of finite measure
via Theorem~\ref{thmsimple}, and we obtain the factorisations needed for Theorem~\ref{thmmain}.

\medskip
\noindent
In a much shorter Part II, we begin to explore connections with other topics -- in particular the theory of interpolation in
Sections~\ref{factinterp} and \ref{qzf}, and the extent to which the theory might apply in the context of more general multilinear
operators in Section~\ref{culture}.

\medskip
\noindent
Finally, in Part III, we revisit the examples discussed earlier in this introduction in the light of the multilinear duality theory which has
been developed. In Section~\ref{classical} we give factorisation-based proofs of the affine-invariant Loomis--Whitney inequality
(see Section~\ref{Loomis--Whitneyrevisited}) and the sharp nonlinear Loomis--Whitney inequality (see Section~\ref{NLLW}).
In Section~\ref{GBL} we pose an interesting question related to
the sharp Young convolution inequality and geometric Brascamp--Lieb inequalities, while in Section~\ref{GGBL} we describe an algorithm for
factorising the general Brascamp--Lieb inequality. In Section~\ref{MKrevisited} we revisit the multilinear Kakeya inequality which
inspired the paper in the light of the findings of Section~\ref{culture}, and make an observation about the size of the constant in the finite-field version 
of the multilinear Kakeya inequality which is derived from our methods.

\subsection{Future work}\label{future}
In a sequel \cite{CHV} to this paper, we broaden the scope of the multilinear duality theory from positive to potentially oscillatory multilinear inequalities. In particular, we extend the multilinear duality and factorisation theorem
(Theorem~\ref{thmmain}) to non-positive operators, and indeed refine it when the normed spaces have
certain additional geometric properties ($p$-convexity in case of positive operators, Rademacher type in
case of non-positive operators). Moreover, in forthcoming work \cite{CHV1}, we will give an alternative proof of
Theorem~\ref{thmmain} which bypasses the need to consider $(L^\infty)^\ast$.

\subsection{Acknowledgements} This paper has benefited substiantally from discussions with many individuals. In particular
we should like to thank Keith Ball, Jon Bennett, Michael Christ, Michael Cowling, Alastair Gillespie, Gilles Pisier, 
Sandra Pott, Stuart White and Jim Wright for the insights they have shared with us on various aspects of the material
of the paper.
The first author is especially grateful to Michael Christ for his ongoing encouragement in the quest to find explicit factorisations. The second author is grateful to Igor Verbitsky for introducing him to the Maurey theory
of factorisation in their collaboration  \cite{HV} on characterising two-weight norm inequalities via
factorisation. The first and third author would like to record their
appreciation of the hospitality and support of the Isaac Newton Institute during the
programme ``Discrete Analysis'' between March and July 2011, where the preliminary stages of this research
were carried out and presented. The second author is supported by the Academy of Finland through
funding of his postdoctoral researcher post (Funding Decision No 297929), and he is a member of the
Finnish Centre of Excellence in Analysis and Dynamics Research.

\part{Statements and proofs of the theorems}

\bigskip
\section{Statement and discussion of the main results}\label{sdmr}

It turns out that the theory we shall develop is most naturally presented in a more general setting. Moreover,
limiting ourselves to the Lebesgue spaces $L^{p_j}$ and $L^q$ in the multilinear duality theory is
unnecessarily restrictive.
For example, one may wish to consider multilinear inequalities of the form~\eqref{submainineq} in which the $L^{p_j}$ and $L^q$ spaces
are replaced by certain Lorentz spaces, Orlicz spaces or mixed-norm spaces, especially if the inequality under consideration is
an endpoint inequality.
We therefore introduce a more general framework in which we consider suitable spaces $\mathcal{X}$ and $\mathcal{Y}_j$ corresponding to
$L^q(X)$ and $L^{p_j}(Y_j)$ respectively.
Thus, for $\sum_{j=1}^d \alpha_j = 1$, which we recall is a standing convention, we now consider inequalities of the form

\begin{equation}
\label{mainineq}
\left\|\prod_{j=1}^d (T_jf_j)^{\alpha_j}\right\|_{\mathcal{X}} \leq A
\prod_{j=1}^d \Big\|f_j\Big\|_{\mathcal{Y}_j}^{\alpha_j}.
\end{equation}

\medskip
\noindent
Each $\mathcal{Y}_j$ will be an (abstract) normed lattice -- such as $L^1$ or $L^{p_j}$ for $p_j \geq 1$. 
On the other hand, we will take $\mathcal{X}$ to be Banach space of locally integrable\footnote{That is,
$\int_E |f| {\rm d} \mu < \infty$ whenever $\mu(E) < \infty$.} functions defined on $X$,
which contains the simple functions, and is such that if $f \in \mathcal{M}(X)$ and
$g \in \mathcal{X}$ satisfy $|f(x)| \leq |g(x)|$
a.e., then $f \in \mathcal{X}$ and $\|f\|_{\mathcal{X}} \leq \|g\|_{\mathcal{X}}$.
We may as well also assume that $X$ is a complete measure space. Our measure space
$X$ will also be assumed to the $\sigma$-finite; these properties together then
identify $\mathcal{X}$ as a {\bf K\"othe space}, see \cite{LT}, Vol.~II, p.~28.
{\em We shall from now on assume that  $\mathcal{X}$ is a  K\"othe space without
  further mention.}
Natural examples of K\"othe spaces include $L^q$ for $1 \leq q \leq \infty$.
We shall need a suitable primordial dual of $\mathcal{X}$, denoted by $\mathcal{X}'$, and defined to be
$$\mathcal{X}^\prime = \{ g \in \mathcal{M}(X) \, : \, \|g\|_{\mathcal{X}'} = \sup_{ \|f\|_\mathcal{X} \leq 1} \int |fg|\, {\rm d} \mu < \infty \}.$$
The space $\mathcal{X}'$ is usually called the {\bf K\"{o}the dual} of $\mathcal{X}$.
If $\mathcal{X} = L^q$ for $1 \leq q \leq \infty$, then $\mathcal{X}^\prime = L^{q'}$ where
$1/q + 1/ q^\prime = 1$. It is clear that $\mathcal{X}'$ is a linear space which contains
the simple functions (as $\mathcal{X}$ is contained in the class of locally integrable functions)
and is contained in the class of locally integrable functions (as $\mathcal{X}$ contains the simple functions). The
quantity $ \|g\|_{\mathcal{X}^\prime}$ defines a norm on $\mathcal{X}'$.
While by definition we always have the H\"older inequality
$$ |\int fg \, {\rm d} \mu| \leq \|f\|_{\mathcal{X}} \|g \|_{\mathcal{X'}},$$
it may or may not be the case that $\mathcal{X}'$ is {\bf norming} (for $\mathcal{X}$), i.e. that
\begin{equation}\label{norming}
\|f\|_{\mathcal{X}} = \sup \{\, | \int fg \, {\rm d} \mu |\, : \, \|g\|_{\mathcal{X}^\prime} \leq 1 \}
\end{equation}
holds for all $f \in \mathcal{X}$.\footnote{By a result of Lorentz and Luxemburg (see \cite{LT},
Vol.~II, p.~29), if $X$ is a K\"othe space, $\mathcal{X}'$ is norming if and
only if $\mathcal{X}$ has the so-called Fatou property, that is, whenever
$f_n \in \mathcal{X}$ are such that $f_n \to f$ a.e., with $f_{n+1} \geq f_n \geq 0$, 
then $\|f_n \|_\mathcal{X} \to \|f\|_\mathcal{X}$. This is 
  automatic when $\mathcal{X}$ is separable. If $\mathcal{X}$ is $L^\infty$
  then \eqref{norming} holds by inspection since $\mathcal{X}^\prime$ is simply $L^1$. We shall
  need the notion of norming only for Proposition~\ref{thmguth}.} The K\"othe dual $\mathcal{X}'$ is always
isometrically embedded in the norm-dual $\mathcal{X}^\ast$,
but the two spaces may not coincide in general.

\medskip
    {\em From now on, we shall adopt once and for all the convention that all named functions
      ($f, g, h, F, G$, $H, \beta, G, S, \psi$ etc., often 
      adorned with subscripts) are assumed to be nonnegative. The two exceptions to this are the
      functions $L$ and $\Lambda$ appearing in the proofs of the main results.}

\subsection{Duality theory -- easy half}\label{deh}
The easy half of our duality theory is expressed in the following 
simple observation, the content of which is that if we have a certain 
factorisation property for the K\"othe dual $\mathcal{X}^\prime$, then the weighted geometric mean norm 
inequality \eqref{mainineq} will hold. 
\begin{prop}
\label{thmguth}
Suppose that $\mathcal{X}$ is a K\"othe space whose K\"othe dual $\mathcal{X}'$ is norming, and that
$\mathcal{Y}_j$ are normed lattices. Suppose that $T_j : \mathcal{Y}_j \to \mathcal{X}$ are positive linear operators. Suppose furthermore that for every
nonnegative $G \in \mathcal{X}^\prime$ there exist nonnegative measurable functions $g_j$ on $X$ such that
\begin{equation}
\label{mainprob}
\begin{aligned}
&G(x) \leq \prod_{j=1}^d g_j(x)^{\alpha_j}\qquad\text{a.e. on $X$,}\\
\text{and}\qquad &\Big\|\,T_j^\ast g_j\Big\|_{\mathcal{Y}_j^\ast} \leq A\Big\|G\Big\|_{\mathcal{X}^\prime}\qquad\text{for all $j$.}
\end{aligned}
\end{equation}
Then, for all nonnegative $f_j \in \mathcal{Y}_j$
$$ \; \; \; \; \; \qquad\left\|\prod_{j=1}^d (T_jf_j)^{\alpha_j}\right\|_{\mathcal{X}}
\leq A \prod_{j=1}^d \Big\|f_j\Big\|_{\mathcal{Y}_j}^{\alpha_j} \; .$$
That is, \eqref{mainineq} holds, for all nonnegative $f_j \in \mathcal{Y}_j$.

\end{prop}
\begin{proof}
Take $f_j\in\mathcal{Y}_j$ for $j=1,\dots,d$, and $G\in\mathcal{X}^\prime$ with $\|G\|_{\mathcal{X}^\prime}\leq1$. Then
\begin{align*}
    \int_X G(x) \prod_{j=1}^d (T_jf_j)^{\alpha_j} {\rm d} \mu(x) & 
    \leq\int_X \prod_{j=1}^dg_j(x)^{\alpha_j}\prod_{j=1}^d T_jf_j(x)^{\alpha_j} {\rm d}\mu(x)
    =\int_X \prod_{j=1}^d(g_j(x) T_jf_j(x))^{\alpha_j} {\rm d} \mu(x)\\
    &
    \leq\prod_{j=1}^d\left(\int_X g_j(x) T_jf_j(x) {\rm d}\mu(x)\right)^{\alpha_j} 
    = \prod_{j=1}^d \,\left((T_j^\ast g_j)(f_j)\right)^{\alpha_j}\\
    & 
    \leq \prod_{j=1}^d \left(\|T_j^\ast g_j\|_{\mathcal{Y}_j^\ast} \|f_j\|_{\mathcal{Y}_j}\right)^{\alpha_j}
    \leq \prod_{j=1}^d\left(A\|G\|_{\mathcal{X}^\prime}\|f_j\|_{\mathcal{Y}_j}\right)^{\alpha_j}
    \leq A\prod_{j=1}^d\|f_j\|_{\mathcal{Y}_j}^{\alpha_j}
\end{align*}
where the inequalities follow in order from the first condition of \eqref{mainprob}, H\"older's inequality,
the second condition of \eqref{mainprob}, and the assumption that $\|G\|_{\mathcal{X}^\prime}\leq1$.
The proposition now follows by taking the supremum over all such $G$, using the fact that $\mathcal{X}'$ is norming for $\mathcal{X}$.
\end{proof}

\begin{remark}
  If the spaces $\mathcal{Y}_j$ are complete, the assumption that $T_j: \mathcal{Y}_j \to \mathcal{X}$ is positive
  automatically implies that $T_j$ is bounded,\footnote{Indeed, if not, we can find nonnegative $f_n$ with
    $\|f_n\| \leq 2^{-n}$ but $\|T_j f_n \| \geq 2^n$. So for each $n$, $2^n \leq \|T_j f_n \| \leq \|T_j (\sum_{n=1}^\infty f_n)\| \leq C$
    for some finite $C$ 
    since $\sum_{n=1}^\infty f_n \in \mathcal{Y}_j$ (because $\mathcal{Y}_j$ is a Banach space). This is a contradiction.} 
and so the adjoint operator $T_j^\ast$ is well-defined. If not, we interpret the second condition of \eqref{mainprob} as
\begin{equation}
\label{mainprobfixed2}
\int_X g_j(x)T_jf_j(x) {\rm d}\mu(x) \leq A\|G\|_{\mathcal{X}^\prime}\|f_j\|_{\mathcal{Y}_j}
\end{equation}
for $f_j \in \mathcal{Y}_j$ and $j=1,\dots,d$, and we can still conclude the validity of \eqref{mainineq} for functions $f_j \in \mathcal{Y}_j$, as the proof clearly demonstrates. 
\end{remark}

\begin{remark}
If the $T_j$ are known to be bounded, it is immediate that \eqref{mainineq} holds with $A$ replaced by
$\prod_{j=1}^d \|T_j\|^{\alpha_j}$.\footnote{This follows since $\| \prod_{j=1}^d h_j^{\alpha_j} \| \leq \prod_{j=1}^d \| h_j\|^{\alpha_j}$ 
which in turn follows from the case where each $\|h_j\| = 1$, which itself follows by Young's numerical inequality and the 
triangle inequality.} 
However, the best constant $A$ in \eqref{mainprob} will in general be much smaller, and this assertion is 
the main content of Proposition~\ref{thmguth}.
\end{remark}

\begin{remark}
Observe that Proposition \ref{thmguth} does not require any topological structure of
the space $X$, only its nature as a measure space.
\end{remark}

\begin{remark}\label{bnm}
Notice that in order for the proof to go through, we only require that the factorisation property -- i.e. the first condition 
of \eqref{mainprob} -- holds for those $x$ which contribute to $\int_X G(x) \prod_{j=1}^d (T_jf_j)^{\alpha_j} {\rm d} \mu(x)$
for some functions $f_j \in \mathcal{Y}_j$. In other words, if a set $E \subseteq X$ with $\mu(E) >0$ has the property that 
for all choices $f_j$ of nonnegative functions in $\mathcal{Y}_j$, $\prod_{j=1}^d (T_jf_j)(x)^{\alpha_j} = 0$ a.e. 
on $E$, then $E$ will play no role in the analysis. There is therefore no loss of generality in assuming such sets 
do not exist. So we may assume without loss of generality that for all $E \subseteq X$ with $\mu(E) > 0$, there 
exist nonnegative $f_j \in \mathcal{Y}_j$ such that ``$\prod_{j=1}^d T_j f_j(x)^\alpha_j = 0$ a.e. on $E$'' fails -- 
i.e. such that there exists $E' \subseteq E$, with $\mu(E') > 0$ such that $\prod_{j=1}^d T_j f_j(x)^\alpha_j > 0$ on $E'$.
That is, we may assume that for all $E \subseteq X$ with $\mu(E) > 0$, there exists $E' \subseteq E$ with $\mu(E') > 0$,
and, for each $j$, a nonnegative $f_j \in \mathcal{Y}_j$ such that for all $x \in E'$, $T_jf_j(x) >0$. This condition is
equivalent to the formally slightly weaker condition that for each $j$, for all $E \subseteq X$ with $\mu(E) > 0$,
there exists $E' \subseteq E$ with $\mu(E') > 0$ and nonnegative $f_j \in \mathcal{Y}_j$ such that for all
$x \in E'$, $T_jf_j(x) >0$.\footnote{If the latter condition holds, apply it to each $j$ in turn to obtain the
  former condition.}
But this is simply the statement that each $T_j$ saturates $X$, as in Definition~\ref{saturation}.
It is unsurprising that we will require saturation when it comes to formulating and proving
the converse statement.

\end{remark}

\begin{remark}\label{weaksuff}
Note that in place of 
$$\qquad \Big\|\,T_j^\ast g_j\Big\|_{\mathcal{Y}_j^\ast} \leq A\Big\|G\Big\|_{\mathcal{X}^\prime}\qquad\text{for all $j$.}$$
we could have assumed the (formally weaker) condition
$$\qquad \prod_{j=1}^d \Big\|\,T_j^\ast g_j\Big\|_{\mathcal{Y}_j^\ast}^{\alpha_j} \leq A\Big\|G\Big\|_{\mathcal{X}^\prime}.$$
(A homogeneity argument shows that the two conditions are indeed equivalent.)
\end{remark}

\begin{remark}
  Similarly, it suffices to suppose a formally weaker hypothesis (``weak factorisation''), namely that for every
  $G \in \mathcal{X}^\prime$, there exist measurable functions $g_{jk}$ on $X$ such that
  $$G(x) \leq \sum_k \prod_{j=1}^d g_{jk}(x)^{\alpha_j}$$
  a.e. on $X$, and
  $$ \sum_k \Big\|\,T_j^\ast g_{jk}\Big\|_{\mathcal{Y}_j^\ast} \leq A\Big\|G\Big\|_{\mathcal{X}^\prime}$$
  for all $j$. But if this holds, and if we define $g_j = \sum_k g_{jk}$, Minkowski's inequality and H\"older's inequality
  yield \eqref{mainprob}. So this observation does not represent a genuine broadening
  of the scope of Proposition~\ref{thmguth}. 
  \end{remark}

\begin{remark}
The argument for Proposition \ref{thmguth} was effectively given by Guth, in a less abstract form, in his proof the 
endpoint multilinear Kakeya inequality \cite{MR2746348}. However, any strategy which includes an application Proposition 
\ref{thmguth} to establish an inequality of the form \eqref{mainineq} involves the potentially difficult matter of 
first finding a suitable factorisation. Indeed, the main work of \cite{MR2746348} consisted precisely in finding 
such. In this context see also \cite{CV}. 
\end{remark}

\subsection{Duality theory -- difficult half}\label{ddh}
As suggested above, the implication in Proposition \ref{thmguth} can be essentially reversed, and 
a principal aim of this paper is to show that the factorisation property \eqref{mainprob} enunciated
in Proposition 
\ref{thmguth} is in fact {\em necessary} as well as sufficient for \eqref{mainineq} to hold under very mild hypotheses. 
More precisely we prove:

\begin{thm}[Multilinear duality and factorisation theorem]
\label{thmmain}
Suppose that $(X, {\rm d} \mu)$ is a $\sigma$-finite measure space, $\mathcal{X}$ is a K\"othe space of measurable functions on $X$, $\mathcal{Y}_j$ are normed lattices, and
$T_j : \mathcal{Y}_j \to \mathcal{M}(X)$ are positive linear maps.
Suppose that each $T_j$ saturates $X$. Suppose that 
\begin{equation*}
\left\|\prod_{j=1}^d (T_jf_j)^{\alpha_j}\right\|_{\mathcal{X}}\leq A 
\prod_{j=1}^d \Big\|f_j\Big\|_{\mathcal{Y}_j}^{\alpha_j}
\end{equation*}
for all nonnegative $f_j \in \mathcal{Y}_j$, $1 \leq j \leq d$. Then there exists a
weight function\footnote{i.e. a measurable function $w$ with $w(x) > 0$ a.e.} $w$ on $X$ such that
for every nonnegative $G\in\mathcal{X}^\prime$, there exist nonnegative measurable functions $g_j \in L^1(X, w {\rm d} \mu)$ such that 
\begin{equation}\label{factorise}
G(x) \leq \prod_{j=1}^d g_j(x)^{\alpha_j}\qquad\mbox{a.e. on $X$,}
\end{equation}
and such that for each $j$, 
\begin{equation}\label{control}
\int_X g_j(x)T_jf_j(x) {\rm d}\mu(x) \leq A\|G\|_{\mathcal{X}^\prime}\|f_j\|_{\mathcal{Y}_j}
\end{equation}
for all $f_j \in \mathcal{Y}_j$.
\end{thm}

\begin{remark}\label{weightrem}\label{mnb}
  The hypothesis that each $T_j$ saturates $X$ is very natural as pointed out in Remark \ref{bnm} above. Indeed, for the
  reasons set out there, without this hypothesis 
  we cannot expect the conclusion to hold. Needless to say, it will play an important role in the proof of Theorem~\ref{thmmain}. In particular, the weight function $w$ arises as a consequence of the saturation hypothesis. For
  its construction, see Section~\ref{tsjkr} below.  
  If $\mu(X)$ is finite and the $T_j$ strongly saturate $X$, we can take $w$ to be
  the constant function $1$, see Theorem~\ref{thmsimple} below. 
\end{remark}

\begin{remark}
  In the case $d=1$ the factorisation is trivial, and \eqref{control} is simply the usual duality relation corresponding to \eqref{lin2}.
  \end{remark}

\begin{remark}\label{smaller}
  If there exist $g_j$ satisfying \eqref{factorise} and \eqref{control}, then by making one of the $g_j$ smaller if necessary,
  we can find $g_j$ satisfying \eqref{factorise} with {\em equality} in addition to \eqref{control}.
\end{remark}

\begin{remark}
  We emphasise that the constant $A$ appearing in \eqref{control} is {\em precisely} the constant $A$ occuring
  in the hypothesis.
\end{remark}

  \begin{remark}\label{speccaserem}
    As in the case of Theorem~\ref{thmmainbaby}, the general case of Theorem~\ref{thmmain} follows from the special case in which $\mathcal{X} = L^1(X)$. Indeed, placing ourselves under the assumptions of the general case, let $G \in \mathcal{X}'$ have norm $1$, and observe that by H\"older's
    inequality we have
\begin{equation*}
\int_X\prod_{j=1}^d (T_jf_j)^{\alpha_j}G(x) {\rm d} \mu(x)\leq A 
\prod_{j=1}^d \Big\|f_j\Big\|_{\mathcal{Y}_j}^{\alpha_j}
\end{equation*}
for all $f_j \in \mathcal{Y}_j$, $1 \leq j \leq d$. This is the main hypothesis
of the special case, but with respect to the measure $G {\rm d} \mu$ instead of
${\rm d} \mu$. It is easily verified that  $G {\rm d} \mu$ is a $\sigma$-finite
measure, and that if $T_j$ saturates $X$ with respect
to ${\rm d}\mu$, it also does so likewise with respect to $G {\rm d} \mu$, and similarly
for strong saturation. We may therefore conclude from the $L^1$ case of Theorem~\ref{thmmain} that there exist nonnegative measurable $\gamma_j$ such that
$$ \prod_{j=1}^d \gamma_j(x)^{\alpha_j} \geq 1 \mbox{ a.e. } G {\rm d} \mu$$
and such that
\begin{equation*}
\int_X \gamma_j(x)T_jf_j(x) G(x) {\rm d}\mu(x) \leq A \|f_j\|_{\mathcal{Y}_j}
\end{equation*}
for all $f_j \in \mathcal{Y}_j$. Setting $g_j = G \gamma_j$, the easy observation that
$$ \prod_{j=1}^d g_j(x)^{\alpha_j} \geq G(x) \mbox{ a.e. } {\rm d} \mu$$
completes the argument. (This argument does not directly place the $g_j$ in a weighted $L^1$-space,
but this feature can in any case be recovered from inequality \eqref{control}.) However, this observation does not 
simplify the proof of Theorem~\ref{thmmain}, and we therefore establish the general case directly.
  \end{remark}
  
\begin{remark}\label{paty}
  Following on from Remark~\ref{extn} above, if the spaces $\mathcal{Y}_j$ are
  additionally supposed to be
  $p_j$-convex for some $p_j \geq 1$, we may use the classical linear Maurey--Nikisin--Stein theory
  for positive operators to upgrade conclusion \eqref{control} of Theorem~\ref{thmmain}
  to boundedness of each $T_j$ into a suitably weighted $L^{p_j}$-space. A similar remark
  applies in the context of Theorem~\ref{thmmainmaurey} below. This perspective is further
  explored in \cite{CHV}.

  \end{remark}

\medskip
\noindent
The proof of Theorem~\ref{thmmain} is highly nonconstructive and comes about 
as a result of duality methods in the theory of convex optimisation which ultimately 
rely upon a form of the minimax principle. For the details of the proof see
Sections~\ref{outline}, \ref{overv} and \ref{dett} below.
Nevertheless, in some cases, constructive factorisations
can be given, and in other cases, the existence of the factorisation raises interesting questions and links with
other areas of analysis. See Sections~\ref{factinterp}, \ref{culture}, \ref{Loomis--Whitneyrevisited},
\ref{NLLW} and \ref{BLrevisited}.

\subsection{Multilinear Maurey-type theory}\label{maureysec}
In this section we state and prove a slight generalisation of Theorem~\ref{thmmainmaurey}, using the case
$\mathcal{X} = L^1(X)$ of Theorem~\ref{thmmain}. Interestingly, the classical Maurey
theorem follows easily from Theorem~\ref{thmmain} specialised to the {\em bilinear} case $d=2$ in which one of the
normed lattices is one-dimensional. (Therefore, the case $d=1$ of what follows is {\em not} trivial, in contrast to the situation for Theorem~\ref{thmmain}.)

\begin{thm}[Multilinear Maurey-type theorem]
\label{thmmainmaureyx}
Suppose $(X, {\rm d} \mu)$ is a $\sigma$-finite measure space, $\mathcal{Y}_j$ are normed lattices, and
$T_j : \mathcal{Y}_j \to \mathcal{M}(X)$ are positive linear maps.
Suppose that each $T_j$ saturates $X$.
Let $0 < q <1$, and suppose
\begin{equation}\label{tez}
\left\|\prod_{j=1}^d (T_jf_j)^{\alpha_j}\right\|_{L^q(X)}
\leq A \prod_{j=1}^d \Big\|f_j\Big\|_{\mathcal{Y}_j}^{\alpha_j}
\end{equation}
for all nonnegative $f_j \in \mathcal{Y}_j$, $1 \leq j \leq d$. 
Then there exist nonnegative measurable functions 
$g_j$ on $X$ such that 
\begin{equation}\label{factorisebabymaureyx}
\|\prod_{j=1}^d g_j(x)^{\alpha_j}\|_{L^{q'}(X)} = 1
\end{equation}
and such that for each $j$, 
\begin{equation}\label{controlbabymaureyx}
\int_X g_j(x)T_jf_j(x) {\rm d}\mu(x) \leq A\|f_j\|_{\mathcal{Y}_j}
\end{equation}
for all $f_j \in \mathcal{Y}_j$.
\end{thm}

It is an easy exercise using H\"older's inequality to show that if there exist $g_j$ such that \eqref{factorisebabymaureyx}
and \eqref{controlbabymaureyx} hold, then \eqref{tez} also holds. As in the case of Theorem~\ref{thmmainmaurey},
Theorem~\ref{thmmainmaureyx} admits an interpretation as a statement about factorisation of operators, see Section~\ref{mmf}.

\begin{proof}
  The main hypothesis is that
  $$ \int_X \prod_{j=1}^d (T_jf_j)^{\alpha_j q} {\rm d} \mu
  \leq A^q \prod_{j=1}^d \Big\|f_j\Big\|_{\mathcal{Y}_j}^{\alpha_j q}$$
  for all $f_j \in \mathcal{Y}_j$.
Let $\beta_j = \alpha_j q$ for $1 \leq j \leq d$ and let $\beta_{d+1} = 1 - \sum_{j=1}^d \beta_j = 1 - q > 0$.
Let $Y_{d+1} = \{0\}$ and let $\mathcal{Y}_{d+1}$ be the trivial normed lattice $\mathbb{R}$ defined on the
singleton measure space $\{0\}$. Let $T_{d+1} : \mathcal{Y}_{d+1} \to \mathcal{M}(X)$ be the linear map
$\lambda \mapsto \lambda {\bf 1}$ where ${\bf 1}$ denotes the constant function taking the value $1$ on $X$.
Then we have
$$ \int_X \prod_{j=1}^{d+1} (T_jf_j)^{\beta_j} {\rm d} \mu
\leq A^q \prod_{j=1}^{d+1} \Big\|f_j\Big\|_{\mathcal{Y}_j}^{\beta_j}$$
for all $f_j \in \mathcal{Y}_j$. So by Theorem~\ref{thmmain} in the case $\mathcal{X} = L^1(X)$ (and with
$d+1$ in place of $d$, see also Remark~\ref{smaller} above), we conclude that there exist measurable functions $G_1, \dots, G_{d+1}$ such that
\begin{equation}\label{factorisez}
\prod_{j=1}^{d+1} G_j(x)^{\beta_j} = 1 \qquad\mbox{a.e. on $X$,}
\end{equation}
and such that for each $1 \leq j \leq d+1$, 
\begin{equation}\label{controlz}
\int_X G_j(x)T_jf_j(x) {\rm d}\mu(x) \leq A^q \|f_j\|_{\mathcal{Y}_j}
\end{equation}
for all $f_j \in \mathcal{Y}_j$.

\medskip
\noindent
For $1 \leq j \leq d$, set $g_j(x) = A^{1-q}G_j(x)$; then \eqref{controlz} immediately gives \eqref{controlbabymaureyx} for
$1 \leq j \leq d$. By \eqref{factorisez} we have
$$ \prod_{j=1}^d g_j(x)^{\alpha_j} = A^{1-q}G_{d+1}(x)^{(q-1)/q}\qquad\mbox{a.e. on $X$,} $$
while \eqref{controlz} for $j = d+1$ gives
$$ \int_X G_{d+1}(x) {\rm d} \mu(x) \leq A^q.$$
Combining these last two relations gives \eqref{factorisebabymaureyx} as desired.
\end{proof}

\section{Discrete case: a convex optimisation problem}\label{outline}
\subsection{Basic set-up}
The idea behind the proof of Theorem \ref{thmmain} is to view problem \eqref{factorise} and \eqref{control} as a convex
optimisation problem. That is, we replace the number $A$ in \eqref{control}
by a variable $K$ and seek to minimise over $K$. To illustrate how this works, we first prove the theorem in a 
model case when $X$ and $Y_j$ are finite sets endowed with counting measure and $\mathcal{Y}_j=L^1(Y_j)$ 
for $j=1,\dots,d$. One reason for doing this case first is that there are no measure-theoretic or 
functional-analytical difficulties to be dealt with in this setting, and indeed
$\mathcal{X}^\prime = \mathcal{X}^\ast$ is simply the class of all functions defined on $X$ with the norm dual to that of
$\mathcal{X}$. It therefore
allows us to emphasise the nature of the problem as one concerning convex optimisation.

\medskip
\noindent
The minimisation problem we propose to examine now reads as follows. Fix $G: X \to [0, \infty)$ and consider
\begin{equation}
\label{finiteminim}
\begin{aligned}
\gamma =&\inf_{K,g_j} K\\
\text{such that}\quad&G(x)\leq\prod_{j=1}^d g_j(x)^{\alpha_j}\quad\text{for all $x\in X$ and}\\
&\max_{y_j\in Y_j}T_j^\ast g_j(y_j)\leq K\|G\|_{\mathcal{X}^\ast}\quad\text{for all $j=1,\dots,d$.}
\end{aligned}
\end{equation}
We note that this is a convex optimisation problem since we are minimising a convex, in fact linear,
function on the convex domain consisting of the $(d+1)$-tuples $(K, g_j)$ satisfying the constraints in
\eqref{finiteminim}. The convexity of this domain follows from the fact that the second set of inequalities 
is linear in the arguments $K$ and $g_j$, and the operation of taking the geometric mean on the right hand side of the 
first set of inequalities is a concave function. We note that the set of $(K, g_j)$ satisfying the constraints in 
\eqref{finiteminim} is not empty and that we can in fact find $(K, g_j)$ satisfying these constraints 
with strict inequality by taking each $g_j$ to be $2G + 1$ and letting $K$ be sufficiently large. Thus problem 
\eqref{finiteminim} satisfies what is known as {\em Slater's condition}. (We do not give full details here as the
discussion will eventually be subsumed into that of the next section.) In particular we certainly have 
$\gamma < + \infty$.

\medskip
We therefore follow a standard approach to convex optimisation problems, see for example \cite{MR2061575}.
We introduce Lagrange multipliers $\psi$ and $h_j$, where $\psi: X \to \mathbb{R}_+$ (for the first set of constraints), 
and $h_j: Y_j \to \mathbb{R}_+$ (for the second set). Note that we are only interested case where these 
functions take nonnegative values since each of the constraints is an inequality constraint. 
We then introduce the Lagrangian functional
\begin{equation}
\label{finitelagrange}
L=K+\sum_{x\in X}\psi(x)\left(G(x)-\prod_{j=1}^d g_j(x)^{\alpha_j}\right)
+\sum_{j=1}^d\sum_{y_j\in Y_j}h_j(y_j)(T_j^\ast g_j(y_j)-K\|G\|_{\mathcal{X}^\ast}).
\end{equation}
We emphasise that this function and the corresponding one defined in the proof of the general
case are the only functions which we allow to take negative values.

\medskip
For nonnegative $K$, $g_j$, $\psi$, and $h_j$
we now consider the two problems\footnote{The subscript $\mathcal{L}$ in $\gamma_\mathcal{L}$ indicates that we are looking 
at the Lagrangian version of the problem as opposed to the original version which has $\gamma$ without a subscript.}
\[
\gamma_\mathcal{L} =\inf_{K,g_j}\sup_{\psi,h_j}L
\qquad\text{and}\qquad
\eta=\sup_{\psi,h_j}\inf_{K,g_j}L
\]
called the primal problem and the dual problem respectively. We shall show that (i) the problem
for $\gamma$ is identical to the problem for $\gamma_\mathcal{L}$, (ii)
$\eta\leq A$ where $A$ is any number such that inequality \eqref{mainineq} holds, and (iii)
$\eta=\gamma_\mathcal{L}$ (it is obvious that $\eta \leq \gamma_\mathcal{L}$).
Finally, we show that the infimum in the definition of $\gamma$ is attained, and
this will complete the proof of the theorem in the special case. 

\subsection{Identification of the problems for $\gamma$ and $\gamma_\mathcal{L}$.}
We begin by studying $\gamma_\mathcal{L}$. Fix $K \geq 0$ and $g_j$ and consider $\sup_{\psi,h_j}L$.
Suppose that any of the conditions in \eqref{finiteminim} is not satisfied at some point.
Then take the relevant function $\psi$ or $h_j$ for some $j$ to have value $t>0$
at a point where an inequality fails and let all of the functions be zero everywhere else.
Then let $t\to\infty$ and notice that $\sup_{\psi,h_j}L$ goes to $+\infty$ since $t$ is multiplied
by a positive number. So if $\sup_{\psi,h_j}L<+\infty$ we must have that the conditions of \eqref{finiteminim}
are satisfied. Conversely, if these conditions are satisfied then all factors multiplying $\psi(x)$
and $h_j(y_j)$ for any $x$ and $y_j$ are non-positive so the supremum is attained by taking them
all to equal $0$. So, for each fixed $(K, g_j)$, we have that $\sup_{\psi,h_j}L < + \infty $ if and only if the 
conditions in \eqref{finiteminim} hold, in which case, $\sup_{\psi,h_j}L = K$.  Thus we see that 
the problem for $\gamma_\mathcal{L}$ is identical to problem \eqref{finiteminim}, yielding 
$\gamma_\mathcal{L}=\gamma$. Moreover the infimum in the definition of $\gamma$ is attained if and only if
the infimum in the definition of $\gamma_\mathcal{L}$ is attained.%\footnote{Am I really convinced about this?}

\subsection{Proof that $\eta\leq A$.}
We rearrange $L$ as follows:
\begin{equation}
\label{finitelagrangere}
\begin{aligned}
L=&
\sum_{x\in X}\psi(x)G(x)+
K\left(1-
\|G\|_{\mathcal{X}^\ast}\sum_{j=1}^d\sum_{y_j\in Y_j}h_j(y_j)\right)\\
&+
\sum_{x\in X}\left(
\sum_{j=1}^dg_j(x)T_jh_j(x)
-\prod_{j=1}^d g_j(x)^{\alpha_j}\psi(x)\right)
\end{aligned}
\end{equation}
Let us fix $\psi$ and $h_j$ and consider $\inf_{K,g_j}L$.
First of all, note that $\inf_{K,g_j}L=-\infty$ unless
\begin{equation}
\label{condKfinite}
\|G\|_{\mathcal{X}^\ast}\sum_{j=1}^d\sum_{y_j\in Y_j}h_j(y_j)\leq 1
\end{equation}
since if this inequality fails then the term multiplying $K$ in $L$ is
negative and so by taking $g_j=0$ and letting $K$ go to infinity we get that
$\inf_{K,g_j}L=-\infty$.
Also note that $\inf_{K,g_j}L=-\infty$ unless
\begin{equation}
\label{condbetafinite}
\psi(x)\leq\prod_{j=1}^d (\alpha_j^{-1}T_jh_j(x))^{\alpha_j}
\end{equation}
for all $x\in X$. Seeing this is a matter of choosing $g_j(x)$ to balance
the arithmetic-geometric mean inequality. Specifically, suppose that this condition \eqref{condbetafinite} fails at
a point $x_0$. Then we let $K=0$ and $g_j(x)=0$ for all $x\neq x_0$ and all $j=1,\dots,d$.
There are now two cases to consider. Firstly, if there exists an index $j_0$
such that $T_{j_0}h_{j_0}(x_0)=0$ then we take $g_j(x_0)=1$ for all $j\neq j_0$ and $g_{j_0}(x_0)=t>1$.
Then
\[L= \sum_{x\in X}\psi(x)G(x) + \sum_{\substack{j=1\\j\neq j_0}}^d T_{j}h_{j}(x_0)-t^{\alpha_{j_0}}\psi(x_0).\]
Since $\alpha_{j_0}>0$ we can let $t$ go to infinity and see that
$\inf_{K,g_j}L=-\infty$.
In the other case we have that $T_{j}h_{j}(x_0)>0$ for all $j=1,\dots,d$. Then we let
\[g_{j}(x_0)=t \alpha_j ((T_jh_j)(x_0))^{-1}\prod_{j'=1}^d (\alpha_{j'}^{-1}T_{j'}h_{j'})(x_0)^{\alpha_{j'}}\]
and note that
\[L= \sum_{x\in X}\psi(x)G(x)+ t\left(\prod_{j=1}^d (\alpha_j^{-1}T_jh_j(x_0))^{\alpha_j} - \psi(x_0)\right).\]
So by the assumption of the failure of \eqref{condbetafinite} at $x_0$ we see that letting $t \to \infty$ yields 
$\inf_{K,g_j}L=-\infty$.

\medskip
Conversely, if conditions \eqref{condKfinite} and \eqref{condbetafinite} hold then
the factor multiplying $K$ is nonnegative and an application of the arithmetic-geometric mean inequality 
gives that for any choice of $g_j$ then for each $x\in X$ the term in the second bracket of \eqref{finitelagrangere} 
is nonnegative, so we attain $\inf_{K,g_j}L$ by letting $K=0$ and $g_j(x)=0$ for all $x\in X$ and $j=1,\dots,d$.
Hence, for each fixed $(\psi, h_j)$, $\inf_{K,g_j}L > - \infty$ if and only if $\psi$ and $h_j$ satisfy conditions 
\eqref{condKfinite} and \eqref{condbetafinite}, in which case $\inf_{K,g_j}L = \sum_{x \in X} \psi(x) G(x)$. 
Noting that there always exist $\psi$ and $h_j$ satisfying conditions \eqref{condKfinite} and \eqref{condbetafinite}, we see that $\eta$ is the solution to
\begin{equation}
\label{finitemaxim}
\begin{aligned}
\eta=&\sup_{\psi,h_j} \sum_{x\in X}\psi(x) G(x)\\
\text{such that}\quad&\psi(x)\leq\prod_{j=1}^d (\alpha_j^{-1}T_jh_j(x))^{\alpha_j}\quad\text{for all $x\in X$ and}\\
&\|G\|_{\mathcal{X}^\ast}\sum_{j=1}^d\sum_{y_j\in Y_j}h_j(y_j)\leq 1.
\end{aligned}
\end{equation}

For any $\psi$ and $h_j$ satisfying the conditions in \eqref{finitemaxim} we can calculate
\begin{align*}
\sum_{x\in X}\psi(x) G(x)
&\leq
\sum_{x\in X}\prod_{j=1}^d (\alpha_j^{-1}T_jh_j(x))^{\alpha_j} G(x)
\leq
\left\|\prod_{j=1}^d (\alpha_j^{-1}T_jh_j)^{\alpha_j}\right\|_{\mathcal{X}} \Big\|G\Big\|_{\mathcal{X}^\ast}\\
&\leq
A\prod_{j=1}^d \|\alpha_j^{-1}h_j\|_{\mathcal{Y_j}}^{\alpha_j}
\|G\|_{\mathcal{X}^\ast}
\leq
A\sum_{j=1}^d \|h_j\|_1
\|G\|_{\mathcal{X}^\ast}\leq A
\end{align*}
where the inequalities follow in order from the first condition of \eqref{finitemaxim},
the definition of the norm on $\mathcal{X}^\ast$, the inequality \eqref{mainineq},
the arithmetic-geometric mean inequality, and the second condition of \eqref{finitemaxim}.
Taking the supremum now yields $\eta\leq A$.

\subsection{Proof that $\gamma_\mathcal{L}=\eta$ and existence of minimisers.} 
This is a minimax argument. As we have noted above, it is immediate that $\eta\leq\gamma_\mathcal{L}$ and this is referred to as 
weak duality. The other direction, giving $\gamma_\mathcal{L}=\eta$, is called strong duality
and does not hold in general. However there are various conditions which guarantee strong duality,
such as Slater's condition which is the condition that the original problem \eqref{finiteminim}
is convex and there exists a point satisfying all of the constraints with
strict inequality. See \cite{MR2061575}, p.226. We have noted above that Slater's 
condition holds in our setting. Moreover, Slater's condition guarantees 
the existence of a maximiser for the dual problem. However, we need optimisers for the primal problem.
If for all $x\in X$ we have $T_j {\bf 1}(x)>0$ -- which is simply the saturation hypothesis in our present case -- then the set of $g_j$'s which satisfy the constraints 
of \eqref{finiteminim} with $K=2A$ will be compact, and therefore a minimiser will exist.

\section{General case: overview of the proof}\label{overv}
Let us now turn to the argument for Theorem~\ref{thmmain} in the general case. It will entail substantial 
measure-theoretic and functional-analytic considerations not present in the case when $X$ and $Y_j$ are finite sets. While it is an attractive idea to try to establish Theorem~\ref{thmmain} by approximating the general case by the
discrete case, this does not seem a feasible route, even when $\mathcal{X}$ and $\mathcal{Y}_j$ are $L^q$ and $L^{p_j}$ spaces respectively,
and a direct approach is therefore required. The bulk of 
the proof of Theorem~\ref{thmmain} will be devoted to establishing a special case in which $X$ is a finite 
measure space\footnote{To be clear, a measure space $(X, {\rm d} \mu)$ with $\mu (X) < \infty$, not a
finite set $X$ with counting measure.}
and where we impose strong saturation on the $T_j$ instead of saturation. This leads to
the crucial conclusion that we can take the factors $g_j$ to lie in $L^1( X, {\rm d} \mu)$. 
The result reads as follows:

\begin{thm}
\label{thmsimple}
Suppose $X$ is a finite measure space, $\mathcal{X}$ is a K\"othe space of functions defined on $X$, $\mathcal{Y}_j$ are normed lattices, and that the linear operators
$T_j: \mathcal{Y}_j \to \mathcal{M}(X)$ are positive. 
Suppose that each $T_j$ strongly saturates $X$. Suppose that
\begin{equation}\label{again}
\left\|\prod_{j=1}^d (T_jf_j)^{\alpha_j}\right\|_{\mathcal{X}}\leq A 
\prod_{j=1}^d \Big\|f_j\Big\|_{\mathcal{Y}_j}^{\alpha_j}
\end{equation}
holds for all 
nonnegative $f_j \in \mathcal{Y}_j$, $1 \leq j \leq d$.\footnote{Notice that a hypothesis of strong saturation is unrealistic in the presence of inequality~\eqref{again} unless $X$ has finite measure.} Then for every nonnegative
$G\in\mathcal{X}^\prime$ there exist nonnegative functions $g_j\in L^1(X, {\rm d}\mu)$ such that 
\begin{equation}\label{factorise1}
G(x) \leq \prod_{j=1}^d g_j(x)^{\alpha_j}\qquad\mbox{a.e. on $X$,}
\end{equation}
and such that for each $j$, 
\begin{equation}\label{control1a}
\int_X g_j(x)T_jf_j(x) {\rm d}\mu(x) \leq A\|G\|_{\mathcal{X}^\prime}\|f_j\|_{\mathcal{Y}_j}
\end{equation}
for all 
$f_j \in \mathcal{Y}_j$.
\end{thm}

\medskip
\noindent
In the proof of this theorem we introduce an extended real-valued Lagrangian function $L(\Phi,\Psi)$ (where $\Phi$
corresponds to the variables $(K, g_j)$ and $\Psi$ corresponds to the variables $(\psi, h_j)$ of the discrete model
case discussed above). See Section~\ref{turn} below for precise details of the definition of $L$.
As in the model case, we relate $\sup_\Psi\inf_\Phi L$ (which we had previously called $\eta$) to 
problem \eqref{factorise1} and \eqref{control1a} and $\inf_\Phi\sup_\Psi L$ (which we had previously called $\gamma$) to inequality \eqref{again}.
We then need to show that
$$ \min_\Phi\sup_\Psi L(\Phi, \Psi) = \sup_\Psi\inf_\Phi L(\Phi, \Psi),$$
and for this we need to use the Lopsided Minimax Theorem (which can be found as Theorem 7
from Chapter 6.2 of \cite{MR749753}):
\begin{thm}\label{minimax}
    Suppose $C$ and $D$ are convex subsets of vector spaces and that
    $C$ is endowed with a topology for which the vector space operations are continuous. 
    Further, suppose that $L : C \times D \to \mathbb{R}$ satisfies
    \begin{enumerate}
        \item [(i)]$\Phi\mapsto L(\Phi,\Psi)$ is convex for all $\Psi\in D$;
        \item [(ii)]$\Psi\mapsto L(\Phi,\Psi)$ is concave for all $\Phi\in C$;
        \item [(iii)]$\Phi\mapsto L(\Phi,\Psi)$ is lower semicontinuous for all $\Psi\in D$; and
        \item [(iv)]there exists a $\Psi_0\in D$ such that the sublevel sets $\{\Phi \in C \, : \, L(\Phi,\Psi_0)\leq\lambda\}$
          are compact for all sufficiently large $\lambda\in\mathbb{R}$.
    \end{enumerate}
    Then
    \begin{equation}\label{minimaxconcl}
\min_{\Phi\in C}\sup_{\Psi\in D} L(\Phi, \Psi) = \sup_{\Psi\in D}\inf_{\Phi\in C} L(\Phi, \Psi).
  \end{equation}
\end{thm}

\begin{remark}
  The existence of the minimum on the left-hand side of \eqref{minimaxconcl} is part of the conclusion:
  there exists a $\bar{\Phi}\in C$ such that $ \sup_{\Psi\in D} L(\bar{\Phi},\Psi) = \inf_{\Phi\in C}\sup_{\Psi\in D} L(\Phi, \Psi)$.
  Once we know that $ \inf \sup = \sup \inf$ this is easy because (iii) tells us that the map
  $\Phi \mapsto \sup_{\Psi\in D} L(\Phi, \Psi)$ is lower semicontinuous, and (iv) then tells us that the sublevel sets
  $$ \{ \Phi \in C \, : \, \sup_{\Psi\in D} L(\Phi, \Psi) \leq \lambda\} \subseteq
  \{ \Phi \in C \, : \, L(\Phi, \Psi_0) \leq \lambda\} $$
  are closed and compact, and hence $\Phi \mapsto \sup_{\Psi\in D} L(\Phi, \Psi)$ achieves its minimum on any such set.
  The fact that $\sup \inf \leq  \inf \sup$ is trivial, so the main content of the theorem is that  $ \inf \sup \leq \sup \inf$.
\end{remark}

\begin{remark}
  There is nothing to stop both sides of \eqref{minimaxconcl} from being $+\infty$. Indeed,
  a nontrivial conclusion of the theorem is that if the right-hand side is finite, so is the left-hand side.
  \end{remark}

\begin{remark}
  Traditional versions of minimax theorems assume that $C$ itself is compact, rather than compactness of certain sublevel
  sets as condition (iv). However, in our case, we cannot, for
  the reasons set out below, expect $C$ to be compact. {\em It is a remarkable feature of our analysis that the saturation hypothesis we must impose corresponds {\bf precisely} to condition (iv) of the minimax theorem.}
\end{remark}

\begin{remark}\label{urgh}
  The observant reader will have noticed that we have indicated our intention to introduce an extended real-valued Lagrangian $L$, but the minimax
  theorem applies only to real-valued Lagrangians. This mismatch necessitates a small detour which we wish to suppress here.\footnote{B.~Ricceri has recently informed us (private communication) that Theorem~\ref{minimax} continues to hold when the Lagrangian is permitted to take the value $+\infty$. The detour takes no longer than establishing this more general minimax statement.}
  For details see Section~\ref{detourdetails} below.
  \end{remark}

\bigskip
\noindent
It is a somewhat delicate matter to choose the vector space
where we will locate the variables $\Phi$ featuring in the Lagrangian
which we will use. Corresponding to the variables $g_j$ occuring in the discrete
model case of Section~\ref{outline}, we will now have variables $S_j$, which we would like
to take to be elements of $L^1(X)_+$. (It is the weighted geometric mean of a particular collection of these  
which will ultimately furnish the desired factorisation.) However, it turns out to be helpful
to instead allow, in the first instance, the $S_j$ be elements of the larger space $L^\infty(X)^\ast_+$, that is, the positive 
cone of the dual of $L^\infty(X)$.\footnote{Had we instead opted to work from the outset with $S_j \in 
L^1(X)_+$, we would have been forced to place unnatural topological conditions on $X$ in order to 
identify $L^1(X)_+$ with a subspace of a dual space, and in any case we would have to work in the larger space 
of finite regular Borel measures on $X$ in order to exploit weak-star compactness.} Thus we consider the vector space
$\mathbb{R} \times L^\infty(X)^\ast \times \dots \times L^\infty(X)^\ast$ and take $C$ to be a suitable subset of
the positive cone in this space. Ideally we would like to take $C$
to be a {\em norm-bounded} convex subset and then use the Banach--Alaoglu theorem to assert compactness of $C$;
but since we are not expecting any quantitative $L^1$ bounds on the functions $S_j$ appearing in the factorisation,
there is no natural norm-bounded set with which to work. Instead, we take $C$ to be the whole positive cone
$\mathbb{R}_+ \times L^\infty(X)^\ast_+ \times \dots \times L^\infty(X)^\ast_+$, endowed with the weak-star topology.
The price for this is the need to verify hypothesis (iv) of Theorem~\ref{minimax}. Fortunately this turns out
to be not so difficult, and in fact is rather natural in our setting. Carrying out this process will yield some
distinguished members of $L^\infty(X)^\ast_+$. However, working with the dual of $L^\infty(X)$ presents its own difficulties since
some elements of $L^\infty(X)^\ast_+$ are quite exotic. Fortunately the theory of finitely additive measures
comes to the rescue, and we will be able to show that elements satisfying
the properties we require can be in fact be found in the smaller space $L^1(X)_+$. See Section~\ref{turn}
below for more details.

\medskip
\noindent
To set the scene for this, we recall three results of Yosida and Hewitt which can be found 
in \cite{MR0045194}. The setting for each of these results is a $\sigma$-finite measure space $(X, {\rm d \mu})$.

\begin{thm}
\label{thmdual}
There is an isometric isomorphism between the space of finitely additive measures on $X$ of finite total variation
which are $\mu$-absolutely continuous\footnote{This means that the finitely additive measure $\tau$ satisfies $\tau(E) = 0$ 
whenever $\mu(E) = 0$.}
and the space of bounded linear functionals on $L^\infty(X, {\rm d}\mu)$.
For a finitely additive measure $\tau$ with these properties the corresponding element of $L^\infty(X, {\rm d}\mu)^\ast$
is given by $\tau(\psi)=\int_X \psi \, {\rm d} \tau $ (where the integral is the so-called Radon integral).
Furthermore $L^1(X,{\rm d}\mu)$ embeds isometrically into $L^\infty(X, {\rm d}\mu)^\ast$ in such a way that the
application of $g\in L^1(X, {\rm d}\mu)$ to an element of $\psi\in L^\infty(X, {\rm d}\mu)$ is given by
$\int_X \psi g {\, \rm d}\mu $ where the integral is now the Lebesgue integral.
\end{thm}

\begin{thm}
\label{thmYH1}
Any element $S\in L^\infty(X, {\rm d}\mu)^\ast$ can be written uniquely as $S=S_{\mathrm{ca}}+S_{\mathrm{pfa}}$
where $S_{\mathrm{ca}}$ is countably additive (and hence is given by integration against a function in 
$L^1(X, {\rm d} \mu)$) and $S_{\mathrm{pfa}}$ is purely finitely additive.
Furthermore $S \geq 0$ if and only if $S_{\mathrm{ca}}\geq 0$ and $S_{\mathrm{pfa}}\geq 0$.
\end{thm}

\medskip
We need not concern ourselves here with the definition of purely finitely additive measures, but, in order to be
able to use these results, we do need a useful characterisation of which measures are purely finitely additive.

\begin{thm}
\label{thmYH2}
A nonnegative finitely additive measure $\tau$ which is $\mu$-absolutely continuous is purely finitely additive
if and only if for every nonnegative countably additive measure $\sigma$ which is $\mu$-absolutely continuous,
every measurable set $E$ and every pair of positive numbers $\delta_1$ and $\delta_2$, there is a measurable
subset $E'$ of $E$ such that $\sigma(E')<\delta_1$ and $\tau(E\setminus E')<\delta_2$.
\end{thm}

\medskip
\noindent
We wish to remark that analysis related to the dual space of $L^\infty$ has also been employed in a
number of other contexts recently. See for example \cite{BGM}, \cite{SS}, \cite{Toland1} and, in the
financial mathematics literature, \cite{MMS}.

\section{General case: details of the proof}\label{dett}
\subsection{Preliminaries}\label{Preliminaries}
We shall first need two lemmas which will be useful for the proof of  Theorem~\ref{thmsimple} and also that of 
Theorem~\ref{thmmain} itself. The first one is they key technical tool which, in the context of Theorem~\ref{thmsimple}, will allow us do induce existence of
suitable integrable functions from existence of corresponding members of the dual of $L^\infty$. We shall continue
to assume that $\alpha_j > 0$ and that $\sum_{j=1}^d \alpha_j = 1$.

\begin{lem}\label{abscont} 
Let $(X, {\rm d} \mu)$ be a $\sigma$-finite measure space and suppose that $S_j\in L^\infty(X)_+^\ast$. Suppose
that $G$ is a measurable function such that
\begin{equation}\label{opl}
    \int_X G(x)\prod_{j=1}^d \beta_j^{\alpha_j}(x) \, {\rm d}\mu(x)\leq\sum_{j=1}^d \alpha_j S_j(\beta_j)
    \qquad\text{for all simple functions $\beta_j$ on $X$.}
\end{equation}
If $({S}_{j\mathrm{rn}})$ denotes the Radon--Nikodym derivative with respect to $\mu$ of the
component of $S_j$ which is countably additive, then
\begin{equation}\label{lpo}
G(x)\leq \prod_{j=1}^d {S}_{j\mathrm{rn}}(x)^{\alpha_j}\quad\text{a.e. on $X$. }
\end{equation}
Conversely, if $G$ is such that \eqref{lpo} holds, then \eqref{opl} holds.
\end{lem}

\begin{remark}
This result extends the special case $d=1$ which is implicit in Theorem \ref{thmYH1}.
\end{remark}

\begin{proof}
The converse statement follows immediately from the arithmetic-geometric mean inequality, so we turn to the forward assertion.
Suppose not. Then there exists a set $E_0$ with $\mu(E_0)>0$ such that inequality \eqref{lpo} fails on $E_0$ and
we can find an $\varepsilon>0$ and $E_1\subseteq E_0$ with $\mu(E_1)>0$ such that
\begin{displaymath}
    G(x)-\varepsilon> \prod_{j=1}^d S_{j\mathrm{rn}}^{\alpha_j}(x)
\end{displaymath}
for all $x\in E_1$. 
Now take $\beta_j$ to be simple functions supported on $E_1$. Then we get
\begin{equation}
    \label{bigexpr}
\begin{gathered}
    \int_{E_1} \prod_{j=1}^d S_{j\mathrm{rn}}^{\alpha_j}(x)\prod_{j=1}^d \beta_j^{\alpha_j}(x) \, {\rm d}\mu(x)
    + \varepsilon \int_{E_1} \prod_{j=1}^d \beta_j^{\alpha_j}(x) \, {\rm d} \mu(x)\\
    \leq
    \int_{E_1} G(x)\prod_{j=1}^d \beta_j^{\alpha_j}(x) \, {\rm d}\mu(x)
    \leq
    \sum_{j=1}^d \alpha_j S_j(\beta_j)\\
    =
    \sum_{j=1}^d \alpha_j
    \int_{E_1} S_{j\mathrm{rn}}(x) \beta_j(x) \, {\rm d}\mu(x)
    +
    \sum_{j=1}^d \alpha_j
    \int_{E_1} \beta_j(x) \, {\rm d}\tau_{j\mathrm{pfa}}
\end{gathered}
\end{equation}
where $\tau_{j\mathrm{pfa}}$ is the purely finitely additive measure associated to the purely finitely additive component 
$S_{j \mathrm{pfa}}$ of $S_j$.

\medskip
We can find a subset $E_2 \subseteq E_1$ with $\mu(E_2)>0$ and a $C>0$ such that $S_{j\mathrm{rn}}(x)\leq C$
for all $x\in E_2$ and all $j$.
There are now two cases to consider. 

\medskip
First, assume that there exists a subset $E_3 \subseteq E_2$ such that
$\mu(E_3)>0$ and an index $j_0$ such that $S_{j_0\mathrm{rn}}(x)=0$ for all $x\in E_3$.
Now let $\delta$ be small and positive (to be specified later) and $E_4$ a subset of $E_3$ with $0< \mu(E_4) < \infty$
(also to be specified later\footnote{We are using $\sigma$-finiteness of $\mu$ to ensure that we can find such an
  $E_4$ with $\mu(E_4)$ finite.}), 
and take $\beta_j=\delta\chi_{E_4}$ for $j\neq j_0$ and $\beta_{j_0}=\delta^{1-\alpha^{-1}_{j_0}}\chi_{E_4}$.
This implies that $\prod_{j=1}^d \beta_j^{\alpha_j}(x)=1$ for all $x\in E_4$ and
so the top line of~\eqref{bigexpr} equals $0+\varepsilon\mu(E_4)$.
The first term on the bottom line of~\eqref{bigexpr} can be bounded by $C\delta\mu(E_4)$ since
there is no contribution from the term with index $j_0$. 
The second term we can bound by
$\delta^{1-\alpha^{-1}_{j_0}} \left(\sum_j\tau_{j\mathrm{pfa}}\right)(E_4)$.
We have not chosen $E_4$ precisely yet. To do this we use Theorem \ref{thmYH2} above.

\medskip
Indeed, applying Theorem \ref{thmYH2} with $\tau := \sum_j\tau_{j\mathrm{pfa}}$, $\sigma := \mu$, $E := E_3$
and $\delta_1 := \mu(E_3)/2$ gives that for all $\delta_2 > 0$ there is an $E_4 \subseteq E_3$ such that
$\mu(E_3 \setminus E_4) < \mu(E_3)/2$ and $\tau(E_4) < \delta_2$. So ~\eqref{bigexpr} implies
$$ \varepsilon \mu(E_4) \leq C\delta\mu(E_4) + \delta^{1-\alpha^{-1}_{j_0}}\tau(E_4) 
\leq C\delta\mu(E_4) + \delta^{1-\alpha^{-1}_{j_0}} \delta_2.$$
Now choose $\delta_2 = \varepsilon \mu(E_3)/4 \delta^{1-\alpha^{-1}_{j_0}}$, so that for some $E_4 \subseteq E_3$ we have 
$$ \varepsilon \mu(E_4) \leq C\delta\mu(E_4) + \varepsilon \mu(E_3)/4 \leq C\delta\mu(E_4) + \varepsilon \mu(E_4)/2.$$
Finally, choosing $\delta < \varepsilon/(2C)$ yields a contradiction, since by construction $\mu(E_4) > 0$.

\medskip
In the other case, we have that $S_{j\mathrm{rn}}(x)>0$ for a.e. $x\in E_2$ and all $j$. Then we can find a subset
$E_3 \subseteq E_2$ with $0<\mu(E_3)<\infty$ and a number $c>0$ such that $S_{j\mathrm{rn}}(x)\geq c$ for all
$x\in E_3$ and all $j$. We define $u_{j}$ on this set as
\begin{displaymath}
    u_{j}(x)=S_{j\mathrm{rn}}(x)^{-1}\prod_{k=1}^d S_{k\mathrm{rn}}^{\alpha_k}(x)
\end{displaymath}
and note that $u_{j}(x)\leq c^{-1}C$ and that $\prod_j u_{j}^{\alpha_j}(x)=1$.
Since these functions are bounded then if we are given $\delta>0$
we can find simple functions $\tilde{\beta}_{j}$ such that
$u_j(x)-\delta\leq \tilde{\beta}_j(x)\leq u_j(x)$ for all $x\in E_3$.
We may assume that $\tilde{\beta}_j(x)\geq c C^{-1}$ for all $x\in E_3$.
Let us take $\beta_j=\tilde{\beta}_j\chi_{E_4}$ where $E_4$ is a subset of $E_3$ to be chosen.
Then for $x\in E_4$ we have that
\begin{align*}
    0\leq&
    \prod_{j=1}^du_j^{\alpha_j}(x)-
    \prod_{j=1}^d\beta_j^{\alpha_j}(x)
    =
    \sum_{k=1}^d \left(\prod_{j=1}^{k} u_j^{\alpha_j}(x)
    \left(u_{k}^{\alpha_{k}}(x)-\beta_{k}^{\alpha_{k}}(x)\right)
    \prod_{j=k+1}^d\beta_{j}^{\alpha_j}(x)\right)\\
    =&
    \sum_{k=1}^d \left(\prod_{j=1}^{k} u_j^{\alpha_j}(x)
    \alpha_k \frac{\xi_{k}^{\alpha_k}(x)}{\xi_{k}(x)}
    \left(u_{k}(x)-\beta_{k}(x)\right)
    \prod_{j=k+1}^d\beta_{j}^{\alpha_j}(x)\right)
    \leq d C^2 c^{-2}\delta.
\end{align*}
Here $\xi_{j_0}(x)$ lies between $\beta_{j_0}(x)$ and $u_{j_0}(x)$ and we have used that
$c/C \leq \beta_j(x),\xi_j(x),u_j(x)\leq C/c $.
Now we can estimate the first term on the top line of \eqref{bigexpr} from below by
\begin{displaymath}
    \int_{E_4}\prod_{j=1}^d S_{j\mathrm{rn}}^{\alpha_j}(x) \left(\prod_{j=1}^d u_j^{\alpha_j}(x)- d C^2c^{-2}\delta \right)  {\rm d}\mu(x)
    \geq
    \int_{E_4}\prod_{j=1}^d S_{j\mathrm{rn}}^{\alpha_j}(x) \, {\rm d}\mu(x) - \mu(E_4) d C^3c^{-2}\delta,
\end{displaymath}
and the second term on the top line of \eqref{bigexpr} we can estimate from below by
$\varepsilon\mu(E_4)(1-dC^2c^{-2}\delta)$ since $\prod_j u_{j}^{\alpha_j}(x)=1$ on $E_4$.

\medskip
The first term on the bottom line of \eqref{bigexpr} we can estimate from above using $\beta_j \leq u_j$ on $E_4$ and the definition of $u_j$
by
\begin{displaymath}
    \int_{E_4}\prod_{j=1}^d S_{j\mathrm{rn}}^{\alpha_j}(x) \, {\rm d}\mu(x).
\end{displaymath}
The second term we can bound by
$Cc^{-1}\left(\sum_j\tau_{j\mathrm{fa}}\right)(E_4)$.

Collecting this we have that
\begin{gather*}
    \int_{E_4}\prod_{j=1}^d S_{j\mathrm{rn}}^{\alpha_j}(x) \, {\rm d} \mu(x)-\mu(E_4)dC^3c^{-2}\delta
    +\varepsilon\mu(E_4)(1-d C^2c^{-2}\delta)\\\leq
    \int_{E_4}\prod_{j=1}^d S_{j\mathrm{rn}}^{\alpha_j}(x) \, {\rm d} \mu(x)
    +Cc^{-1}\left(\sum_j\tau_{j\mathrm{pfa}}\right)(E_4).
\end{gather*}
The integrals cancel\footnote{Since $S_j \in (L^\infty)^*$ we have $S_{j\mathrm{rn}} \in L^1$, and the terms we are cancelling
  are indeed finite.} and we get
$$ \varepsilon\mu(E_4) \leq \mu(E_4)(\varepsilon d C^2c^{-2}\delta + dC^3c^{-2}\delta) + Cc^{-1}\tau(E_4)$$
where $\tau = \sum_j\tau_{j\mathrm{pfa}}$. We can then choose $E_4$ and $\delta$ in much the same way as before to yield
a contradiction. Note that $\delta$ will only depend on $d$, $C$, $c$ and $\varepsilon$.

\medskip
So in both cases we have a contradiction to the existence of $E_0$, and so \eqref{lpo} must hold.

\end{proof}

\medskip
\noindent
It turns out that we shall need to consider the action of $S \in L^\infty(X, {\rm d}\mu)^\ast_+$ not just on
$L^\infty(X)$, but on general nonnegative measurable functions in $\mathcal{M}(X)$. The reasons for this are 
explained in Section~\ref{turn} below. To this end, we extend $S$ to $\mathcal{M}(X)_+$ by declaring, for $F \in \mathcal{M}(X)_+$,
$$ S(F) := \sup\{ S(f) \, : \, 0 \leq f \leq F, f \in L^\infty(X)\} = \sup\{ S(\phi) \, : \, 0 \leq \phi \leq F, \; \phi \mbox{ simple}\}.$$
Of course $S(F)$ will often now take the value $+ \infty$.

\medskip
\noindent
The second lemma concerns continuity properties of
this extension. Consider the map $S \mapsto S(F)$ for fixed $F \in \mathcal{M}(X)_+$ as
$S$ ranges over $L^\infty(X)_+^*$. If $F \in L^\infty(X)$ this map is norm continuous and hence weak-star
continuous. For $F \in  \mathcal{M}(X)_+$ we can assert less.

\begin{lem}\label{lsc}
  Fix  $F \in \mathcal{M}(X)_+$. Then the map $S \mapsto S(F)$ from $L^\infty(X)_+^*$ to $\mathbb{R} \cup \{+ \infty\}$
  is weak-star lower semicontinuous.
  \end{lem}

\noindent
We remark that we have to be cautious here since $L^\infty(X)_+^*$ with the weak-star topology is not a metric space;
so we cannot simply concern ourselves with sequential lower semicontinuity.

\begin{proof}
  Let $S \in L^\infty(X)_+^*$. Either $S(F) = + \infty$ or $S(F) < + \infty$. Let us first deal with the latter case.
  We need to show that for every
  $\epsilon > 0$ there is a weak-star open neighbourhood $U$ of $S$ such that for $R \in U$ we have
  $R(F) \geq S(F) - \epsilon$.

  \medskip
  \noindent
  Since $S(F) < + \infty$ there is an $f \in L^\infty(X)$ with $ 0 \leq f \leq F$ such that $S(f) > S(F) - \epsilon$.
  Let
  $$U = \{R \in L^\infty(X)_+^* \, : \, R(f) >  S(F) - \epsilon\}.$$
  Then $S \in U$, and $U$ is weak-star open since for each $f \in L^\infty(X)$ the functional $R \mapsto R(f)$
  is weak-star continuous. So for $R \in U$ we have
  $$ R(F) \geq R(f) > S(F) - \epsilon$$
  which is what we needed.

  \medskip
  \noindent
  Now we look at the case $S(F) = +\infty$. We now need to show that for every $N \in \mathbb{N}$ there is a
  weak-star open neighbourhood $U$ of $S$ such that for $R \in U$ we have
  $R(F) \geq N$.

  \medskip
  \noindent
  Since $S(F) = + \infty$ there is an $f \in L^\infty(X)$ with $ 0 \leq f \leq F$ such that $S(f) > N$.
  Let
  $$U = \{R \in L^\infty(X)_+^* \, : \, R(f) >  N \}.$$
  Then $S \in U$, and $U$ is weak-star open since for each $f \in L^\infty(X)$ the functional $R \mapsto R(f)$
  is weak-star continuous. So for $R \in U$ we have
  $$ R(F) \geq R(f) > N$$
  which is what we needed.
  
\end{proof}

\subsection{Proof of Theorem~\ref{thmsimple}}\label{turn}
Suppose we are in the situation in the statement of Theorem~\ref{thmsimple}. In particular, we assume 
that $G \in \mathcal{X}^\prime$, and we may clearly assume that $\|G\|_{\mathcal{X}'} \neq 0$.

\medskip
We recall from Section~\ref{overv} that we take $C$ (in which we locate the variables $\Phi = (K, S_j)$)
to be the positive cone $\mathbb{R}_+ \times (L^\infty(X)^\ast_+)^d$
in the vector space $\mathbb{R} \times (L^\infty(X)^\ast)^d$, and $C$ is given the topology inherited from the product
topology of the corresponding weak-star topologies. We take $D$
(in which we locate the variables $\Psi=(\beta_j,h_j)$) to be the positive cone
in the vector space $\mathcal{S}(X)^d \times \mathcal{Y}_1 \times
\dots \times \mathcal{Y}_d$.

\medskip
\noindent
Therefore, for $K\in\mathbb{R}_+$, $S_j\in L^\infty(X)^\ast_+$, 
$\beta_j$ simple functions on $X$ and $h_j \in \mathcal{Y}_j$ we consider
the functional
\begin{align*}
    L = \, &K+\left(\int_X G(x)\prod_{j=1}^d \beta_j^{\alpha_j}(x) \, {\rm d}\mu(x)-\sum_{j=1}^d \alpha_j S_j(\beta_j) \right)\\
    &+\sum_{j=1}^d\left(S_j(T_jh_j)- K\|G\|_{\mathcal{X}^\prime}\|h_j\|_{\mathcal{Y}_j}\right).
\end{align*}

\medskip
Note that the integral term is well-defined since $G \in \mathcal{X}^\prime$ (which is contained in $L^1(X, {\rm d} \mu)$ when
$\mu$ is a finite measure, as we have previously observed) and the $\beta_j$ are simple functions, and that the terms
$S_j(\beta_j)$ are also well-defined since the $\beta_j$ are bounded functions. The terms $S_j(T_j h_j)$ are
well-defined via the extension of $S_j$ to $\mathcal{M}(X)_+$ as discussed in Section~\ref{Preliminaries} above.
Thus $L: C \times D \to \mathbb{R} \cup \{+ \infty\}$ is well-defined
and takes values in $\mathbb{R} \cup \{+ \infty\}$, with the possible value $+\infty$ arising when $T_jh_j$ is not a
{\em bounded} measurable function.

\medskip
\noindent
We next want to see how we can apply the minimax theorem, Theorem~\ref{minimax} to this Lagrangian.
Recall from Remark~\ref{urgh} above that we have a problem in so doing, since our Lagrangian may take the value $+ \infty$,
while Theorem~\ref{minimax} requires that the Lagrangian be real-valued. To be clear, what we desire -- and what we shall indeed obtain -- is the conclusion
$$\min_{\Phi\in C}\sup_{\Psi\in D} L(\Phi, \Psi) = \sup_{\Psi\in D}\inf_{\Phi\in C} L(\Phi, \Psi)$$
of Theorem~\ref{minimax} in our case, but in order to achieve this we need to make a detour.

\subsubsection{A detour}\label{detourdetails} We now describe the necessary detour. This involves modifying the Lagrangian
we have defined in order to make it real-valued, but without altering its essential purpose. The main
technical difference is that instead of allowing $S_j$ to act on the possibly unbounded $T_jh_j$, we have it act
on an arbitrary nonnegative simple function $\psi_j$ satisfying $\psi_j \leq T_j h_j$, 

\medskip
\noindent
We therefore introduce a new Lagrangian $\Lambda: C \times \tilde{D} \to \mathbb{R}$, where $C$ is as before, and where
$$\tilde{D} = \{(\beta_j, h_j, \psi_j) \in \mathcal{S}(X)^d \times \mathcal{Y}_1 \times \dots \times \mathcal{Y}_d \times \mathcal{S}(X)^d \, : \, \beta_j \geq 0,\, h_j \geq 0,\,  0 \leq \psi_j \leq T_j h_j \}.$$
Note that $\tilde{D}$ is convex.

\medskip
\noindent
For $(K, S_j) \in C$ and $(\beta_j, h_j, \psi_j) \in \tilde{D}$ we define
\begin{align*}
    \Lambda = \, &K+\left(\int_X G(x)\prod_{j=1}^d \beta_j^{\alpha_j}(x) \, {\rm d}\mu(x)-\sum_{j=1}^d \alpha_j S_j(\beta_j) \right)\\
    &+\sum_{j=1}^d\left(S_j\psi_j- K\|G\|_{\mathcal{X}^\prime}\|h_j\|_{\mathcal{Y}_j}\right).
\end{align*}
Note that $\Lambda$ is real-valued since $S_j \psi_j$ is real-valued. Moreover, note that by the definition of the extension of $S_j$
to $\mathcal{M}(X)_+$, we have 
\begin{equation}\label{lah}
  L((K, S_j),(\beta_j, h_j)) = \sup_{\{\psi_j \, : \, \psi_j \leq T_jh_j\}} \Lambda((K, S_j),(\beta_j, h_j,\psi_j)).
  \end{equation}
We will momentarily check that the Lagrangian $\Lambda$ satisfies the hypotheses of Theorem~\ref{minimax}, but taking
this as read for now, we deduce using \eqref{lah} that 
$$  \min_{(K,S_j)\in C}\sup_{(\beta_j, h_j)\in {D}} L =  \min_{(K,S_j)\in C}\sup_{(\beta_j, h_j, \psi_j)\in \tilde{D}} \Lambda = \sup_{(\beta_j, h_j, \psi_j)\in \tilde{D}}\inf_{(K,S_j)\in C} \Lambda.$$

\medskip
\noindent
But since trivially $\sup \inf \leq \inf \sup$, we have, using \eqref{lah} once more,
$$ \sup_{(\beta_j, h_j, \psi_j)\in \tilde{D}}\inf_{(K,S_j)\in C} \Lambda \leq  \sup_{(\beta_j, h_j)\in {D}}\inf_{(K,S_j)\in C}
\sup_{\psi_j\leq T_jh_j} \Lambda =  \sup_{(\beta_j, h_j)\in {D}}\inf_{(K,S_j)\in C}L.$$
Combining the last two displays we obtain
$$ \min_{(K,S_j)\in C}\sup_{(\beta_j, h_j)\in {D}} L \leq \sup_{(\beta_j, h_j)\in {D}}\inf_{(K,S_j)\in C}L.$$
Since the reverse inequality is once again trivial we conclude that
$$\min_{\Phi\in C}\sup_{\Psi\in D} L(\Phi, \Psi) = \sup_{\Psi\in D}\inf_{\Phi\in C} L(\Phi, \Psi)$$
as we needed.

\medskip
Now we need to look at conditions (i) -- (iv) of Theorem~\ref{minimax} in our case where $\Lambda$ replaces $L$.
Concerning (i), the map $S \mapsto S(F)$ is linear on $L^\infty(X)^\ast$ for each fixed $F \in \mathcal{S}(X)$.
Therefore, for each fixed $\tilde\Psi \in \tilde{D}$, the map $\Phi \mapsto \Lambda(\Phi, \tilde\Psi)$
is affine, thus convex on $C$. Concerning (ii), the map $F \mapsto S(F)$ is linear
and hence concave on $\mathcal{S}(X)$ for each fixed $S \in L^\infty(X)^\ast_+$. Moreover the geometric mean is a concave
operation and the map $h \mapsto \|h\|_{\mathcal{Y}_j}$ is convex. Therefore, for each fixed $\Phi \in C$,
the map $\tilde\Psi \mapsto \Lambda(\Phi, \tilde\Psi)$ is concave on $\tilde{D}$. Concerning (iii), this follows directly
from the norm-continuity of $S \mapsto S(F)$ on $L^\infty(X)^\ast$ for each fixed $F \in \mathcal{S}(X)$.

\medskip
\noindent
Condition (iv) is more interesting, and it is in verification of this condition that we use the crucial strong saturation
hypothesis of Theorem~\ref{thmsimple}. We need to see that for some $\tilde\Psi_0 \in \tilde{D}$ the sublevel sets 
$\{\Phi \in C \, : \, L(\Phi,\tilde\Psi_0)\leq\lambda\}$ are 
compact for all sufficiently large $\lambda$. We will show that for a suitable choice of $\tilde\Psi_0$ these sets 
are norm-bounded, and from this the Banach--Alaoglu theorem will give us compactness. 

\medskip
\noindent
We take $\tilde\Psi_0 = (\beta_j, h_j, \psi_j)$ to have $\beta_j=0$ for all $j$.
We take $h_j \in \mathcal{Y}_j$ such that $T_j h_j \geq c_0 > 0$ a.e. on $X$, as guaranteed by the hypothesis
of Theorem~\ref{thmsimple}. By multiplying by a suitable positive constant if necessary, we can certainly assume
that $ \sum_{j=1}^d\|h_j\|_{\mathcal{Y}_j} < (2 \|G\|_{\mathcal{X}'})^{-1}$. Finally, we take $\psi_j = c_0 {\bf 1}$ which
satisfies $0 \leq \psi_j \leq T_j h_j$.

\medskip
\noindent
For such a choice of $\tilde\Psi_0$ we have 
\begin{displaymath}
    \Lambda((K,S_j), \tilde \Psi_0)=
    K\left(1-\|G\|_{\mathcal{X}^\prime}\sum_{j=1}^d\|h_j\|_{\mathcal{Y}_j}\right)
    +\sum_{j=1}^dS_j \psi_j \geq K/2 + c_0 \sum_{j=1}^dS_j {\bf 1}.
\end{displaymath}

Therefore, for given $\lambda > 0$,
$$ \{(K, S_j) \in C \, : \,  \Lambda((K,S_j), \tilde \Psi_0) \leq \lambda\} \subseteq
[0, 2 \lambda] \times \{S \in L^\infty(X)^\ast_+ \, : S ({\bf 1}) \leq c_0^{-1}\lambda\}^d.$$

\medskip
\noindent
But it is easy to see that for $S \geq 0$, $S({\bf 1}) = \|S\|_{L^\infty(X)^\ast}$. Indeed, by definition we have
$$\|S\|_{L^\infty(X)^\ast}=\sup\{|S(u)| \, : \, u\in L^\infty(X),\|u\|_\infty\leq 1\}.$$
So let us take such a function $u$ with $\|u\|_\infty \leq 1$. Since $S(-u)=-S(u)$ we may by choosing either $u$ or $-u$
assume that $S(u)\geq 0$. We have that $u\leq 1$ a.e. and therefore the non-negativity
of $S$ gives us that $S(u)\leq S({\bf 1})$, as needed.

\medskip
\noindent
Therefore 
$$ \{(K, S_j) \in C \, : \,  \Lambda((K,S_j), \tilde \Psi_0) \leq \lambda\} \subseteq
[0, 2 \lambda] \times \{S \, : \|S\|\leq c_0^{-1}\lambda\}^d$$
is a norm-bounded, weak-star closed, hence weak-star compact subset of $\mathbb{R} \times (L^\infty(X)^\ast)^d$,
by the Banach--Alaoglu theorem. This completes the verification of condition (iv) of Theorem~\ref{minimax} in our case,
and we conclude that
$$\min_{\Phi\in C}\sup_{\Psi\in D} L(\Phi, \Psi) = \sup_{\Psi\in D}\inf_{\Phi\in C} L(\Phi, \Psi).$$

\subsubsection{Return to the main argument}
We may therefore conclude, by Theorem~\ref{minimax}, that if for non-zero $G \in \mathcal{X}^\prime $ fixed
we define\footnote{The notation here is perhaps confusing. We shall consider
four problems: $\gamma$, $\gamma^\ast$, $\gamma_\mathcal{L}$ and $\gamma_\mathcal{L}^\ast$. When there is no 
superscript we are dealing with the variant of the problem pertaining to $L^1$, 
and presence of the superscript $\ast$ denotes that we are dealing with the 
variant of the problem which pertains to $(L^\infty)^\ast$; when there is no 
subscript we are dealing with the original version of the problem, and presence 
of the subscript $\mathcal{L}$ denotes that we are dealing with the Lagrangian formulation. 
This is consistent with the notation we adopted in the treatment of the finite discrete case above; 
in that case there was no distinction between $L^1$ and $(L^\infty)^\ast$. We do not adorn $\eta$ with either
a superscript $*$ nor a subscript $\mathcal{L}$ since there is only one $\eta$-problem. Nevertheless we emphasise that
the $\eta$-problem does indeed deal with the Langrangian formulation in the form pertaining to $(L^\infty)^\ast$.} 

$$\gamma_\mathcal{L}^\ast = \inf_{(K,S_j) \in C}\sup_{(\beta_j,h_j) \in D} L
\qquad\text{and}\qquad
\eta = \sup_{(\beta_j,h_j) \in D}\inf_{(K,S_j) \in C} L$$
then $\eta = {\gamma}_\mathcal{L}^\ast$ and the infimum in the problem for $\gamma_\mathcal{L}^\ast$ is
achieved as a minimum. (It should be noted that we are not yet in a position to assert the finiteness
of either of these numbers.)

\bigskip
In order to progress further, we shall also consider the problem
\begin{equation}
    \label{convprob}
\begin{aligned}
\gamma =& \inf K\\
\text{such that}\quad &
G(x)\leq \prod_{j=1}^d S_j(x)^{\alpha_j}\quad\text{a.e. on $X$, and}\\
&
\int_X S_j(x) T_j h_j(x) \, {\rm d}\mu(x) \leq K \|G\|_{\mathcal{X}^\prime} \|h_j\|_{\mathcal{Y}_j} \quad\text{for all $j$ and 
all $h_j \in \mathcal{Y}_j$}\\
\end{aligned}
\end{equation}
where the $S_j$ are taken to be in $L^1(X, {\rm d \mu)}$. We emphasise that this is the problem we really want to solve:
if we can prove that $\gamma \leq A$ and that minimisers exist, we will have our desired factorisation. Nevertheless,
we should point out that it is not yet even clear that there exist $(K, S_j)$ satisfying the constraints
of \eqref{convprob}. We shall be able to infer the existence of such $(K, S_j)$, and hence the finiteness of $\gamma$, 
only from the conclusion of Theorem~\ref{thmsimple}.

\medskip
Our strategy is to show that (i) $\gamma_\mathcal{L}^\ast = \gamma$ and that if the problem for $\gamma_\mathcal{L}^\ast$
admits minimisers $\Phi$, then the problem for $\gamma$ also admits minimisers; and (ii) $0 \leq \eta \leq A$.
Combining these with the minimax result $\gamma_\mathcal{L}^\ast=\eta$ and existence of minimsiers for
$\gamma_\mathcal{L}^\ast$, we can conclude that the problem for $\gamma$ admits
minimisers and that $\gamma \leq A$,
which will conclude the proof of Theorem~\ref{thmsimple}.

\medskip
\textbf{Proof that $\gamma = \gamma_\mathcal{L}^\ast$ and that existence of minimisers for $\gamma_\mathcal{L}^\ast$ implies existence of
  minimisers for $\gamma$.}
We begin by studying $\gamma_\mathcal{L}^\ast$ so
let us consider for which $(K,S_j) \in C$ we have
$\sup_{\beta_j,h_j} L < \infty$. Fix $(K, S_j)$. 
First of all, 
suppose that $S_j$ are such that there exists a tuple $(\beta_j)$ such that
\begin{displaymath}
    \int_X G(x)\prod_{j=1}^d \beta_j^{\alpha_j}(x)\, {\rm d}\mu(x)-\sum_{j=1}^d\alpha_jS_j(\beta_j)>0.
\end{displaymath}
Then by setting $h_j=0$ and substituting $\beta_j \mapsto t \beta_j$ and letting $t \to \infty$
we see that the supremum is infinite. Therefore, if $\sup_{\beta_j,h_j} L < \infty$, we must have
\begin{equation}
    \label{condbeta}
    \int_X G(x)\prod_{j=1}^d \beta_j^{\alpha_j}(x) \, {\rm d}\mu(x)\leq\sum_{j=1}^d\alpha_jS_j(\beta_j)
    \qquad\text{for all simple functions $\beta_j$,}
\end{equation}
which, by Lemma~\ref{abscont}, is equivalent to
\begin{equation}\label{eqnabscont}
G(x)\leq \prod_{j=1}^d {S}_{j\mathrm{rn}}(x)^{\alpha_j}\quad\text{a.e. on $X$. }
\end{equation}

\medskip
Now assume there exists a $j_0$ 
and an $h_{j_0} \in \mathcal{Y}_{j_0}$
such that $S_{j_0}(T_{j_0}h_{j_0})>K\|G\|_{\mathcal{X}^\prime}\|h_{j_0}\|_{\mathcal{Y}_{j_0}}$.
Taking $\beta_j = 0$,
and $h_j= 0$ for $j\neq j_0$ and multiplying $h_{j_0}$ by a factor $t$ which
we send to infinity we again see that the supremum is infinite.
Therefore, if $\sup_{\beta_j,h_j} L < \infty$, then
we must also have
\begin{equation}
    \label{condScircT}
    S_{j}(T_{j}h_{j})\leq K\|G\|_{\mathcal{X}^\prime}\|h_j\|_{\mathcal{Y}_j}
\end{equation}
for all nonnegative $h_j$ and all $j$.
From the positivity of $S_{j\mathrm{pfa}}$ we see that this implies
\begin{equation}\label{condTastS}
    \int_X S_{j\mathrm{rn}}(x)T_jh_j(x)\, {\rm d}\mu(x)\leq K\|G\|_{\mathcal{X}^\prime}\|h_j\|_{\mathcal{Y}_j}
\end{equation}
for all nonnegative $h_j$ and all $j$.

\medskip
On the other hand, if for fixed $(K, S_j)$ conditions~\eqref{eqnabscont}
and~\eqref{condScircT}
are satisfied, then when we are looking for $\sup_{\beta_j,h_j} L$,
we can do no better than taking $\beta_j =0$ and $h_j=0$
for all $j$. So for fixed $(K, S_j)$, we have $\sup_{\beta_j,h_j} L < \infty$
if and only if conditions ~\eqref{eqnabscont} and~\eqref{condScircT} hold, in which case 
$\sup_{\beta_j,h_j} L = K$. So the problem for $\gamma_\mathcal{L}^\ast$ is identical with the 
problem
\begin{align*}
    \gamma^\ast =& \inf K\\
\text{such that}\quad &
G(x)\leq \prod_j S_{j\mathrm{rn}}(x)^{\alpha_j}\qquad\text{a.e.},\\
& S_{j}(T_{j}h_{j})\leq K\|G\|_{\mathcal{X}^\prime}\|h_j\|_{\mathcal{Y}_j}\quad
    \text{for all $j$ and all $h_j \in \mathcal{Y}_j$,}
\end{align*}
where we emphasise that the $\inf$ is taken over $(K, S_j)$ with $S_j \in L^\infty(X)_+^\ast$.

\medskip
Likewise, the problem 
$$\gamma_\mathcal{L} := \inf_{(K,S_j) \in \mathbb{R}_+ \times (L^1(X)_+)^d}\sup_{\beta_j, h_j}L$$
is identical with problem~\eqref{convprob} for $\gamma$.

\medskip
It is clear that $\gamma_\mathcal{L}^\ast \leq \gamma_\mathcal{L}$ as the infimum for the left hand side is over a 
larger set than for the right hand side.

\medskip
{\bf Claim:} $\gamma_\mathcal{L} \leq \gamma_\mathcal{L}^\ast$, and if minimisers $\Phi = (K, S_j)$ exist for problem $\gamma_\mathcal{L}^\ast$,
they also exist for problem $\gamma_\mathcal{L}$.

\medskip
Indeed, assume that $\gamma_\mathcal{L}^\ast < \infty$, let $\varepsilon > 0$ and let $(K, S_j)$ with $S_j \in L^\infty(X)^\ast$
and satisfying conditions~\eqref{eqnabscont} and~\eqref{condScircT} be such that $K < \gamma_\mathcal{L}^\ast + \varepsilon$. 
Then the absolutely continuous component $S_{j\mathrm{rn}}$ satisfies~\eqref{eqnabscont} and~\eqref{condTastS}, and so
$(K,S_{j\mathrm{rn}})$ contributes to the infimum in the problem for $\gamma_\mathcal{L}$. Thus 
$\gamma_\mathcal{L} \leq \gamma_\mathcal{L}^\ast + \varepsilon$. Letting $\varepsilon \to 0$ establishes the first part of 
the claim.
Now suppose that minimisers $\Phi = (K, S_j)$ exist for problem $\gamma_\mathcal{L}^\ast$. In particular this supposes that
$\gamma_\mathcal{L}^\ast < \infty$. Let $(K, S_j)$ with $S_j \in L^\infty(X)^\ast$
and satisfying conditions~\eqref{eqnabscont} and~\eqref{condScircT} be such that $K = \gamma_\mathcal{L}^\ast$. Then the
absolutely continuous component $S_{j\mathrm{rn}}$ satisfies~\eqref{eqnabscont} and~\eqref{condTastS}, and so
$(K,S_{j\mathrm{rn}})$ contributes to and indeed achieves the infimum in the problem for $\gamma_\mathcal{L}$ (otherwise
$\gamma_\mathcal{L}$ would be strictly less than $\gamma_\mathcal{L}^\ast$).

\medskip Summarising, the problems for $\gamma$ and $\gamma_\mathcal{L}$ are equivalent; the problems for
$\gamma^\ast$ and $\gamma_\mathcal{L}^\ast$ are equivalent; $\gamma_\mathcal{L} = \gamma_\mathcal{L}^\ast$, and if extremisers exist for
$\gamma_\mathcal{L}^\ast$, they also exist for $\gamma_\mathcal{L}$, and hence too for $\gamma$.
 
\medskip
\textbf {Proof that $0 \leq \eta \leq A$.}
We wish to carry out a similar analysis for $\inf_{K,S_j} L$, and for that
we first of all rewrite $L$ as
\begin{align*}
    L=&\int_X G(x)\prod_{j=1}^d \beta_j^{\alpha_j}(x) \, {\rm d}\mu(x)
    +K\left(1-\|G\|_{\mathcal{X}^\prime}\sum_{j=1}^d\|h_j\|_{\mathcal{Y}_j}\right)\\
    &+\sum_{j=1}^dS_j\left(T_jh_j-\alpha_j\beta_j \right).
\end{align*}
We consider for which $(\beta_j,h_j) \in D$ we have $\inf_{K,S_j} L >-\infty$.

\medskip
First, by taking $S_j=0$ for all $j=1,\dots,d$ and
letting $K$ to go infinity 
we see that if $\inf_{K,S_j} L > -\infty$
then we must have
\begin{equation}
    \label{condsumhj}
    \|G\|_{\mathcal{X}^\prime}\sum_{j=1}^d \|h_j\|_{\mathcal{Y}_j}\leq1.
\end{equation}

Secondly, assume that there exists an index $j_0$ and a
set $E\subseteq X$
with $\mu(E)>0$ such that $T_{j_0}h_{j_0}(x)< \alpha_{j_0} \beta_{j_0}(x)$ for a.e. $x\in E$.
Then by taking $S_{j_0}= t \chi_E\in L^1$, $S_j=0$ for $j\neq j_0$ and
$K=0$ and letting $t \to \infty$, then we see that $\inf_{K,S_j}L =-\infty$.
Thus if $\inf_{K,S_j} L > -\infty$, we must also have
\begin{equation}
    \alpha_j\beta_j(x)\leq T_jh_j(x)\quad\text{a.e. on $X$ for all $j$.}
    \label{condThbeta}
\end{equation}

If conditions \eqref{condsumhj} and \eqref{condThbeta} are both
satisfied we can do no better than take $K=0$ and $S_j=0$ for all $j$. 
So, for fixed $(\beta_j,h_j)$, $\inf_{K,S_j} L >-\infty$ if and only if conditions
\eqref{condsumhj} and \eqref{condThbeta} hold, in which case $\inf_{K,S_j} L = 
\int_X G(x)\prod_{j=1}^d \beta_j^{\alpha_j}(x) {\rm d}\mu(x)$.
We can always find $(\beta_j,h_j)$ such that conditions \eqref{condsumhj} and \eqref{condThbeta} hold,
so $\eta = \sup_{\beta_j, h_j} \int_X G(x)\prod_{j=1}^d \beta_j^{\alpha_j}(x) {\rm d}\mu(x)$ subject to 
conditions \eqref{condsumhj} and \eqref{condThbeta}. In particular this tells us that $\eta \geq 0$.

\medskip
Let us now derive an upper bound for $\eta$.
Examining the condition \eqref{condThbeta} on $\beta_j$ we see that
\begin{align*}
\eta \leq& \sup_{h_j} \int_{X} G(x)
\prod_{j=1}^d \left(\alpha_j^{-1}T_{j}h_j(x)\right)^{\alpha_j} \, {\rm d}\mu(x)\\
\text{such that}\quad & 
\|G\|_{\mathcal{X}^\prime }\sum_{j=1}^d \|h_j\|_{\mathcal{Y}_j}\leq1.
\end{align*}
Clearly there exist functions $h_j \in \mathcal{Y}_j$ such that 
$\|G\|_{\mathcal{X}^\prime}\sum_j \|h_j\|_{\mathcal{Y}_j}\leq1$, and for any such we have
\begin{align*}
\int_X G(x) \prod_{j=1}^d \left(\alpha_j^{-1} T_jh_j(x)\right)^{\alpha_j} \, {\rm d}\mu(x)
&\leq \|G\|_{\mathcal{X}^\prime}
\|
\prod_{j=1}^d \left(T_j(\alpha_j^{-1} h_j)(x)\right)^{\alpha_j } \|_\mathcal{X}
\\
&\leq \|G\|_{\mathcal{X}^\prime}
A\prod_{j=1}^d\|\alpha_j^{-1}h_j\|_{\mathcal{Y}_j}^{\alpha_j}
\\
&\leq \|G\|_{\mathcal{X}^\prime}
A\sum_{j=1}^d\alpha_j\| \alpha_j^{-1}h_j\|_{\mathcal{Y}_j}
\\
&= \|G\|_{\mathcal{X}^\prime}
A\sum_{j=1}^d\|h_j\|_{\mathcal{Y}_j}\leq A
\end{align*}
by H\"older's inequality in the form $\int Gf \leq \|G\|_{\mathcal{X}'} \|f\|_{\mathcal{X}}$,
the multilinear inequality \eqref{mainineq} which is our main hypothesis,
the arithmetic-geometric mean inequality and finally the assumption on the $h_j$.
This clearly implies $\eta \leq A$, and thus concludes the proof of  Theorem~\ref{thmsimple}.

\qed

\subsection{Consequences of saturation}\label{tsjkr}
We give two lemmas needed for Theorem~\ref{thmmain}. These allow us
to construct suitable
exhausting sequences of subsets of $X$ of finite measure, in order that we might apply Theorem~\ref{thmsimple}. Then
we construct the weight $w$ of the statement of Theorem~\ref{thmmain}.

\begin{lem}\label{strsat}
  Let $(X, {\rm d} \mu)$ be a $\sigma$-finite measure space, and suppose that $\mathcal{P} \subseteq \mathcal{M}(X)_+$ has the property that for every measurable set $E \subseteq X$ with $\mu(E) > 0$, there exists an $f \in \mathcal{P}$ and a subset $E' \subseteq E$ with $\mu(E')>0$, such that $f > 0$ a.e. on $E'$. Then there exists a countable subset $\{f_n\}_{n \in \mathbb{N}} \subseteq \mathcal{P}$ such that, with $E_n := \{x \in X \, : \, f_n(x) > 0\}$,
  $$ \mu(X\setminus\bigcup_{n=1}^\infty E_n) = 0.$$
  \end{lem}

\begin{proof}
  By exhausting $X$ by a countable sequence of subsets, each of finite measure, we may assume that $\mu(X)$ is finite. We claim that for every $\epsilon >0$ there is a finite subset $\{f_1, \dots ,f_N\} \subseteq \mathcal{P}$ such that
  $$  \mu(X\setminus\bigcup_{n=1}^N E_n) < \epsilon.$$
  Once we have this claim, we take the union of the finite subsets of $\mathcal{P}$ obtained for each $\epsilon = 1/m$, $m \in \mathbb{N}$, and we are finished.

  \medskip
  Suppose, for a contradiction, that there is some $\epsilon > 0$ such that for all $N$, for all finite subfamilies $\{f_1, \dots , f_N\} \subseteq \mathcal{P}$ we have
  $$ \mu(X \setminus \bigcup_{n=1}^N E_n) \geq \epsilon > 0.$$
  Let
  $$ t = \inf_N \inf_{\{ f_1, \dots , f_N\} \subseteq \mathcal{P}}\, \mu(X \setminus \bigcup_{n=1}^N E_n).$$
  Then $t\geq \epsilon >0$ and also $t < \infty$ since $\mu(X)$ is finite.
  For $m \in \mathbb{N}$ let $\mathcal{P}_m = \{f_1, \dots , f_{N(m)}\}$ be such that 
  $$ \mu(X \setminus \bigcup_{n=1}^{N(m)} E_n) \leq t + 1/m;$$
  we may assume that  $\mathcal{P}_m \subseteq \mathcal{P}_{m+1}$ for all $m$. Letting $m \to \infty$ we obtain
  $$ \mu(X \setminus \bigcup_{n=1}^{\infty} E_n) \leq t.$$
  If $ \mu(X \setminus \bigcup_{n=1}^{\infty} E_n) =0$ we are done; otherwise $E = X \setminus \bigcup_{n=1}^{\infty} E_n$ has positive measure, and therefore, by hypothesis, there is a subset $E' \subseteq E$ with $\mu (E') = \delta > 0$ such that for some $f_0 \in \mathcal{P}$ we have $E' \subseteq E_0$. Then, (with the union now starting at $n=0$),
  $$ \mu(X \setminus \bigcup_{n=0}^{N(m)} E_n ) \leq t + 1/m - \delta.$$
  If we choose $m > \delta^{-1}$, we then have
   $$ \mu(X \setminus \bigcup_{n=0}^{N(m)} E_n ) < t,$$
  in contradiction to the definition of $t$.
  
\end{proof}

\begin{lem}\label{sat}
  Let $(X, {\rm d} \mu)$ be a $\sigma$-finite measure space, $\mathcal{Y}$ a normed lattice, and suppose that $T: \mathcal{Y} \to \mathcal{M}(X)$ is a positive linear operator which saturates $X$. Then there is an increasing 
  exhausting sequence of subsets $(G_n)$ of $X$, each of finite measure, such that $T$
  strongly saturates each $G_n$. More precisely, there exists a sequence $(h_n) \subseteq \mathcal{Y}_+$ such that $h_{n+1} \geq h_n$ for all $n$, such that $\|h_n\|_\mathcal{Y} \leq 1$ for all $n$, and such that for all $n$, $Th_n(x) \geq 1/n$ for $x \in G_n$. 
\end{lem}

\begin{proof}
  Let $\mathcal{P} = T(\mathcal{Y}_+) \subseteq \mathcal{M}(X)_+$. The saturation hypothesis allows us to deduce from the previous lemma that there exists a sequence $h_n \in \mathcal{Y}_+$
  such that if $E_n = \{Th_n>0\}$, then $\{E_n\}_{n\in \mathbb{N}}$ covers $X$ up to a set of
  measure zero. Letting
  $$ \tilde{h}_n = 2^{-1}\frac{h_1}{\|h_1\|_{\mathcal{Y}}} + 2^{-2} \frac{h_2}{\|h_2\|_{\mathcal{Y}}} + \cdots + 2^{-n}\frac{h_n}{\|h_n\|_{\mathcal{Y}}},$$
  we see that we may additionally assume that $\|h_n\|_{\mathcal{Y}} \leq 1$
  for all $n$, and that $h_{n+1} \geq h_n$. Thus without loss of generality
  $E_n \subseteq E_{n+1}$ for all $n \in \mathbb{N}$.  

  \medskip
  Since $X$ is $\sigma$-finite there is an increasing sequence of subsets $F_n$
  of finite measure which exhausts $X$. Now we set
  $$G_n  := \{ x \, : \, Th_n > 1/n \} \cap F_n.$$
  Clearly $G_n \subseteq G_{n+1}$ for all $n$, and each $G_n$ has finite measure.
  We check that $\{G_n\}$ is exhausting.
  Let $x \in X$. Then $x \in E_k$ for some $k$, i.e. $Th_k(x) >0$, and therefore
  $Th_k(x) > 1/l$ for some $l \in \mathbb{N}$. Since $X$ is exhausted by $\{F_m\}$
  there is some $m$ such that $x \in F_m$. Therefore, for $n$ such that
  $n \geq \max\{k,l.m\}$, we have that $x \in G_n$. Finally, it is clear by definition
  that $T$ strongly saturates $G_n$. 
  
  \end{proof}

\medskip
\noindent
As an immediate consequence, we have:

\begin{cor}\label{corp}
Let $(X, {\rm d} \mu)$ be a $\sigma$-finite measure space and let $\mathcal{Y}_j$ be normed lattices for $1 \leq j \leq d$. Assume that $T_j: \mathcal{Y}_j \to \mathcal{M}(X)$ for $1 \leq j \leq d$ are positive linear operators, each of which saturates $X$.
Then for each $1 \leq j \leq d$ there exists a sequence $(h_{j,n})_n \subset \mathcal{Y}_j$
such that $\|h_{j,n}\|_{\mathcal{Y}_j} \leq 1$, $h_{j,n} \leq h_{j,m}$ for $m \geq n$, and there exists an increasing and exhausting sequence of
subsets $E_n \subseteq X$, each of finite measure,
such that for each $j$ and $n$, $T_j h_{j,n}(x) \geq 1/n$  for $x \in E_n$.
\end{cor}

\medskip
\noindent
With this in hand, we can now define the weight $w$ referred to in Remark~\ref{weightrem} above. Let $w_j(x)$ for $x \in E_m \setminus E_{m-1}$ be $T_j h_{j,m}(x)$, where we take $E_0 = \emptyset$. Define $w(x) = \min_j w_j(x)$. Note that $w$ is a.e. positive and a.e. finite. (If the sets $E_m$ stabilise in the sense that for some $M \in \mathbb{N}$, $E_M = X$ up to a set of measure zero, then $w \geq 1/M$, and we can simply take $w$ to be $1$).

\subsection{Proof of Theorem~\ref{thmmain}}\label{tsjk}
We will prove Theorem~\ref{thmmain} by reducing it to Theorem~\ref{thmsimple}.
We will need the following lemma whose proof is an easy exercise in elementary
point-set topology, and which is therefore omitted.

\begin{lem}\label{top}
  Let $Z$ be a compact topological space and suppose $(z_n)$ is an infinite sequence of distinct points in $Z$.
  Then there
  exists a point $z \in Z$ such that every open neighbourhood of $z$ contains infinitely many $z_n$'s.
  \end{lem}

\medskip
{\em Proof of Theorem~\ref{thmmain}}.
We may assume that $A < \infty$ otherwise there is nothing to prove. Take a nonzero $G \in \mathcal{X}^\prime$, and take
$E_n$ as in Corollary~\ref{corp}. For each $m$ we can apply Theorem~\ref{thmsimple}, with $X$ replaced by 
$E_m$, to conclude that there exist $g_{j,m} \in L^1(E_m, {\rm d} \mu)$ such that 
\begin{equation}\label{kjh}
G(x) \leq \prod_{j=1}^d g_{j,m}(x)^{\alpha_j}\qquad\mbox{a.e. on $E_m$,}
\end{equation}
and such that for each $j$, 
\begin{equation}\label{control1b}
\int_{E_m} g_{j,m}(x)T_jf_j(x) {\rm d}\mu(x) \leq A\|G\|_{\mathcal{X}^\prime}\|f_j\|_{\mathcal{Y}_j}
\end{equation}
for all $f_j \in \mathcal{Y}_j$.

\medskip
\noindent
If $E_M = X$ (up to a set of zero measure) for some $M$, we simply
take $g_j = g_{j,M}$ and we are finished. So we may assume that the
sets $E_m$ do not stabilise, and therefore that there are infinitely
many distinct $g_{j,m}$ for each $j$.

\medskip
\noindent
With $w$ defined as in the previous subsection, let us now calculate
\begin{align*}
\|g_{j,m}\|_{L^1(w \,{\rm d}\mu)} &\leq \int_{{E}_m} g_{j,m}(x)\,w_j (x)\, {\rm d}\mu(x)
=\sum_{n=0}^m\int_{{E}_n\setminus {E}_{n-1}} g_{j,m}(x)\,T_j h_{j,n}(x) \, {\rm d}\mu(x)\\
&\leq \int_{{E}_m} g_{j,m}(x)\, T_j h_{j,m}(x) \, {\rm d}\mu(x)
\leq A \|G\|_{\mathcal{X}^\prime}\|h_{j,m}\|_{\mathcal{Y}_j}\leq A\|G\|_{\mathcal{X}^\prime}.
\end{align*}
Thus the functions $g_{j,m}$ all lie in a ball in $L^\infty(X, \, w\, {\rm d}\mu)^\ast$
which, by the Banach--Alaoglu theorem, is weak-star compact. It is therefore tempting to extract a weak-star
convergent subsequence. However, we must resist this temptation since $L^\infty(X, \, w\, {\rm d}\mu)$ is not separable,
and thus  $L^\infty(X, \, w\, {\rm d}\mu)^\ast$ is not metrisable. We therefore proceed with some caution.
We will use Lemma~\ref{top} as a substitute for the existence of weak-star convergent subsequences.

\medskip
\noindent
It is convenient to consider the vectors
$${\bf g}_n = (g_{1,n}, \dots , g_{d,n}) \in L^1(X, w \,{\rm d}\mu) \times \dots \times L^1(X, w \,{\rm d}\mu)$$
$$  \subseteq L^\infty(X, \, w\, {\rm d}\mu)^\ast \times \dots \times L^\infty(X, \, w\, {\rm d}\mu)^\ast = (L^\infty(X, \, w\, {\rm d}\mu) \times \dots \times L^\infty(X, \, w\, {\rm d}\mu))^\ast.$$
By Lemma~\ref{top} there is a point ${\bf S} = (S_1, \dots , S_d) \in (L^\infty(X, \, w\, {\rm d}\mu) \times \dots \times L^\infty(X, \, w\, {\rm d}\mu))^\ast$ such that every weak-star open neighbourhood of  ${\bf S}$ contains infinitely many of the ${\bf g}_n$.

\begin{lem}\label{vector}
  Suppose $({\bf g}_n)$ and ${\bf S}$ are as above.

  \medskip
  \noindent
  (a) If for some ${\bf q} \in \mathcal{M}(X, w \, {\rm d} \mu)_+^d$ we have
  $$ {\bf g}_n ({\bf q}) = \sum_{j=1}^d \int_X g_{j,n} q_j \, w \, {\rm d} \mu = \int_X {\bf g}_n \cdot {\bf q}
  \, w \, {\rm d} \mu \leq K$$ for all sufficiently large $n$, then
  $$ {\bf S}({\bf q}) \leq K.$$

  \medskip
  \noindent
  (b) If 
  for some ${\bf q} \in L^\infty(X, w \, {\rm d} \mu)^d$ we have
  $$  {\bf g}_n ({\bf q}) = \sum_{j=1}^d \int_X g_{j,n} q_j w \, {\rm d} \mu = \int_X {\bf g}_n \cdot {\bf q}
  \, w \, {\rm d} \mu \geq L$$ for all sufficiently large $n$, then
  $$ {\bf S}({\bf q}) \geq L.$$
  \end{lem}

\begin{proof}
  (a) Suppose for a contradiction that ${\bf S}({\bf q}) \geq K'> K$ for some finite $K'$.
  Let
  $$U = \{{\bf R} \in ((L^\infty(w \, {\rm d} \mu))^d)^* \, : \; {\bf R} ({\bf q}) > (K + K')/2 \}.$$
  Then $ {\bf S} \in U$, and $U$ is weak-star open since for each ${\bf q} \in \mathcal{M}(w \,{\rm d} \mu)_+^d$ the
  functional $ {\bf R} \mapsto {\bf R} ({\bf q})$ is weak-star lower semicontinuous, by (a vector-valued version of)
  Lemma~\ref{lsc}. Thus $U$ is an open neighbourhood of ${\bf S}$ in the weak-star topology.
  By the above remarks, $U$ must contain infinitely many of the $({\bf g}_n)$.
  But for all $n$ sufficiently large,
  $$ {\bf g}_n ({\bf q}) = \int_X {\bf g}_n \cdot {\bf q} \, w \, {\rm d}\mu \leq K < (K + K')/2$$
  and so none of these ${\bf g}_n$ can be in $U$. This is a contradiction, and therefore ${\bf S}({\bf q}) \leq K$.

  \medskip
  \noindent
  (b) Suppose for a contradiction that ${\bf S}({\bf q})= L' < L$.
  Let
  $$U = \{{\bf R} \in ((L^\infty(w \, {\rm d} \mu))^d)^* \, : \; {\bf R} ({\bf q}) < (L + L')/2 \}.$$
  Then $ {\bf S} \in U$, and $U$ is weak-star open since for each ${\bf q} \in L^\infty(w \,{\rm d} \mu)_+^d$ the
  functional $ {\bf R} \mapsto {\bf R} ({\bf q})$ is weak-star continuous. Thus $U$ is an open neighbourhood
  of ${\bf S}$ in the weak-star topology.
  By the above remarks, $U$ must contain infinitely many of the $({\bf g}_n)$.
  But for all $n$ sufficiently large,
  $$ {\bf g}_n ({\bf q}) = \int_X {\bf g}_n \cdot {\bf q} \, w \, {\rm d}\mu \geq L > (L + L')/2$$
  and so none of these ${\bf g}_n$ can be in $U$. This is a contradiction, and therefore ${\bf S}({\bf q}) \geq L$.
\end{proof}

\medskip
We now wish to verify that the absolutely continuous components $({S}_{j\mathrm{rn}})$ (where the Radon--Nikodym derivative is
with respect to the measure $w \, {\rm d} \mu$) of $(S_j)$ satisfy

\begin{equation}\label{rew}
  \begin{aligned}
& G(x)\leq \prod_{j=1}^d {S}_{j\mathrm{rn}}(x)^{\alpha_j}\quad\text{a.e. on $X$, and}\\
    & \int_X {S}_{j\mathrm{rn}}(x) T_j f_j(x) \, {\rm d}\mu(x) \leq A \|G\|_{\mathcal{X}^\prime} \|f_j\|_{\mathcal{Y}_j}\\
\end{aligned}
\end{equation}
for all $j$ and for all $f_j \in \mathcal{Y}_j$. Since we know that $S_j \in (L^\infty(X, w \, {\rm d} \mu))^\ast$, we will therefore have
${S}_{j\mathrm{rn}} \in L^1( w \, {\rm d} \mu)$, and this will conclude the proof of Theorem~\ref{thmmain}.

\medskip
\noindent
We may suppose that $\|G\|_{\mathcal{X}'} = 1$.

\medskip
\noindent
We look at the second inequality from \eqref{rew} first.
Fix $m$ and consider
$$ \int_{E_m} S_{jrn} (x) T_j f_j(x) {\rm d} \mu (x) = \int_X S_{jrn} (x) w(x)^{-1} \chi_{E_m}(x) T_j f_j(x) w(x) {\rm d} \mu (x)$$
$$ \leq S_j( w^{-1} \chi_{E_m} T_j f_j)$$
by positivity of each component in the Yosida--Hewitt decomposition of $S_j$, (recall Theorem~\ref{thmYH1}). Now, for $n \geq m$,
$$ \int g_{jn}(x) [w(x)^{-1} T_j f_j(x) \chi_{E_m}(x)] w(x) {\rm d} \mu (x) \leq  \int g_{jn}(x) T_j f_j(x) {\rm d} \mu (x) \leq A \|f_j \|_{\mathcal{Y}_j},$$
so that by Lemma~\ref{vector}(a) (in the scalar case),
$$ S_j(w^{-1} T_j f_j \chi_{E_m}) \leq  A \|f_j \|_{\mathcal{Y}_j}.$$
Thus
$$ \int_{E_m} S_{jrn} (x) T_j f_j(x) {\rm d} \mu (x) \leq A \|f_j \|_{\mathcal{Y}_j}$$
and we now let $m \to \infty$ to get the second inequality of \eqref{rew}.

\medskip
\noindent
Now we look at the first inequality from \eqref{rew}. By Lemma~\ref{abscont} (using the measure $w \, {\rm d} \mu)$ and the fact
that the $E_m$ exhaust $X$, it suffices to show that for each fixed $m$, and all simple $\beta_j$, 
$$ \int_{E_m} G(x) \prod_{j=1}^d \beta_j(x)^{\alpha_j} w(x) {\rm d} \mu(x) \leq \sum_{j=1}^d \alpha_j S_j(\beta_j).$$ 
Take $n \geq m$. By \eqref{kjh} and the arithmetic-geometric mean inequality, the left-hand side is at most
$$ \int_{E_m} \prod_{j=1}^d [g_{jn}(x)\beta_j(x)]^{\alpha_j} w(x) {\rm d} \mu(x)
\leq \sum_{j=1}^d \alpha_j \int_{E_m} g_{jn}(x)\beta_j(x)  w(x) {\rm d} \mu(x)$$
$$ \leq \sum_{j=1}^d \alpha_j \int_X  g_{jn}(x)\beta_j(x)  w(x) {\rm d} \mu(x)
= \sum_{j=1}^d \int_X  g_{jn}(x) [\alpha_j \beta_j(x)]  w(x) {\rm d} \mu(x). $$
Thus for all $n \geq m$, 
$$  \sum_{j=1}^d \int_X  g_{jn}(x) [\alpha_j \beta_j(x)]  w(x) {\rm d} \mu(x) \geq  \int_{E_m} G(x) \prod_{j=1}^d \beta_j(x)^{\alpha_j} w(x) {\rm d} \mu(x).$$
Since the simple functions $\beta_j$ are bounded, Lemma~\ref{vector}(b) gives us that
$$ \sum_{j=1}^d S_j (\alpha_j \beta_j) \geq \int_{E_m} G(x) \prod_{j=1}^d \beta_j(x)^{\alpha_j} w(x) {\rm d} \mu(x),$$
which is what we want.

\medskip
\noindent
This completes the proof of Theorem~\ref{thmmain}.
\qed

%%%%%%%%%%%%%%%%%%%%%%%%%%%%%%%%%%%%%%%%%%%%%%%

\part{Connections with other topics}
\bigskip
\section{Complex interpolation and factorisation}\label{factinterp}
We begin by observing that the trivial identity of Example \ref{Loomis--Whitney}, 
$$ \int_{\mathbb{R}^2} f_1(x_2) f_2(x_1) \, {\rm d}x_1  {\rm d}x_2 = \int_\mathbb{R} f_1   \int_\mathbb{R} f_2,$$
immediately implies via Theorem \ref{thmmainbaby} that, for every nonnegative $G \in L^2(\mathbb{R}^2)$, there
exist nonnegative $g_1$ and $g_2$ such that
$$G(x) \leq \sqrt{g_1(x) g_2(x)} \;\; \mbox{  for almost every  }\; x \in \mathbb{R}^2$$
and 
$$ {\rm ess\, sup}_{x_2} \int g_1(x_1, x_2) {\rm d} x_1 \leq \|G\|_2 \; \mbox{   and   } \;  {\rm ess\, sup}_{x_1} \int g_2(x_1, x_2) {\rm d} x_2 \leq \|G\|_2.$$
While it is not perhaps entirely obvious how to do this explicity (a point to which we return in
Sections~\ref{Loomis--Whitneyrevisited}, \ref{NLLW} and \ref{BLfactdetails} below), for now we want to point out 
that this example highlights the connection between our multilinear duality theory and the theory
of interpolation of Banach spaces. In particular, we consider the {\em upper} method of complex interpolation of
A.~P.~Calder\'on, \cite{Calderon}.

\medskip
\noindent
Suppose that $Z_0$ and $Z_1$ are Banach lattices of measurable functions defined on some measure space. 
We define 
$$Z_0^{1-\theta}Z_1^{\theta} = \{f \; : \, \mbox{ there exist } f_j \in Z_j \mbox{ such that } |f| \leq |f_0|^{1-\theta}|f_1|^{\theta}\}$$
with 
$$\|f\|_{Z_0^{1-\theta}Z_1^{\theta}} = \inf\{ \|f_0\|_{Z_0}^{1-\theta} \|f_1\|_{Z_1}^{\theta}\},$$
the inf being taken over all possible decompositions of $f$. Under the assumption that the unit ball of $ Z_0^{1-\theta}Z_1^{\theta}$ is closed in $Z_0 + Z_1$, Calder\'on showed that
$$ Z_0^{1-\theta}Z_1^{\theta} = [Z_0, Z_1]^\theta$$
where $[Z_0, Z_1]^\theta$ is the interpolation space between $Z_0$ and $Z_1$ obtained by the upper complex method.

\medskip
\noindent
With this in mind, the factorisation statement in our example is tantamount to the statement 
$$ L^2(\mathbb{R}^2) \hookrightarrow [L^\infty_{x_1} (L^1_{x_2}),L^\infty_{x_2} (L^1_{x_1})]^{1/2}.$$
Many other special cases of our theory can be similarly expressed in the language of interpolation.
We leave it to the interested reader to pursue this point of view more systematically.

\medskip
\noindent
In this particular example, there is further structure, see for example Pisier \cite{Pis2},
(in which some of the ideas are attributed to Lust-Piquard). There it is established that we have
$$ L^2(\mathbb{R}^2) = HS(L^2(\mathbb{R})) \hookrightarrow \mathcal{L}_{{\rm reg}}(L^2) 
= [L^\infty_{x_1} (L^1_{x_2}),L^\infty_{x_2} (L^1_{x_1})]^{1/2}$$
where $HS$ denotes the class of Hilbert--Schmidt operators and $\mathcal{L}_{{\rm reg}}(L^2)$ is the space
of {\em regular} bounded linear operators on $L^2$. In rough terms, a regular bounded linear operator on $L^2$
is one such that if its kernel is $K(s,t)$, then $|K(s,t)|$ is also the kernel of a bounded linear operator.

\medskip
\noindent
The implicit factorisation arguments involved in establishing results of this type rely on the Hahn--Banach
theorem or the Perron--Frobenius theorem, and are thus related to minimax theory; they are similarly
non-constructive.

\section{Factorisation and convexity}\label{qzf}
It is also natural to enquire about how factorisation and interpolation interact at the level of particular
families of inequalities. For the sake of concreteness, suppose we are in the setting of multilinear
generalised Radon transforms on euclidean spaces -- so that $T_jF_j = F_j \circ B_j$ for suitable $B_j$.
We shall suppress consideration of any of the technical hypotheses of Theorem~\ref{thmmain} in what follows.
Suppose that we have the pair of inequalities
\begin{equation}\label{int1}
\left\|\prod_{j=1}^d T_jF_j\right\|_{L^{q_k}} \leq A_k
\prod_{j=1}^d \left\|F_j\right\|_{L^{p_{jk}}}
\end{equation}
for $k = 0,1$, where $q_k, p_{jk} \geq 1$.

\medskip
\noindent
Each of these has a family of corresponding equivalent factorisation statements, according to Theorem~\ref{thmmain}
and the remarks in Section~\ref{exples}. See also Section~\ref{rrr} below. After some changes of notation,
one such equivalent pair of statements is as follows. For $k = 0,1$, let $s_k : = q_k \sum_j p_{jk}^{-1}$. Then
for all nonnegative
$G_k$ ($k = 0,1$) such that $\int G_k^{s_k'} = 1$, there are nonnegative $g_{10}, \dots , g_{d0}$ and  $g_{11}, \dots , g_{d1}$ such that
\begin{equation}\label{6571}
  G_k(x) \leq \prod_{j=1}^d g_{jk}(x)^{q_k/p_{jk}s_k}\mbox{ a.e.}
  \end{equation}
and such that for all $f_j$ with $\int f_j \leq 1$, 
\begin{equation}\label{5671}
  \int f_j(B_jx) g_{jk}(x) {\rm d} x \leq A_k^{q_k/s_k}
  \end{equation}
for  $k = 0,1$.

\medskip
\noindent
From \eqref{6571} and \eqref{5671} we shall deduce a factorisation statement which implies the
natural interpolation statement
\begin{equation}\label{int2}
\left\|\prod_{j=1}^d T_jF_j\right\|_{L^{q_\theta}} \leq A_0^{1-\theta}A_1^\theta 
\prod_{j=1}^d \left\|F_j\right\|_{L^{p_{j\theta}}}
\end{equation}
for $0 < \theta < 1$, where, as usual, $1/q_\theta = (1-\theta)/q_0 + \theta/q_1$, and similarly for $1/p_{j\theta}$.

\medskip
\noindent
Indeed, given a nonnegative $G$ with $\int G = 1$, let $G_k = G^{1/s_k'}$.  Taking convex combinations in \eqref{6571} gives us
\begin{equation}\label{765}
  G(x) \leq \prod_{j=1}^d g_{j0}(x)^{q_0 s_0'(1-\theta)/p_{j0}s_0} g_{j1}(x)^{q_1 s_1'\theta/p_{j1}s_1} \mbox{ a.e.}
\end{equation}

\medskip
\noindent
Next, we define
$$ \gamma_j(\theta) := \frac{q_0s_0'}{p_{j0} s_0}(1-\theta) +  \frac{q_1s_1'}{p_{j1} s_1} \theta$$
and define $g_{j\theta}$ by
$$ g_{j\theta}^{\gamma_j(\theta)} :=  g_{j0}(x)^{q_0 s_0'(1-\theta)/p_{j0}s_0} g_{j1}(x)^{q_1 s_1'\theta/p_{j1}s_1}.$$
Then, by \eqref{5671}, we have
$$  \int f_j(B_jx) g_{j\theta}(x) {\rm d} x =  \int f_j(B_jx)  g_{j0}(x)^{q_0 s_0'(1-\theta)/p_{j0}s_0 \gamma_j(\theta)}
g_{j1}(x)^{q_1 s_1'\theta/p_{j1}s_1\gamma_j(\theta)} {\rm d} x$$
$$ \leq \left(\int f_j(B_jx) g_{j0}(x) {\rm d} x\right)^{q_0 s_0'(1-\theta)/p_{j0}s_0 \gamma_j(\theta)}
\left(\int f_j(B_jx) g_{j1}(x) {\rm d} x\right)^{q_1 s_1'\theta/p_{j1}s_1 \gamma_j(\theta)}$$
by H\"older's inequality, since $\gamma_j(\theta)$ is defined precisely to ensure the two exponents
on the right hand side here add to $1$.

\medskip
\noindent
Therefore, if $\int f_j \leq 1$, 
$$  \int f_j(B_jx) g_{j\theta}(x) {\rm d} x \leq \left[A_0^{q_0/s_0}\right]^{q_0 s_0'(1-\theta)/p_{j0}s_0 \gamma_j(\theta)}
\left[A_1^{q_1/s_1}\right]^{q_1 s_1'\theta/p_{j1}s_1 \gamma_j(\theta)}.$$

\medskip
\noindent
Now let $\beta_j(\theta) := \lambda(\theta) \gamma_j(\theta)$ where $\lambda(\theta)$ is defined so that
$\sum_{j=1}^d \beta_j(\theta)  = 1$. By the definition of $s_0$ and $s_1$ we have
$$\sum_{j=1}^d \gamma_j(\theta)
= \sum_{j=1}^d\left(\frac{q_0s_0'}{p_{j0} s_0}(1-\theta) +  \frac{q_1s_1'}{p_{j1} s_1} \theta\right) = (1-\theta) s_0' + \theta s_1'.$$
So, we take 
$$ \lambda(\theta) := \frac{1}{(1-\theta) s_0' + \theta s_1'}.$$

\medskip
\noindent
Now, bearing in mind Remark~\ref{weaksuff}, we conclude that
$$\prod_{j=1}^d\left(\int f_j(B_jx) g_{j\theta}(x) {\rm d} x\right)^{\beta_j(\theta)} \leq
\left[A_0^{q_0/s_0}\right]^{\sum_jq_0 s_0'(1-\theta)\beta_j(\theta)/p_{j0}s_0 \gamma_j(\theta)}
\left[A_1^{q_1/s_1}\right]^{\sum_jq_1 s_1'\theta \beta_j(\theta)/p_{j1}s_1 \gamma_j(\theta)}$$
$$ = \left[A_0^{q_0/s_0}\right]^{\lambda(\theta)\sum_jq_0 s_0'(1-\theta)/p_{j0}s_0}
\left[A_1^{q_1/s_1}\right]^{\lambda(\theta)\sum_jq_1 s_1'\theta/p_{j1}s_1}
=  \left[A_0^{q_0/s_0}\right]^{\lambda(\theta)s_0'(1-\theta)}
\left[A_1^{q_1/s_1}\right]^{\lambda(\theta)s_1'\theta}$$
$$ = A_0^{\frac{s_0'q_0 \lambda(\theta)(1-\theta)}{s_0}} A_1^{\frac{s_1'q_1 \lambda(\theta)\theta}{s_1}}
=  \left\{\left[A_0^{\frac{s_0'q_0 \lambda(\theta)(1-\theta)}{s_0}} A_1^{\frac{s_1'q_1 \lambda(\theta)\theta}{s_1}}\right]^{S(\theta)/Q(\theta)}\right\}^{Q(\theta)/S(\theta)}$$
for a certain quantity $S(\theta)/Q(\theta)$ to which we turn our attention next. Indeed we define
this quantity ({\em not} $S(\theta)$, $Q(\theta)$ separately), so that the exponents on $A_0$ and $A_1$
inside the curly brackets sum to $1$. That is,
$$ \frac{Q(\theta)}{S(\theta)}
:= \lambda(\theta) \left(\frac{s_0'q_0 (1-\theta)}{s_0} + \frac{s_1'q_1 \theta}{s_1}\right). $$
Let us define these exponents of $A_0$ and $A_1$ as $1 - \alpha(\theta)$ and $\alpha(\theta)$ respectively;
that is, we define $\alpha(\theta)$ by
$$ \alpha(\theta) := \frac{S(\theta)}{Q(\theta)} \lambda(\theta) \frac{s_1'q_1 \theta}{s_1}.$$

\medskip
\noindent
Next, we want the $\beta_j = \lambda \gamma_j$ to be of the form $\beta_j(\theta)
= \frac{Q(\theta)}{P_j(\theta)S(\theta)}$ for certain $P_j(\theta)$; that is, $\frac{1}{P_j(\theta)} = \frac{S(\theta) \beta_j(\theta)}{Q(\theta)} = \frac{S(\theta) \lambda(\theta) \gamma_j(\theta)}{Q(\theta)}$.
So, bearing in mind the definitions of $\gamma_j$ and $S/Q$, we define $P_j(\theta)$ by
$$\frac{1}{P_j(\theta)} := \frac{\frac{q_0s_0'}{p_{j0} s_0}(1-\theta) +  \frac{q_1s_1'}{p_{j1} s_1} \theta}
{\frac{s_0'q_0 (1-\theta)}{s_0} + \frac{s_1'q_1 \theta}{s_1}}.$$

\medskip
\noindent
Finally, we define $Q(\theta)$ by
$$ \frac{1}{Q(\theta)} := (1 - \alpha(\theta))\frac{1}{q_0} + \alpha(\theta) \frac{1}{q_1}.$$
It is not hard to check that with all these definitions in place, we have, for each $j$,
$$ \frac{1}{P_j(\theta)} = (1 - \alpha(\theta))\frac{1}{p_{j0}} + \alpha(\theta) \frac{1}{p_{j1}}.$$

\medskip
\noindent
We therefore have that for each $ 0 \leq \theta \leq 1$, for all $G_\theta = G^{1/S'(\theta)}$ such that $\int G_\theta^{S'(\theta)} = 1$,
there exist $g_{j\theta}$ such that 
$$G_\theta(x) \leq \prod_{j=1}^d g_{j\theta}(x)^{Q(\theta)/P_{j}(\theta)S(\theta)}$$
and, for $f_j$ such that $\int f_j \leq 1$,
$$ \prod_{j=1}^d \left(\int f_j(B_jx) g_{j \theta}(x) {\rm d}x\right)^{Q(\theta)/P_j(\theta)S(\theta)}
\leq  \left(A_0^{1 - \alpha(\theta)} A_1^{\alpha(\theta)}\right)^{Q(\theta)/S(\theta)}.$$
Note particularly that the exponents $Q(\theta)/P_{j}(\theta)S(\theta)$ sum to $1$ since $\sum_{j=1}^d \beta_j = 1$.

\medskip
\noindent
Consequently, using the flexibility that Remark~\ref{weaksuff} affords us,
\begin{equation*}
\left\|\prod_{j=1}^d T_jF_j\right\|_{L^{Q(\theta)}} \leq A_0^{1-\alpha(\theta)}A_1^{\alpha(\theta)} 
\prod_{j=1}^d \left\|F_j\right\|_{L^{P_{j}(\theta)}}
\end{equation*}
for $0 < \theta < 1$. Noting that the map $\alpha: [0,1] \to [0,1]$ is a surjection
completes the argument proving \eqref{int2}.

\medskip
\noindent
The argument given here provides no insight into cases in which \eqref{int2} might hold with a smaller
constant than $A_0^{1 - \theta} A_1^\theta$.

\subsection{Factorisation and multiple manifestations of generalised Radon transforms}\label{rrr}
As we have observed in Section~\ref{exples} there may be multiple equivalent manifestations
of the same multilinear inequality. For concreteness, suppose that we are once again considering multilinear
generalised Radon transforms on euclidean spaces so that
$T_j f = f \circ B_j$ for suitable $B_j$. Then the two inequalities
\begin{equation*}
\left\|\prod_{j=1}^d (T_jf_j)^{\alpha_j}\right\|_q
\leq A \prod_{j=1}^d \Big\|f_j\Big\|_{p_j}^{\alpha_j}
\end{equation*}
and 
\begin{equation*}
\left\|\prod_{j=1}^d (T_j\tilde{f}_j)^{\tilde{\alpha}_j}\right\|_{\tilde{q}}
\leq \tilde{A} \prod_{j=1}^d \Big\|\tilde{f}_j\Big\|_{\tilde{p}_j}^{\tilde{\alpha}_j}
\end{equation*}
(where we are imposing $\sum_{j=1}^d \alpha_j = 1 = \sum_{j=1}^d \tilde{\alpha}_j$)
are clearly equivalent provided that $A^q = \tilde{A}^{\tilde{q}}$ and
$\alpha_j \tilde {p}_j/\tilde{\alpha}_j p_j = \tilde{q}/q$
for all $j$. The corresponding factorisation statements 

\bigskip
{\em For all nonnegative $G \in L^{q^\prime}$ there exist nonnegative locally integrable functions 
$g_j$ such that} 
\begin{equation*}
G(x) \leq \prod_{j=1}^d g_j(x)^{\alpha_j}\qquad\mbox{a.e.}
\end{equation*}
{\em and such that for each $j$, for all $f_j \in L^{p_j}$,} 
\begin{equation*}
\int g_j(x)f_j(B_jx) {\rm d} x \leq A\|G\|_{q^{\prime}}\|f_j\|_{p_j}.
\end{equation*}
and

\bigskip
\noindent
{\em For all nonnegative $\tilde{G} \in L^{\tilde{q}^\prime}$ there exist nonnegative locally integrable functions 
$\tilde{g}_j$ such that} 
\begin{equation*}
\tilde{G}(x) \leq \prod_{j=1}^d \tilde{g}_j(x)^{\tilde{\alpha}_j}\qquad\mbox{a.e.}
\end{equation*}
{\em and such that for each $j$, for all $\tilde{f}_j \in L^{\tilde{p}_j}$,} 
\begin{equation*}
\int \tilde{g}_j(x)\tilde{f}_j(B_jx) {\rm d}x \leq \tilde{A}\|\tilde{G}\|_{L^{\tilde{q}^{\prime}}}\|\tilde{f}_j\|_{\tilde{p}_j}.
\end{equation*}

\bigskip
are therefore also equivalent (subject to suitable hypotheses), by  Proposition~\ref{thmguthbaby}
and Theorem~\ref{thmmainbaby}. However it is not immediately apparent whether this equivalence can be seen directly
via changes of notation coupled with simple convexity arguments. In this connection the remarks in Section~5.7 of
\cite{MR2061575} may be helpful.

\section{Factorisation and more general multilinear operators}\label{culture}\label{fdmgs}
The multilinear operators we have considered have a rather special form in so far as they are built out of a collection
of positive linear operators by taking a pointwise geometric mean. One may ask to what extent the theory we
have developed is valid for more general multilinear operators
$T : \mathcal{Y}_1 \times \dots \times \mathcal{Y}_d \to \mathcal{X}$. In such a setting we will no longer be able to
attribute different ``weights'' $\alpha_j$ to the different components $\mathcal{Y}_j$, and all of them will need to
be treated on an equal footing.

\medskip
\noindent
For a nonnegative kernel $K$, let us therefore consider multilinear operators of the form
$$ T(f_1, \dots , f_d)(x) = \int_{Y_d} \dots \int_{Y_1} K(x, y_1, \dots , y_d) f_1(y_1) \dots f_d(y_d) {\rm d} \nu_1(y_1) \dots {\rm d} \nu_d(y_d)$$
and inequalities of the form
\begin{equation}\label{newmultilinear}
    \left\|T(f_1,\dots,f_d)^{1/d}\right\|_{\mathcal{X}}\leq A
    \prod_{j=1}^d \|f_j\|_{\mathcal{Y}_j}^{1/d}.
\end{equation}

\medskip
\noindent
When $K$ is of the form $K(x, y_1, \dots , y_d) = K_1(x, y_1) \dots K_d(x, y_d)$, these are the special cases
of the inequalities \eqref{mainineq} given by $\alpha_j=1/d$ for $j=1,\dots,d$.
It is very natural to ask whether there is a general duality/factorisation result
along the same lines as Proposition~\ref{thmguth} and Theorem~\ref{thmmain} which yields a
necessary and sufficient condition for the validity of inequality \eqref{newmultilinear}.

\medskip
\noindent
As the reader will readily verify by following the proof of Proposition~\ref{thmguth},
inequality \eqref{newmultilinear} does indeed hold (under hypotheses on $\mathcal{X}$ and $\mathcal{Y}_j$
similar to those of Proposition~\ref{thmguth}),
if, for all $G\in \mathcal{X}^\prime$ such that
$\|G\|_{\mathcal{X}^\prime}\leq 1$, we have that there there exist nonnegative functions $g_j$ on $X\times Y_j$ such that
\begin{equation}\label{newfactorisation}
\begin{aligned}
    K(x,y_1,\dots,y_d)^{1/d}G(x)&\leq \prod_{j=1}^d g_j(x,y_j)^{1/d}\quad\text{a.e.}\\
    \mbox{and }\left\|\int_X g_j(x,\cdot) d\mu(x)\right\|_{\mathcal{Y}_j^\ast}&\leq A.
\end{aligned}
\end{equation}

\medskip
\noindent
This observation has proved very useful in multilinear Kakeya theory, see Section~\ref{MKrevisited} below.

\medskip
\noindent
However, the converse is not true, namely inequality \eqref{newmultilinear} does not in general imply
the existence of $S_j$ such that \eqref{newfactorisation} holds even if we assume that the integral kernel $K$ is invariant under permutations of the $y$-variables:

\begin{prop}
Let $d = 2$. Let $X=Y_1=Y_2 = \{1,2\} = \Omega$ with counting measure, 
$\mathcal{X}=L^4(\Omega)$, and $\mathcal{Y}_1=\mathcal{Y}_2=L^2(\Omega)$. There exists a bilinear
$T:  L^2(\Omega) \times L^2(\Omega) \to L^4(\Omega)$ such that \eqref{newmultilinear} holds with $A = 2^{1/4}$ but such that
\eqref{newfactorisation} can only hold with $A \geq 2^{1/2}$. 
\end{prop} 

\begin{proof}
Let the integral kernel $K$ of $T$ satisfy
\begin{displaymath}
    K(1,1,1)=K(2,1,1)=K(2,2,2)=1
\end{displaymath}
and let $K$ equal zero otherwise. Let $f_1=(a_1,a_2)$ and $f_2=(b_1,b_2)$ and we assume
$a_1^2+a_2^2=b_1^2+b_2^2=1$.
Then
\begin{displaymath}
    T(f_1,f_2)(1) = a_1b_1\quad\text{and}\quad
    T(f_1,f_2)(2) = a_1b_1+a_2b_2
\end{displaymath}
so
\begin{displaymath}
    \int T(f_1,f_2)(x)^2 {\rm d}x = (a_1b_1)^2+(a_1b_1+a_2b_2)^2.
\end{displaymath}
This is clearly maximised, subject to the normalisation constraints,
by taking $a_1=b_1=1$, $a_2=b_2=0$ and the maximum is $2$.
So we see that the multilinear inequality \eqref{newmultilinear} holds for this operator with $A=2^{1/4}$.

\medskip
For problem \eqref{newfactorisation}, consider $G=(0,1)$. Then the non-trivial constraints
are
\begin{displaymath}
    1\leq \sqrt{g_1(2,1)g_2(2,1)}
    \quad\text{and}\quad
    1\leq \sqrt{g_1(2,2)g_2(2,2)}
\end{displaymath}
and
\begin{gather*}
    \sqrt{(g_1(1,1)+g_1(2,1))^2+(g_1(1,2)+g_1(2,2))^2}\leq A\\
    \sqrt{(g_2(1,1)+g_2(2,1))^2+(g_2(1,2)+g_2(2,2))^2}\leq A.
\end{gather*}
Using $uv\leq (u^2+v^2)/2$ on the lower bounds gives
\begin{displaymath}
    1\leq (g_1(2,1)^2+g_2(2,1)^2)/2
    \quad\text{and}\quad
    1\leq (g_1(2,2)^2+g_2(2,2)^2)/2,
\end{displaymath}
so
\begin{displaymath}
    {2}\leq g_1(2,1)^2+g_2(2,1)^2
    \quad\text{and}\quad
    {2}\leq g_1(2,2)^2+g_2(2,2)^2,
\end{displaymath}
so
\begin{displaymath}
    4\leq g_1(2,1)^2+g_2(2,1)^2+ g_1(2,2)^2+g_2(2,2)^2,
\end{displaymath}
and thus
\begin{displaymath}
    2\leq\max\{g_1(2,1)^2+g_1(2,2)^2,g_2(2,1)^2+g_2(2,2)^2\}
\end{displaymath}
giving $A \geq 2^{1/2}$, which is strictly larger than $2^{1/4}$. So while inequality \eqref{newmultilinear} holds
in this case, there are $G$ for which there are {\em no} $g_j$ satisfying \eqref{newfactorisation} with the same value of $A$.
\end{proof}

\medskip
\noindent
We invite the reader to use this idea to construct examples where \eqref{newmultilinear} holds with $A=1$ but
for which \eqref{newfactorisation} holds for no finite $A$.

\medskip
\noindent
See \cite{GT} for a different approach to inequalities of the form \eqref{newmultilinear}, based upon
considerations related to Schur's lemma rather than duality.

%%%%%%%%%%%%%%%%%%%%%%%%%%%%%%%%%%%%%%%%%%
\part{Examples and illustrations of the theory}
\bigskip
\noindent
In this part we revisit the examples in the introduction which motivated our study. We examine what
insights our duality--factorisation results bring to, and have gained from, each of them.
In some cases we reap the benefits of more direct and streamlined factorisation-based proofs
of known inequalities. In others, an interesting challenge is posed -- it can be argued that we cannot
claim to have a full understanding of an inequality until we can exhibit its equivalent
factorisation statement.

\section{Classical inequalities revisited}\label{classical}
\subsection{H\"older's inequality}\label{Holderrevisited}
We observed above that the multilinear form of H\"older's inequality for nonnegative functions is equivalent,
    for any fixed set of exponents $\alpha_j > 0$ with $\sum_{j=1}^d \alpha_j = 1$, to
$$ \| f_1^{\alpha_1} \cdots f_d^{\alpha_d} \|_q \leq  \|f_1 \|_{q_1}^{\alpha_1} \cdots\|f_d \|_{q_d}^{\alpha_d}$$
    for any choice of indices $ 1 \leq q_j < \infty$ and $1 \leq q < \infty$ which satisfies
    $\sum_{j=1}^d \alpha_j q_j^{-1} = q^{-1}$. 

\medskip
\noindent By Theorem~\ref{thmmain}, each instance of this inequality is equivalent to the existence of a
subfactorisation of any $ G \in L^{q'}$ as
$$ G(x) \leq \prod_{j=1}^d g_j(x)^{\alpha_j} \; \; \mbox{ a.e.}$$
where
$$\|g_j \|_{{q_j}^\prime} \leq \|G\|_{q'}.$$

\medskip
\noindent
Taking $g_j = \lambda_j G^{\gamma_j}$ for appropriate $\lambda_j$ and $\gamma_j$ verifies this.
In particular, if we take $q_j = q \geq 1$ for all $j$, then we can simply take $g_j = G$ for all $j$.

\subsection{The affine-invariant Loomis--Whitney inequality}\label{Loomis--Whitneyrevisited}
Recall that the Loomis--Whitney inequality is
$$ |\int_{\mathbb{R}^n} F_1(\pi_1 x) \cdots F_n(\pi_n x) \, {\rm d} x| 
\leq \|F_1\|_{L^{n-1}(\mathbb{R}^{n-1})} \cdots \|F_n\|_{L^{n-1}(\mathbb{R}^{n-1})},$$
where $\pi_j x = (x_1, \dots, \widehat{x_j}, \dots, x_n)$ is projection onto the hyperplane perpendicular
to the $j$'th standard basis vector $e_j$.
For every $0 < p < \infty $ this is equivalent to the inequality
$$ \|f_1(\pi_1 x)^{1/n} \cdots f_n(\pi_n x)^{1/n} \|_{L^{np/(n-1)}(\mathbb{R}^n)} 
\leq \|f_1\|_{L^{p}(\mathbb{R}^{n-1})}^{1/n} \cdots \|f_n\|_{L^{p}(\mathbb{R}^{n-1})}^{1/n}.$$
Each of these inequalities with $p \geq 1$ falls under the scope of our theory.

\medskip
\noindent
For example when $p =1$ we have the equivalent formulation
\begin{equation*}
\int_{\mathbb{R}^n} \prod_{j=1}^n f_j(\pi_{j} x)^{1/(n-1)} {\rm d}x \leq
\prod_{j=1}^n \left(\int_{\mathbb{R}^{n-1}} f_j\right)^{1/(n-1)}.
\end{equation*}

\medskip
\noindent
More generally, if $\pi_{\omega^\perp}$ represents orthogonal
projection onto the hyperplane perpendicular to $\omega \in
\mathbb{S}^{n-1}$, we have the affine-invariant Loomis--Whitney inequality
\begin{equation}\label{afflw}
\int_{\mathbb{R}^n} \prod_{j=1}^n f_j(\pi_{\omega_j^\perp} x)^{1/(n-1)} {\rm d}x \leq
(\omega_1 \wedge \dots \wedge
\omega_n)^{-1/(n-1)}
\prod_{j=1}^n \left(\int_{\mathbb{R}^{n-1}} f_j\right)^{1/(n-1)},
\end{equation}
where $(\omega_1 \wedge \dots \wedge
\omega_n)^{-1/(n-1)}$ is the best constant in the inequality. Here, $\omega_1 \wedge \dots \wedge \omega_n$ 
is the modulus of the determinant of the matrix whose columns are $\omega_1, \dots, \omega_n$, and it is the 
volume of the parallepiped whose sides are given by the vectors $\omega_j$.
(Clearly if we choose all the $\omega_j$ to be the same we cannot expect a finite constant, and the constant 
in general should reflect ``quantitative linear independence'' of the $\omega_j$.)

\medskip
\noindent
We give a direct and elegant proof of \eqref{afflw} by explicitly establishing a
suitable factorisation. Indeed, according to Proposition~\ref{thmguthbaby}, it is sufficient that for
every nonnegative $G \in L^n(\mathbb{R}^n)$ we can find
$g_1, \dots, g_n$ such that
$$ G(x) = g_1(x)^{1/n} \cdots g_n(x)^{1/n} \mbox{ a.e.}$$
and, for all $j$ and almost every $x$,
$$ \int g_j(x + t \omega_j) {\rm d}t = (\omega_1 \wedge \dots \wedge
\omega_n)^{-1/n} \|G\|_n.$$
This is because for any $f : \mathbb{R}^{n-1} \to \mathbb{R}$ and $g: \mathbb{R}^n \to \mathbb{R}$, writing
$x \in \mathbb{R}^n$ as $x = u + t \omega_j$ with $u \in \omega_j^\perp$, we have
$$\int f(\pi_{\omega_j^\perp} x) g(x) {\rm d}x 
= \int_{\mathbb{R}^{n-1}} \int_\mathbb{R} f(\pi_{\omega_j^\perp}(u + t \omega_j))g(u + t \omega_j) {\rm d}t {\rm d}u  $$
$$ = \int_{\mathbb{R}^{n-1}} f(u) \left(\int_{\mathbb{R}} g(u + t \omega_j) {\rm d}t \right){\rm d}u.$$

\medskip
\noindent
Let $G :\mathbb{R}^n \to \mathbb{R}$ be a nonnegative function which satisfies 
$\int_{\mathbb{R}^n} G(x)^n {\rm d}x = 1$. For $\omega_1, \dots, \omega_n \in \mathbb{S}^{n-1}$ and 
$\xi \in \mathbb{R}^n$ let us first note that if we set, for $s = (s_1, \dots , s_n) 
\in \mathbb{R}^n$, 
$$y(s) = \xi + s_1 \omega_1 + \dots + s_{n-1} \omega_{n-1} + s_n \omega_n,$$ 
then we have that the Jacobian map $\partial y / \partial s $ satisfies
$$ |\det \left(\partial y / \partial s \right)| = \omega_1 \wedge \dots  \wedge \omega_n.$$

\medskip
\noindent
Therefore, for every $\xi \in \mathbb{R}^n$,
$$\int G(\xi + s_1 \omega_1 + \dots + s_{n-1} \omega_{n-1} + s_n \omega_n)^n {\rm d}s_1 {\rm d}s_2 \dots {\rm d}s_{n}
= \int G(y(s))^n {\rm d}s $$
$$ = \int G(y)^n \frac{1}{ |\det \left(\partial y / \partial s \right)|} {\rm d}y = 
(\omega_1 \wedge \dots  \wedge \omega_n)^{-1}.$$

\medskip
\noindent
Secondly, $G(x)^n$ can be written (for a.e. $x$) as a telescoping product
\begin{eqnarray*}
\begin{aligned}
& \frac{G(x)^n}{\int G(x + s_1 \omega_1)^n {\rm d}s_1} \times \frac{\int G(x
    + s_1 \omega_1)^n
    {\rm d}s_1}{\int G(x + s_1 \omega_1 + s_2 \omega_2)^n {\rm d}s_1 {\rm d}s_2} \times
    \dots \\
& \times \frac{\int G(x + s_1 \omega_1 + \dots + s_{n-1} \omega_{n-1})^n{\rm d} s_1
     {\rm d}s_2 \dots {\rm d}s_{n-1}}{\int G(x + s_1 \omega_1 + \dots + s_{n-1}
    \omega_{n-1} + s_n \omega_n)^n {\rm d}s_1
   {\rm d}s_2 \dots {\rm d}s_{n}} \times \left(\omega_1 \wedge \dots \wedge \omega_n\right)^{-1}
\end{aligned}
\end{eqnarray*}
\medskip
$$ := g_1(x) \dots g_n(x)$$ 
where 
$$ g_j(x) = \frac{\int G(x + s_1 \omega_1 + \dots + s_{j-1} \omega_{j-1})^n {\rm d}s_1
    {\rm d}s_2 \dots {\rm d}s_{j-1}}{\int G(x + s_1 \omega_1 + \dots + s_{j-1}
    \omega_{j-1} + s_j \omega_j)^n{\rm d}s_1
    {\rm d}s_2 \dots {\rm d}s_{j}} \times \left(\omega_1 \wedge \dots \wedge \omega_n\right)^{-1/n}.$$

\medskip
\noindent
If we replace $x$ by $x + t \omega_j$ in this formula, the denominator is unchanged, and so if we 
then integrate with respect to $t$ we immediately see that 
$$\int g_j(x + t \omega_j) {\rm d} t = \left(\omega_1 \wedge \dots \wedge \omega_n\right)^{-1/n}$$
identically for $x \in \mathbb{R}^n$, as we needed.

\medskip
\noindent
A similar approach works when we instead consider projections onto subspaces
whose codimensions sum to $n$. Indeed, suppose that we have subspaces $E_j$
of $\mathbb{R}^n$ with $\dim E_j = k_j$ and $\sum_{j=1}^d k_j = n$ and 
assume that $\mathbb{R}^n = E_1 + \dots + E_d$ as an algebraic direct sum.

\medskip
\noindent
We identify a quantity which measures lack of orthogonality of these subspaces in the same way that the
wedge product $\omega_1 \wedge \dots \wedge \omega_n$ measures the degeneracy
in the directions $\omega_1, \dots , \omega_n \in \mathbb{S}^{n-1}$. Let
$\{e_{j1}, e_{j2}, \dots , e_{jk_j}\}$ be an orthonormal basis for $E_j$ and define
$$E_1 \wedge \dots \wedge E_d := \wedge_{j=1}^d \wedge_{k=1}^{k_j} e_{jk};$$
that is, $E_1 \wedge \dots \wedge E_d$ is the absolute value of the determinant of the 
$n \times n$ matrix whose $j$'th block of $k_j$ columns comprises an orthonormal basis for $E_j$.
It is easily checked that this quantity is independent of the particular orthonormal bases 
chosen, and it can of course be defined in a more canonical and invariant way.

\begin{prop}
For $E_j$ as above, let $\pi_j$ be the projection whose kernel is $E_j$. Then we have the
affine-invariant $k_j$-plane Loomis--Whitney inequality:
\begin{equation}\label{aikjplw}
\begin{aligned}
&\int_{\mathbb{R}^n} f_1(\pi_1 x)^{1/(d-1)} \dots f_d(\pi_d x)^{1/(d-1)} {\rm d}x \\
&\leq
  \left(E_1 \wedge \dots \wedge E_d\right)^{-1/(d-1)}
\left(\int f_1 \right)^{1/(d-1)} \dots \left(\int f_d \right)^{1/(d-1)}.
\end{aligned}
\end{equation}
\end{prop}

\medskip
\noindent
The proof via factorisation is formally the same as the case when $k_j = 1$ for all $j$,
where now the roles of the variables $s_j \in \mathbb{R}^1$ are replaced by copies of 
$\mathbb{R}^{k_j}$. We leave the details to the reader.

\medskip
\noindent
In the special case of the trivial identity,
$$ \int_{\mathbb{R}^2} F_1(x_2)F_2(x_1) \; {\rm d} x_1  {\rm d} x_2 = \int_\mathbb{R} F_1 \int_\mathbb{R} F_2,$$
(see Section~\ref{factinterp}), a suitable factorisation of $G \in L^2(\mathbb{R}^2)$ with $\|G\|_2 = 1$ is given by $G(x)^2 = g_1(x) g_2(x)$ a.e., where
$$g_1(x_1, x_2) = \frac{G(x_1, x_2)^2}{\int_\mathbb{R} G(s, x_2)^2 {\rm d} s}$$
and
$$g_2(x_1, x_2) = \int_\mathbb{R} G(s, x_2)^2 {\rm d} s.$$

\medskip
\noindent
Note that this factorisation depends upon the order we have assigned to $\{1,2\}$. 
On the other hand, given this ordering, the essentially unique way to write
$$ G(x)^2 = g_1(x_1, x_2)g_2(x_2)$$
where $\|g_1(\cdot, x_2)\|_1 = 1$ for all $x_2$ and $\|g_2\|_1 = 1$ is as we have given.
See Section \ref{BLfactdetails}, where this observation drives related issues.

\medskip
\noindent
There are many variants of the Loomis--Whitney inequality -- for example Finner's inequalities \cite{Fi}
-- which can likewise be established by the same factorisation method.

\subsection{The nonlinear Loomis--Whitney inequality}\label{mgrtrevisited}\label{NLLW}
Nonlinear Loomis--Whitney inequalities (and some multilinear generalised Radon transforms)
can likewise be established by similar methods. In fact the first proof of the nonlinear
Loomis--Whitney inequality with essentially the sharp constant was obtained via an explicit
factorisation technique. We give the details.

\medskip
\noindent
Let $V$ be an open neighbourhood of $0$ in $\mathbb{R}^n$ and
$U$ an open neighbourhood of $0$ in $\mathbb{R}^{n-1}$. Let $\pi: V \to U$ be a $C^1$ submersion onto $U$, and
for $x \in V$ let $\omega(x)$ be the wedge product of the rows of
${\rm d} \pi(x)$. We assume that the fibres $\pi^{-1}(u)$ for $u \in U$
can be parametrised by $C^1$ curves $t \mapsto \gamma(t,x)$ in such a way that

\begin{itemize}
\item for all $x \in V $, $\gamma(0,x) = x$
\item for all $x \in V $, for all $t$,  $\pi \gamma(t,x) = \pi x$
\item (semigroup property) for all $x \in V $, for all $t$ and $s$,
$$\gamma(t, \gamma(s,x)) = \gamma(s+t, x)$$
\item for all $x$ and $t$,
$\frac{\rm d}{{\rm d}t} \gamma(t,x) = \omega(\gamma(t,x))$.
\end{itemize}

\medskip
\noindent
The domain of each curve $\gamma(\cdot, x)$ will be an open interval $I_x$ containing $0$
which we largely suppress in what follows, but we stress that $\gamma(I_x,x)$ is the entire fibre
containing $x$. In all the $t$-integrals 
below it is assumed that we are integrating over such maximal domains.

\medskip
\noindent
We note that under these assumptions, especially the last one, the co-area formula gives 
$$ \int_V f(\pi x) g(x) {\rm d}x = \int_U f(u) \left(\int g(\gamma(t, \tilde{u})) {\rm d} t \right){\rm d}u$$
for any reasonable functions $f$ and $g$.

\medskip
\noindent
We now assume that we have $n$ submersions $\pi_1, \dots, \pi_n$ as above, and we assume that $\omega_1(0) \wedge
\dots \wedge \omega_n(0) \neq 0$. For each $x \in V$ we define the maps ${\bf t} \mapsto \Phi_x({\bf t})$ by
$$ \Phi_x : (t_1, \dots, t_n) \mapsto \gamma_1(t_1, \gamma_2(t_2, \dots ,\gamma_n(t_n,x)) \dots )$$
which satisfy $\Phi_x(0) = x$ and also
$$ |\det( D \Phi_x) (0)| = (\omega_1 \wedge \dots \wedge
\omega_n)(x) \neq 0$$
provided $x$ is sufficiently close to $0$.

\medskip
\noindent
We shall assume that $V$ is sufficiently small so that for each $x \in V$, the map $ \Phi_x $ is
injective -- as was pointed out in \cite{BCW}, 
even in two dimensions some global hypothesis of this sort is needed.

\medskip
\noindent
With the set-up above, for $x \in V$ let 
\begin{equation}\label{WWW}
  W(x) := \inf_{\xi \in V} \det |(D \Phi_\xi)(\Phi_\xi^{-1}(x))|.
\end{equation}
Note that $W(x) \leq \omega_1(x) \wedge \dots \wedge \omega_n(x)$,
(take $\xi = x$), and that $W(\Phi_x({\bf t})) \leq |\det(D \Phi_x)({\bf t})|$ for all $x$ and ${\bf t}$.

\medskip
\noindent
For $1 \leq j \leq n$ and suitable $F$ let 
$$ S_j(x) = \frac{\int \dots \int F(\gamma_1(t_1, \gamma_2(t_2, \dots , \gamma_{j-1}(t_{j-1},x)) \dots )) {\rm d}t_{j-1} \dots {\rm d}t_{1}}
{\int \dots \int F(\gamma_1(t_1, \gamma_2(t_2, \dots , \gamma_{j}(t_{j},x)) \dots )) {\rm d}t_{j} \dots {\rm d}t_{1}}$$
(so that $S_1$ has no integrals in the numerator).

\medskip
\noindent
Then we have
$$ S_j( \gamma_j (\tau, x)) 
= \frac{\int \dots \int F(\gamma_1(t_1, \gamma_2(t_2, \dots ,\gamma_{j-1}(t_{j-1},\gamma_j (\tau, x))) \dots )) {\rm d}t_{j-1} \dots {\rm d}t_{1}}
{\int \dots \int F(\gamma_1(t_1, \gamma_2(t_2, \dots ,\gamma_{j}(t_{j},\gamma_j (\tau, x))) \dots )) {\rm d}t_{j} \dots {\rm d}t_{1}}.$$

\medskip
\noindent
We claim that for each $j$ and each $x$,
$$\int S_j( \gamma_j (\tau, x)) {\rm d} \tau =1.$$

\medskip
\noindent
Indeed, notice that the denominator in the previous expression, 
$$\int \dots \int F(\gamma_1(t_1, \gamma_2(t_2, \dots ,\gamma_{j}(t_{j},\gamma_j (\tau, x))) \dots )) {\rm d}t_{j} \dots {\rm d}t_{1},$$
equals
$$ \int \dots \int F(\gamma_1(t_1, \gamma_2(t_2, \dots ,\gamma_{j}(t_{j} + \tau , x )) \dots )) {\rm d}t_{j} \dots {\rm d}t_{1}$$
$$ = \int \dots \int F(\gamma_1(t_1, \gamma_2(t_2, \dots ,\gamma_{j}(t_{j}, x )) \dots )) {\rm d}t_{j} \dots {\rm d}t_{1}$$
by the semigroup property, and is therefore independent of $\tau$. So
$$ \int S_j( \gamma_j (\tau, x)) {\rm d} \tau = 
\frac{\int \dots \int F(\gamma_1(t_1, \gamma_2(t_2, \dots ,\gamma_{j-1}(t_{j-1},\gamma_j (\tau, x))) \dots )) {\rm d}t_{j-1} \dots {\rm d}t_{1} {\rm d} \tau}
{\int \dots \int F(\gamma_1(t_1, \gamma_2(t_2, \dots ,\gamma_{j}(t_{j}, x )) \dots )) {\rm d}t_{j} \dots {\rm d}t_{1}}$$
which equals $1$ by Fubini's theorem.

\medskip
\noindent
On the other hand,
$$ \prod_{j=1}^n S_j(x) 
= \frac{F(x)}{\int \dots \int F(\gamma_1(t_1, \gamma_2(t_2, \dots ,\gamma_{n}(t_{n},x)) \dots )) {\rm d}t_{n} \dots {\rm d}t_{1}},$$
so that
$$F(x) = \prod_{j=1}^n S_j(x)  \int F(\Phi_x({\bf t})){\rm d} {\bf t}.$$

\medskip
\noindent
Taking $F(x) = S(x)^n W(x)$, we therefore have
$$ S(x)^n W(x) = \prod_{j=1}^n S_j(x)  \int S(\Phi_x({\bf t}))^n W(\Phi_x({\bf t})) {\rm d} {\bf t}$$
$$ \leq \prod_{j=1}^n S_j(x)  \int S(\Phi_x({\bf t}))^n |\det (D \Phi_x)({\bf t})| {\rm d} {\bf t}
= \prod_{j=1}^n S_j(x) \int_V S(y)^n {\rm d} y $$
since $W(\Phi_x({\bf t})) \leq |\det (D \Phi_x)({\bf t})|$ for all ${\bf t}$ and since 
each $\Phi_x$ is injective. We also have that for each $j$ and each $x$,
$$ \int_V f(\pi_j x) S_j(x) {\rm d}x =
\int_{U_j} f(u)\left(\int S_j( \gamma_j (\tau, \tilde{u})) {\rm d} \tau \right) {\rm d} u
= \int_{U_j} f.$$

  \medskip
  \noindent
  By the easy half of the duality argument, this shows that for all nonnegative $f_j \in L^1(U_j)$
  we have
  $$\| \prod_{j=1}^n f_j(\pi_j x)^{1/n} W(x)^{1/n} \|_{L^{n/(n-1)}(V)} \leq \prod_{j=1}^n
  \left(\int_{U_j} f_j \right)^{1/n}.$$
Consequently we have:

  \begin{prop}\label{FFive}
Under the above assumptions, with $W$ defined as in \eqref{WWW}, we have 
   $$ \int_V \prod_{j=1}^n f_j(\pi_j x)^{1/(n-1)} W(x)^{1/(n-1)} \; {\rm d} x \leq \prod_{j=1}^n
  \left(\int_{U_j} f_j \right)^{1/(n-1)}.$$
\end{prop}

\medskip
  \noindent
  Noting that $W(x) \leq \omega_1(x) \wedge \dots \wedge \omega_n(x)$, one might ask whether
  $$ \int_V \prod_{j=1}^n f_j(\pi_j x)^{1/(n-1)} \omega_1(x) \wedge \dots \wedge \omega_n(x)^{1/(n-1)} \; {\rm d} x \leq \prod_{j=1}^n
  \left(\int_{U_j} f_j \right)^{1/(n-1)}$$
  holds for sufficiently small $V$.

  \medskip
  \noindent
  As an immediate corollary of Proposition~\ref{FFive}, we obtain the sharp form of the
  nonlinear Loomis--Whitney inequality of \cite{BCW}:

  \begin{cor}
    Let $V$ be an open neighbourhood of $0$ in $\mathbb{R}^{n}$ and $U$ an open
    neighbourhood of $0$ in $\mathbb{R}^{n-1}$. For $1 \leq j \leq n$,
    let $\pi_j: V \to U$ be $C^1$ submersions onto $U$, and for $x \in V$ let $\omega_j(x)$
    be the wedge product of the rows of ${\rm d} \pi_j$. Assume that $ \omega_1(0) \wedge \dots \wedge \omega_n(0) \neq 0$. Then, for all $\epsilon > 0$, there is a neighbourhood $V' \subseteq V$ of $0$, such that for all $f_j$
 $$ \int_{V'} \prod_{j=1}^n f_j(\pi_j x)^{1/(n-1)} \; {\rm d} x \leq (1+\epsilon)
  \omega_1(0) \wedge \dots \wedge \omega_n(0)^{-1/(n-1)} \prod_{j=1}^n \left(\int_{U_j} f_j \right)^{1/(n-1)}$$
  \end{cor}

  \begin{proof}
Given $\epsilon >0$ we can choose $V'$ sufficiently small that $\omega_1(0) \wedge \dots \wedge \omega_n(0)
  \leq (1+ \epsilon)^{n-1}W(x)$ for all
  $x \in V'$.
\end{proof}
  
  Since the work in this section was presented in various public fora, Bennett et al
  have shown, using methods based on induction on scales, that any
 Brascamp--Lieb inequality has a corresponding nonlinear counterpart with the same loss in
the constant of at most $(1 + \epsilon)$. See \cite{Betal}.
  
\section{Brascamp--Lieb inequalities revisited}\label{BLrevisited}
We shall discuss the Brascamp--Lieb inequalities under two headings. Firstly we shall address geometric Brascamp--Lieb
inequalities (where in particular we can identify the sharp constant and existence of Gaussian extremisers),
and secondly we will examine general Brascamp--Lieb inequalities with a finite (but unquantified) constant. 

\subsection{Geometric Brascamp--Lieb inequalities}\label{GBL}
The next result is a direct application of Theorem~\ref{thmmainbaby} to the geometric Brascamp--Lieb inequalities
    of Example~\ref{BL}. 

    \begin{thm}
      For $1 \leq j \leq d$ let $V_j$ be a subspace of $\mathbb{R}^n$.
      Let $B_j: \mathbb{R}^n \to V_j$ be orthogonal projection.
       Suppose there exist $p_j$ with $0 < p_j < \infty$ such that
       $$ \sum_{j=1}^d p_j B_j^* B_j = I_n.$$
       Let $1 \leq q_j < \infty$, and define $q = \sum_{j=1}^d p_j q_j$.
       Then for all $G \in L^{q'}(\mathbb{R}^n)$ there exist
       $g_1, \dots , g_d$ such that
       $$ G(x) \leq g_1(x)^{p_1q_1/q} \dots g_d(x)^{p_dq_d/q} \; \; {\mbox{ a.e.}}$$
       and, for all $j$,
       $$ \left\| \int_{V_j^\perp} g_j \right\|_{L^{{q_j}^{\prime}}(V_j)} \leq \left\|G\right\|_{q'}.$$
      \end{thm}

    One simply needs to note (see the discussion in Example~\ref{BL}) that under the hypothesis of this theorem,
$$ \left\| \prod_{j=1}^d (f_j\circ B_j)^{p_jq_j/q} \right\|_{L^q(\mathbb{R}^n)} \leq \prod_{j=1}^d \left\|f_j\right\|_{L^{q_j}(V_j)}^{p_jq_j/q}$$
    and $\sum_{j=1}^d p_j \geq 1$, and thus
    $q \geq 1$. Therefore Theorem~\ref{thmmainbaby} applies.
    
    \medskip
    \noindent
    Except in some rather trivial cases\footnote{For example when the $V_j$ are mutually orthogonal and $p_j = 1$ for all $j$.}
    we do not know any such explicit factorisations with the sharp
    constant $1$.
 For example,
    let $v_1, v_2$ and $v_3$ be unit vectors in $\mathbb{R}^2$ with angle $2\pi/3$ between each pair. Then, with
    $B_j$ being orthogonal projection onto the span of $v_j$, we have
    $$ \frac{2}{3} (B_1^*B_1 + B_2^*B_2 + B_3^*B_3) = I_2.$$
    Take $q_j = 1$ for each $j$ so that $q = 2$. Consequently, for all $G \in L^2(\mathbb{R}^2)$,
    there exist $g_1, g_2, g_3$ such that
     $$ G(x) \leq g_1(x)^{1/3} g_2(x)^{1/3} g_3(x)^{1/3} \; \; \mbox{ a.e.}$$
       and, for each $j$,
       $$ {\rm ess \, sup}_s \int_{\mathbb{R}} g_j(sv_j + tv_j^\perp) {\rm d}t  \leq \|G\|_2.$$
Even in such simple cases as this the factorisation is not yet understood explicitly.

\subsection{General Brascamp--Lieb inequalities}\label{GGBL}
    On the other hand, under the conditions
    \begin{equation}\label{scalingxx}
      \sum_{j=1}^d p_j \, {\rm dim \;  im} B_j = n
    \end{equation}
and 
    \begin{equation}\label{dimensionxx}
      {\rm dim} \, V \leq \sum_{j=1}^d  p_j \, {\rm dim} B_j V
      \end{equation}
    for all $V$ in the lattice of subspaces of $\mathbb{R}^n$ generated by $\{ \ker B_j\}_{j=1}^d$,
    we now indicate how to construct semi-explicit factorisations yielding the finiteness of the
    constant $C$ in \eqref{BL1}. We use the term ``semi-explicit'' because the construction is
    algorithmic in nature. Notwithstanding, we give an informal discursive treatment rather than a collection of flow-charts.
    We assume throughout the discussion that the $B_j$ are nonzero mappings, that is, $n_j = {\rm rank}(B_j) \geq 1$
    for each $j$. (If some $B_j =0$ it plays no role in inequality \eqref{BL1}, nor in \eqref{scalingxx} or \eqref{dimensionxx},
    and it can simply be dropped.) When $n=1$ matters quickly reduce to consideration of H\"older's inequality,
    which is treated in Section~\ref{Holderrevisited} above, so we shall focus on what happens when $n \geq 2$.

    \medskip
    \noindent
    We now sketch how this is done, and we begin with a couple of definitions from \cite{BCCT1} and \cite{BCCT2}.
    Given a collection of linear surjections $\{B_j\}$, its {\em Brascamp--Lieb polytope} is defined by
    $$\mathcal{P}(\{B_j\}) = \{(p_1, \dots, p_d) \in [0, \infty)^d \, : \, {\rm dim} \, V \leq \sum_{j=1}^d p_j \, {\rm dim} B_j V \mbox{ for all subspaces }V \}.$$
    This is manifestly a closed convex set, and, as has been previously noted, is contained in $[0,1]^d$, and is
    therefore the convex hull of its extreme points. Given data $\{B_j\}$ and $\{p_j\}$, a {\em critical subspace}
 is a nontrivial proper subspace $V$ of $\mathbb{R}^n$ for which
 \begin{equation}\label{crit}
{\rm dim} \, V = \sum_{j=1}^d p_j \, {\rm dim} B_j V.
\end{equation}

The construction of the factorisations hinges on the question of existence or non-existence of
critical subspaces for the problem with data $\{B_j, p_j\}$. Indeed, if there is a critical
subspace $V$ for $\{B_j, p_j\}$, then the problem of factorising a function on $\mathbb{R}^n$ decomposes
into two factorisation subproblems on the spaces $V$ and $V^\perp$, each of which has positive but
{\em strictly smaller} dimension than $n$.\footnote{To facilitate the discussion which follows we should
  strictly speaking replace the roles of $\mathbb{R}^n$ and $\mathbb{R}^{n_j}$
  by those of abstract $n$- and $n_j$-dimensional real Hilbert spaces respectively.} This will allow us in effect to induct
on the parameter $n$. These subproblems inherit the same $\{p_j\}$ and
have ``new'' $B_j$ which are related to the ``old'' $B_j$ in a precise way. The two subproblems
inherit the conditions corresponding to \eqref{scalingxx} and \eqref{dimensionxx}: indeed, \eqref{scalingxx}
for each of the two subproblems holds precisely because the subspace $V$ is critical, and we shall make crucial
use of this fact. We isolate the details of how this works -- in particular how factorisations for the two
subproblems combine to give a factorisation for the original problem -- in Section~\ref{BLfactdetails} below.

\medskip
\noindent
On the other hand, if there is no critical subspace for the problem $\{B_j, p_j\}$,
then $(p_1, \dots , p_d)$ lies in the interior of $\mathcal{P}(\{B_j\})$. To establish a
factorisation for the problem in this case, it therefore suffices to (i) establish
factorisations for the extreme points of $\mathcal{P}(\{B_j\})$ and (ii) to show,
given factorisations at the extreme points, how to establish factorisations at all
interior points of $\mathcal{P}(\{B_j\})$. Point (ii) is tantamount to showing that
factorisations behave well under multilinear interpolation, and this we have already
successfully addressed separately in Section~\ref{qzf}.

\medskip
\noindent
To deal with point (i), we consider the Brascamp--Lieb problems at the extreme points 
$(\tilde{p}_1, \dots, \tilde{p}_d)$ of $\mathcal{P}(\{B_j\})$,\footnote{An algorithm for locating
  these extreme points can be found in \cite{Vald}.} and, at each of them, ask the same
question -- does there exist a critical subspace? Since $(\tilde{p}_1, \dots, \tilde{p}_d)$
is an extreme point of $\mathcal{P}(\{B_j\})$, there will certainly be subspaces $V$ of $\mathbb{R}^n$
satisfying
\begin{equation}\label{crit'}
{\rm dim} \, V = \sum_{j=1}^d \tilde{p}_j \, {\rm dim} B_j V.
\end{equation}
If there is a {\em nontrivial and proper} such subspace, we have a critical subspace for the
problem $\{B_j, \tilde{p}_j\}$, and we can proceed as above, in effect going around the loop.
The only remaining possibility is that the only subspaces $V$ of $\mathbb{R}^n$ satisfying
\eqref{crit'} are $\{0\}$ and $\mathbb{R}^n$ itself.

\medskip
\noindent
We are thus left to deal with the special case of our original problem in which $(p_1, \dots, p_d)$ is
an extreme point
of $\mathcal{P}(\{B_j\})$, but for which the only subspaces of $V$ of $\mathbb{R}^n$ satisfying $\eqref{crit}$
are $\{0\}$ and $\mathbb{R}^n$ itself. Matters quickly reduce to rather trivial considerations.
Indeed, in this situation, $\mathcal{P}(\{B_j\})$ consists
precisely of those $(p_1, \dots, p_d) \in [0,\infty)^d$ lying on the hyperplane $\sum_{j=1}^d p_j n_j = n$, and its
extreme points are precisely those of the form $(0, \dots, 0, n/n_j , 0 , \dots, 0)$.
Since $n_j \leq n$ always, and since $\mathcal{P}(\{B_j\}) \subseteq  [0,1]^d$, the only circumstances in which
this case arises is when $n_j = n$ for all $j$. In this case, our Brascamp--Lieb problem at an extreme point of
$\mathcal{P}(\{B_j\})$ is necessarily of the form (modulo permutations of the coordinate axes) 
$$ \int_{\mathbb{R}^n} f_1(B_1 x)^1f_2(B_2 x)^0  \dots  f_d(B_d x)^0 {\rm d}x
\leq C \left(\int_{\mathbb{R}^n} f_1\right)^1 \left(\int_{\mathbb{R}^n} f_2\right)^0 \dots  \left(\int_{\mathbb{R}^n} f_d\right)^0$$
or, equivalently,
$$  \int_{\mathbb{R}^n} f_1(B_1 x){\rm d}x
\leq C \left(\int_{\mathbb{R}^n} f_1\right)$$
where $B_1$ is invertible. This of course holds with equality with $C = (\det B_1)^{-1}$, and a trivial factorisation applies.

\medskip
\noindent
Running the machine described above in reverse will thus eventually furnish a factorisation in the
general case, and, indeed, the only possible loss in terms of sharp constants occurs at steps where
interpolation is employed.

\subsubsection{Factorisation in the presence of a critical subspace}\label{BLfactdetails}
We give the details needed to close the argument set out above in the presence of a critical subspace.
The only place we use criticality is that it implies that \eqref{scalingxx} and \eqref{dimensionxx}
hold for the
two subproblems which arise --  see \cite{BCCT1}. Since these are the necessary and sufficient conditions
for finiteness of the constant, we may assume that factorisations for the two subproblems
exist. (Formally we proceed by induction on $n$, and the case $n=1$ is trivial.)

\medskip
\noindent
Let $B_j : \mathbb{R}^n \to \mathbb{R}^{n_j}$ be linear surjections. Suppose that $U$ is a nontrivial proper
subspace of $\mathbb{R}^n$. (As indicated above, we do not assume that it is a critical
subspace.)
Define $ \tilde{B}_j :\, U \to B_jU$ and $\tilde{\tilde{B}}_j :\,  U^\perp \to (B_jU)^\perp$ by
$$ \tilde{B}_j(x) = B_j x$$
and
$$ \tilde{\tilde{B}}_j(y) = \Pi_{(B_jU)^\perp} B_j y.$$
If some $\tilde{B}_j$ or $ \tilde{\tilde{B}}_j$ is zero we can simply discard it. (It cannot be the case that every
$\tilde{B}_j$ is zero, for if this happened, we would have $U \subseteq \cap_{j=1}^d \ker B_j$, and, as we have noted previously, a necessary condition for finiteness of the Brascamp--Lieb constant is that $\cap_{j=1}^d \ker B_j = \{0\}$. For similar reasons
it cannot be the case that every $\tilde{\tilde{B}}_j$ is zero.)

\medskip
\noindent
Also define $\Gamma_j : \, U^\perp \to B_jU$ by
$$ \Gamma_j(y) = \Pi_{(B_jU)} B_j y.$$
Here, $\Pi_W$ denotes orthogonal projection onto a subspace $W$. So for $x \in U$ and $y \in U^\perp$,
$$B_j(x + y) = \tilde{B}_jx + \tilde{\tilde{B}}_j y + \Gamma_j y = \left(\tilde{B}_jx + \Gamma_j y\right)
+  \tilde{\tilde{B}}_j y \in  B_jU \oplus (B_jU)^\perp.$$

\medskip
\noindent
The two Brascamp--Lieb subproblems arising can be written in the form  
$$ \int_U \prod_{j=1}^d f_j(\tilde{B}_j x)^{p_j/p} {\rm d} x \leq C \prod_{j=1}^d \left(\int f_j\right)^{p_j/p}$$
and
$$ \int_{U^\perp} \prod_{j=1}^d f_j(\tilde{\tilde{B}}_j x)^{p_j/p} {\rm d} x \leq C \prod_{j=1}^d \left(\int f_j\right)^{p_j/p}$$
where $p = \sum_j p_j \geq 1$. With $\alpha_j = p_j/p$, we are entitled to suppose that the following
two corresponding factorisation statments hold:

\medskip
\noindent
For all $H \in L^{p'}(U)$ of norm $1$ there exist $H_1, \dots, H_d$ such that
$$ H(x) \leq \prod_{j=1}^d H_j(x)^{\alpha_j}$$
and, for all $\phi \in L^1(B_jU)$ of norm at most $1$,
$$ \int_U \phi(\tilde{B}_j x) H_j(x) {\rm d}x \leq K_1;$$

\medskip
\noindent
For all $M \in L^{p'}(U^\perp)$ of norm $1$ there exist $M_1, \dots, M_d$ such that
$$ M(y) \leq \prod_{j=1}^d M_j(y)^{\alpha_j}$$
and, for all $\psi \in L^1((B_jU)^\perp)$ of norm at most $1$,
$$ \int_{U^\perp} \psi(\tilde{\tilde{B}}_j y) M_j(y) {\rm d}y \leq K_2.$$

\medskip
\noindent
Given $G  \in L^{p'}(\mathbb{R}^n)$ of norm $1$, we want to subfactorise it as
$$  G(z) \leq \prod_{j=1}^d G_j(z)^{\alpha_j}$$
such that for all $f \in L^1(\mathbb{R}^{n_j})$ of norm at most $1$,
$$ \int_{\mathbb{R}^n} f({B}_j z) G_j(z) {\rm d}z \leq K_1 K_2.$$

\medskip
\noindent
This is a factorisation statement corresponding to the problem
$$ \int_{\mathbb{R}^n} \prod_{j=1}^d f_j({B}_j x)^{p_j/p} {\rm d} x \leq C \prod_{j=1}^d \left(\int f_j\right)^{p_j/p}.$$
If we can do this, then the procedure described above for factorising Brascamp--Lieb problems closes. 

\medskip
\noindent
We begin by writing $G \in L^{p'}$ of norm $1$ as 
\begin{equation}\label{qfj}
  G(x,y) = H_y(x) M(y)
  \end{equation}
where $\|H_y\|_{p'} = 1$ for all $y$ and $\|M\|_{p'} = 1$. We will then factorise
$M$ and each $H_y$ as above, and combine the factorisations to obtain a suitable factorisation for $G$. 

\medskip
\noindent
Indeed, defining $H_y$ and $M$ by 
$$ G(x,y) = \frac{G(x,y)}{\left(\int G(x,y)^{p'} {\rm d} x\right)^{1/p'}} \left(\int G(x,y)^{p'} {\rm d} x\right)^{1/p'} :=  H_y(x) M(y)$$
is essentially the unique way to achieve \eqref{qfj} with the desired conditions.\footnote{Indeed, suppose $G$ is in the mixed-norm space
  $L^r_{{\rm d}y}(L^s_{{\rm d}x})$ and we want to write $G(x,y) = H(x,y)M(y)$ where $\|M\|_r = \|G\|_{L^r(L^s)}$
  and where $\|H(\cdot, y)\|_s = 1$ for all $y$. Integrating $G(x,y)^s = H(x,y)^sM(y)^s$ with respect to $x$ shows that
  the only way to do this is to take $M(y) = \|G(\cdot, y)\|_s$
  and $H(x,y) = G(x,y)/ \|G(\cdot, y)\|_s$. See the remarks at the end of Section~\ref{Loomis--Whitneyrevisited}.}

\medskip
\noindent
Therefore,
$$ G(x,y) \leq  \prod_{j=1}^d [H_{jy}(x) M_j(y)]^{\alpha_j} :=  \prod_{j=1}^d G_j(x,y)^{\alpha_j}$$
where for all $y \in U^\perp$, for all $\phi \in L^1(B_jU)$ of norm at most $1$,
$$ \int_U \phi(\tilde{B}_j x) H_{jy}(x) {\rm d}x \leq K_1$$
and where for all $\psi \in L^1((B_jU)^\perp)$ of norm at most $1$,
$$ \int_{U^\perp} \psi(\tilde{\tilde{B}}_j y) M_j(y) {\rm d}y \leq K_2.$$

\medskip
\noindent
We want to show that for all $f \in L^1(\mathbb{R}^{n_j})$ of norm at most $1$,
$$ \int_{\mathbb{R}^n} f({B}_j z) G_j(z) {\rm d}z = \int_{U^\perp} \int_U f({B}_j(x,y))H_{jy}(x) M_j(y) {\rm d}x {\rm d}y
\leq K_1 K_2.$$

\medskip
\noindent
Fix $y \in U^\perp$ and write the inner integral over $U$ as
$$\int_U f({B}_j(x,y))H_{jy}(x) {\rm d}x
= \int_U f({B}_jx + B_jy) H_{jy}(x) {\rm d}x.$$
Now $f({B}_jx + B_jy) = f((\tilde{B}_j x +\Gamma_j y)+
\tilde{\tilde{B}}_jy)$. For  $w \in B_jU$ and $\xi \in (B_jU)^\perp$
let $\phi_\xi(w) := f( w + \xi)$.
Therefore $f({B}_jx + B_jy) = \phi_{\tilde{\tilde{B}}_j y}(\tilde{B}_j x + \Gamma_j y)
=( \tau_{(\Gamma_j y)} \phi_{\tilde{\tilde{B}}_j y})(\tilde{B}_j x)$, where $(\tau_\eta \chi)( \cdot)
= \chi( \cdot + \eta)$ denotes translation by $\eta$. So,
$$\int_U f({B}_j(x,y))H_{jy}(x) {\rm d}x
= \int_U ( \tau_{(\Gamma_j y)} \phi_{\tilde{\tilde{B}}_j y})(\tilde{B}_j x) H_{jy}(x) {\rm d}x
\leq K_1 \| \tau_{(\Gamma_j y)} \phi_{\tilde{\tilde{B}}_j y}\|_1$$
by what we are assuming.

\medskip
\noindent
Now, by translation invariance,  $\| \tau_{(\Gamma_j y)} \phi_{\tilde{\tilde{B}}_j y}\|_1 = \|\phi_{\tilde{\tilde{B}}_j y}\|_1$.
Therefore, letting $\psi(\xi) :=   \|\phi_{\xi}\|_1$ for $\xi \in (B_jU)^\perp$,
$$ \int_{\mathbb{R}^n} f({B}_j z) G_j(z) {\rm d}z \leq K_1 \int_{U^\perp}  \psi(\tilde{\tilde{B}}_j y)  M_j(y){\rm d}y
\leq K_1 K_2 \|\psi\|_1.$$

\medskip
\noindent
Finally,
$$ \|\psi\|_1 = \int_{(B_jU)^\perp} \left(\int_{B_jU} \phi_\xi(w) {\rm d} w \right) {\rm d} \xi = \int_{\mathbb{R}^{n_j}} f(z) {\rm d} z = 1,$$
and this gives what we wanted.

\section{Multilinear Kakeya inequalities revisited}\label{MKrevisited}
Recall that we have families $\mathcal{P}_j$ of $1$-tubes in $\mathbb{R}^n$, and for $P \in \mathcal{P}_j$, 
its direction $e(P) \in \mathbb{S}^{n-1}$ satisfies
$|e(P) - e_j| \leq c_n$ where $c_n$ is a small dimensional constant. The multilinear Kakeya theorem of
Guth \cite{MR2746348} (see also  \cite{CV}) states that
$$ \Big\|\prod_{j=1}^n \left(\sum_{P _j \in \mathcal{P}_j} a_{P_j} \chi_{P_j}(x) \right)^{1/n} \Big\|_{L^{n/(n-1)}(\mathbb{R}^n)} 
\leq C_n \prod_{j=1}^n \left(\sum_{P_j \in \mathcal{P}_j} a_{P_j}\right)^{1/n}.$$
This inequality is of the form \eqref{submainineq} with $X = \mathbb{R}^n$, $q = n/(n-1)$, $Y_j = \mathcal{P}_j$ with
counting measure, $p_j = 1$ for all $j$, $\alpha_j = 1/n$ for all $j$, and $T ((a_{P_j}))(x) = \sum_{P _j \in \mathcal{P}_j} a_{P_j} \chi_{P_j}(x)$.

\medskip
\noindent
Guth proved this result essentially by establishing a suitable subfactorisation for each nonnegative
$M \in L^n(\mathbb{R}^n)$, and then applying Proposition~\ref{thmguthbaby}. His subfactorisation is
described in terms of an auxiliary polynomial $p$ of `low' degree
dominated by $\|M\|_n$, whose zero set $Z_p$ has `large' {\bf visibility} on each unit cube $Q$ of $\mathbb{R}^n$ in the
sense that ${\rm vis}(Z_p \cap Q) \gtrsim \int_Q M$. We do not enter into the details of the definition of visibility,
nor into how this gives the desired subfactorisation, but instead refer the reader to \cite{MR2746348} and \cite{CV}.
(In the latter paper the approach using Proposition~\ref{thmguthbaby} is explicit while in the former it is implicit.
And one should note that the definition of visibility used in \cite{CV} is a power of the original one used in
\cite{MR2746348}.)
It was the shock of seeing such an unlikely functional-analytic method succeed which inspired us to study the general 
question of necessity of factorisation as taken up in this paper. In hindsight, our
linkage of subfactorisation of functions with Maurey's theory of factorisation of operators helps place Guth's method
in perspective.

\medskip
\noindent
Bourgain and Guth in \cite{BG} established an affine invariant form of the multilinear Kakeya inequality, removing
the hypothesis that $|e(P) - e_j| \leq c_n$ for $P \in \mathcal{P}_j$, at the price of inserting a damping factor
on the left-hand side which is consistent with the affine-invariant Loomis--Whitney inequality of Section~
\ref{Loomis--Whitneyrevisited}. That is, they proved that for $\mathcal{P}_j$ arbitrary families of $1$-tubes,
$$ \int_{\mathbb{R}^n}\left(\sum_{P_1 \in \mathcal{P}_1} a_{P_1} \chi_{P_1}(x) \dots \sum_{P_n \in \mathcal{P}_n} a_{P_n} \chi_{P_n}(x)
  e(P_1) \wedge \dots \wedge e(P_n) \right)^{1/(n-1)} {\rm d}x $$
$$ \leq C_n \prod_{j=1}^n \left(\sum_{P_j \in \mathcal{P}_j} a_{P_j}\right)^{1/(n-1)}.$$

  \medskip
  \noindent
  As the reader will readily verify (using the same argument as in the proof of Proposition~\ref{thmguthbaby},
  see also Section~\ref{culture} above),
  in order to establish this, it suffices to show that for every nonnegative
  $M \in L^n(\mathbb{R}^n)$ which is constant on unit cubes in a standard lattice $\mathcal{Q}$,
  there exist nonnegative functions $S_j : \mathcal{Q} \times \mathcal{P}_j \to \mathbb{R}$ such that
  $$ M(Q) \lesssim \frac{S_1(Q, P_1)^{1/n} \cdots  S_n(Q, P_n)^{1/n}}{ e(P_1) \wedge \dots \wedge e(P_n)^{1/n}}$$
  whenever the $1$-tubes $P_j$ meet at $Q$, and, for all $j$, for all $P_j \in \mathcal{P}_j$,
  $$ \sum_{Q \in \mathcal{Q}, \, Q \cap P_j \neq \emptyset} S_j(Q, P_j) \lesssim \left(\sum_{Q \in \mathcal{Q}} M(Q)^n\right)^{1/n}.$$
  And indeed this is what Bourgain and Guth essentially did (see also \cite{CV}). It is therefore very tempting to
  ask whether, in analogy with the situation of Theorem~\ref{thmmainbaby}, this method is {\em guaranteed} to work
  in so far as the statement of the affine-invariant multilinear Kakeya inequality automatically implies the existence of a
  subfactorisation as in the last two displayed inequalities.
  Unfortunately, as we have established above in Section~\ref{fdmgs}, there is no such general functional-analytic
  principle which guarantees this.
  
\medskip
\noindent
The recent multilinear Kakeya $k_j$-plane inequalities, and indeed the even more general perturbed Brascamp--Lieb
inequalities, both recently established by Zhang \cite{Z}, also fit into the framework we consider, the latter
as a generalisation of inequality \eqref{BL3}.  

\subsection{The finite field multilinear Kakeya inequality}\label{ffmk}

Zhang \cite{Zh2} has recently solved the discrete analogue of the multilinear Kakeya problem. Let $\mathbb{F}$
  be a field and let $\mathcal{L}_j$ be arbitrary families of lines in $\mathbb{F}^n$.
  For $l_j \in \mathcal{L}_j$ declare $e(l_1) \wedge \dots \wedge e(l_n)$ to be $1$ if the vectors $\{e(l_j)\}$
  are linearly independent and to be $0$ otherwise. Zhang has proved that for a certain $C_n$ depending only on $n$,
\begin{equation}
\label{ffkakeya}
\begin{aligned}
 \sum_{ x \in \mathbb{F}^n}&\left(\sum_{l_1 \in \mathcal{L}_1} a_{l_1} \chi_{l_1}(x) \dots \sum_{l_n \in \mathcal{L}_n} a_{l_n} \chi_{l_n}(x)
  e(l_1) \wedge \dots \wedge e(l_n) \right)^{1/(n-1)} \\
& \leq C_n \prod_{j=1}^n \left(\sum_{l_j \in \mathcal{L}_j} a_{l_j}\right)^{1/(n-1)}.
\end{aligned}
\end{equation}

\medskip
\noindent
  When $n=2$ the constant $C_2 = 1$, as is readily verified using $1/(n-1) = 1$ and changing the order of summation on
  the left-hand side. Moreover, for general $n$, if all the lines in $\mathcal{L}_j$ are parallel to some fixed vector $y_j$ with
  $\{y_j\}_{j=1}^n$ linearly independent, the constant is likewise $1$, since matters can then be reduced to the classical
  Loomis--Whitney inequality via an invertible linear transformation of $\mathbb{F}^n$, (or one can write down a suitable
  factorisation as in Example~\ref{Loomis--Whitneyrevisited}).

  \medskip
  \noindent
  The presence of the factor $e(l_1) \wedge \dots \wedge e(l_n)$ in \eqref{ffkakeya} precludes any assertion that
  \eqref{ffkakeya} is equivalent to a factorisation statement: see Section~\ref{culture} above.
If however the $\mathcal{L}_j$ are presumed to satisfy the
property that if $(l_1, \dots , l_n) \in \mathcal{L}_1 \times \dots \times \mathcal{L}_n$,
then the directions $\{e(l_1), \dots , e(l_n)\}$ are linearly independent,
  we have that the term $e(l_1) \wedge \dots \wedge e(l_n)$ is identically $1$, and the result then falls under the scope of
  Theorem~\ref{thmmain}.

  \medskip
  \noindent
  In particular, when $n = 2$ and we have two finite families of lines $\mathcal{L}_1$
  and $\mathcal{L}_2$ in $\mathbb{F}^2$ such that no line in $\mathcal{L}_1$ is parallel to any line in $\mathcal{L}_2$,
  this holds. Hence we obtain:

  \begin{prop}
Let $\mathcal{L}_1$
and $\mathcal{L}_2$ be finite families of lines in $\mathbb{F}^2$ such that no line in $\mathcal{L}_1$ is parallel to any line in $\mathcal{L}_2$. Let $J \subseteq \mathbb{F}^2$ be the set of points where some $l_1 \in \mathcal{L}_1$ meets an $l_2 \in \mathcal{L}_2$.
  Suppose $\sum_{x \in J} G(x)^2 = 1$. Then there exist
  $g_1, g_2 : J \to \mathbb{R}_+$ such that for all $x \in J$
  $$ G(x) = \sqrt{g_1(x) g_2(x)}$$
  and, moreover, for all $l_j \in \mathcal{L}_j$, $j=1,2$,
  $$ \sum_{x \in J \cap l_j} g_j(x) \leq 1.$$
\end{prop}
  In spite of the extreme simplicity of the original problem, no procedure for coming to an explicit such factorisation is currently known.

  \medskip
  \noindent
 Jon Bennett had asked whether, even in higher dimensions, the constant $C_n$ 
 in the finite field multilinear Kakeya inequality might still be $1$. This is true
 in the case of $\mathbb{F}_2^3$. However, this turns out to have been
 over-optimistic, and we have:

 \begin{prop}
   Suppose the discrete multilinear Kakeya inequality \eqref{ffkakeya} holds in the
   case $n=3$ for $\mathbb{F} = \mathbb{F}_3$. Then $C_3 > 1.04$.
 \end{prop}

 \medskip
 \noindent
We remark that Tidor, Yu and Zhao \cite{TYZ} have very recently established numerical values for the constants in Zhang's theorem, and in particular they show that $C_3 \leq \sqrt{6}$.
 
 \begin{proof}
We construct an example. In this example, for each $j=1,2,3$, we nominate two directions, and the family $\mathcal{L}_j$ will consist of all lines
with one of these directions. The two directions for each $j$ will be chosen so that each of the eight choices of one
direction from each of the three families results in a linearly independent set of directions, so that the terms
$e(l_1) \wedge e(l_2) \wedge e(l_3)$ are all $1$.
Each family of coefficients $a$ -- as a function defined on $\mathcal{L}_j$ and more properly denoted by $a_j$
-- is defined to be supported on three lines from $\mathcal{L}_j$ in such a way
that the $x$-summand on the left-hand side of \eqref{ffkakeya} is non-zero at five points. Each $a$ will take
nonzero values in $\{1,2\}$ and thus each line under consideration will have a {\em weight} equal to $1$ or $2$.
For each $j$ we shall have that two of the three lines pass through two of these five points and the remaining
line passes through the remaining point. 

\medskip
More concretely, let
\begin{itemize}
\item $\mathcal{L}_1$ be the lines with direction $(1,1,0)$ or $(2,1,1)$;
\item $\mathcal{L}_2$ be the lines with direction $(0,1,0)$ or $(0,1,1)$; and
\item $\mathcal{L}_3$ be the lines with direction $(1,0,1)$ or $(0,0,1)$.
\end{itemize}
It is straightforward to verify that the directions of any three lines, one from each collection, span $\mathbb{F}_3^3$.

\medskip
\noindent
We now proceed to properly define the coefficients $a$. We denote by $a_j$ the function whose domain is $\mathcal{L}_j$,
and which is defined as follows:
\begin{itemize}
	\item Let $a_1$ be
		\begin{itemize}
			\item $2$ on the line with direction $(1,1,0)$ passing through $(0,2,2)$ and $(2,1,2)$,
			\item $2$ on the line with direction $(2,1,1)$ passing through $(0,2,1)$ and $(2,0,2)$, 
			\item $1$ on the line with direction $(1,1,0)$ through $(0,0,0)$, and
			\item $0$ on other lines of $\mathcal{L}_1$.
		\end{itemize}
	\item Let $a_2$ be
		\begin{itemize}
			\item $2$ on the line with direction $(0,1,0)$ passing through $(2,0,2)$ and $(2,1,2)$,
			\item $2$ on the line with direction $(0,1,1)$ passing through $(0,0,0)$ and $(0,2,2)$,
			\item $1$ on the line with direction $(0,1,0)$ through $(0,2,1)$, and
			\item $0$ on other lines  of $\mathcal{L}_2$.
		\end{itemize}
	\item Let $a_3$ be 
		\begin{itemize}
			\item $2$ on the line with direction $(0,0,1)$ passing through $(0,2,1)$ and $(0,2,2)$,
			\item $2$ on the line with direction $(1,0,1)$ passing through $(0,0,0)$ and $(2,0,2)$,
			\item $1$ on the line with direction $(0,0,1)$ through $(2,1,2)$, and
			\item $0$ for other lines of of $\mathcal{L}_3$.
		\end{itemize}
\end{itemize}
\noindent
Each $\mathcal{L}_j$ has two lines of $a$-value or weight $2$ and one of weight $1$.

\medskip
We can see that the only points where lines from all three families intersect are the five points mentioned, namely
$(0,0,0)$, $(0,2,1)$, $(0,2,2)$, $(2,0,2)$ and $(2,1,2)$. At the three points $(0,0,0)$, $(0,2,1)$ and $(2,1,2)$
we have two lines of weight $2$ and one of weight $1$ meeting; at the two points $(0,2,2)$ and $(2,0,2)$ we have
three lines of weight $2$ meeting. So the value of the $x$-summand on the left-hand side of \eqref{ffkakeya}
is $2$ at the three points $(0,0,0)$, $(0,2,1)$ and $(2,1,2)$, and is $2^{3/2}$ at the two points $(0,2,2)$ and $(2,0,2)$.
The left-hand side adds up to $3\cdot2+2\cdot2^{3/2} > 11.65$.
The value of the right-hand side of \eqref{ffkakeya} is $C_3 \cdot 5^{3/2}$ $\leq C_3 \cdot 11.19$. This shows that $C_3 \geq
\frac{6 + 2^{3/2}}{5^{3/2}} > 11.65/11.19 > 1.04 > 1$. 

\end{proof}

A counterexample to the conjecture that \eqref{ffkakeya} holds with $C_n=1$ was first found by use of the duality theory
developed above, which, as we have mentioned, is valid under the assumption that any $n$-tuple of lines taken from
$\mathcal{L}_1 \times \dots \times \mathcal{L}_n$ has linearly independent directions.
To explain why this route was taken, let us assume that we are considering a finite field of
size $q$. If we let $\mathcal{L}_j$ consist of all lines with directions in some given set of size $r$ then the input to
\eqref{ffkakeya}, namely the tuple $(a_1,\dots,a_n)$ belongs to a real vector space of dimension $nq^{n-1}r$.
The input to problem \eqref{mainprob} is the function $G$ which belongs to a real vector space of dimension $q^{n}$.
In our case we have $n=q=3$ and $r=2$ so the input to the problem \eqref{mainprob} belongs to a smaller vector space
than the input to \eqref{ffkakeya}.
The additional cost of solving the convex optimisation problem compared with the cost of simply evaluating
each side of \eqref{ffkakeya} does not significantly alter the balance of cost.

\medskip
The solution to the convex optimisation problem was found using the software package CVXOPT \cite{CVXOPT},
which yields the solution for both the primal and dual problems. The solution to the dual problem
was then slightly simplified by hand for neater exposition and this is what is presented here.
\bibliographystyle{plain}

\bibliography{}
\end{document}